\documentclass[11pt]{article}
\usepackage{amsmath,amsthm,amsfonts,amssymb,mathrsfs,bm, geometry}
\usepackage[usenames]{color}
\usepackage{hyperref}
\usepackage{fancyhdr}
\parskip 	\bigskipamount
\newtheorem{theorem}{Theorem}[section]
\newtheorem{lemma}[theorem]{Lemma}
\newtheorem{corollary}[theorem]{Corollary}
\newtheorem{proposition}[theorem]{Proposition}

\newtheorem{remark}{Remark}[section]

\newtheorem{definition}{Definition}

\newtheorem*{notation}{Notation}
\def\D{\mathcal{D}}
\def\Z1{\m}
\def\Z{\mathbb{Z}}

\def\R{\mathbb{R}}
\def\C{\mathbb{C}}
\def\P{\mathbb{P}}
\def\eps{\varepsilon}
\def\O{\Omega}
\def\A{\mathcal{A}}
\def\a{\mathcal{I}}
\def\b{V_{\eps}}

\def \del{\delta}
\def \z{\zeta}
\def \o{\omega}
\def \Z{\boldsymbol{Z}}
\def \s{\sigma}
\def \c{\sqrt{{n \choose k}k!}}
\def\cnew{\sqrt{{n \choose r}r!}}

\def \X{\boldsymbol{X}}
\def \Y{\boldsymbol{Y}}
\def \L{\boldsymbol{L}}
\def \M{\boldsymbol{M}}
\def \N{\boldsymbol{N}}
\def\E{\mathbb{E}}
\def\e{\boldsymbol{E}}
\def\xo{\mathcal{X}}
\def\Eu{\mathcal{E}}
\def\tP{\mathbb{P}}
\def\S{\mathcal{S}}
\def\B{\mathcal{B}}

\renewcommand{\l}[0]{\left }
\renewcommand{\r}[0]{\right}

\def\on{\uo(n_k)}
\def\zn{\uz(n_k)}
\def\xin{X_{\mathrm{in}}}
\def\sn{s(n_k)}
\def\fn{f_{n_k}}
\def\onk{\Omega_{n_k}(j)}
\def\onp{\Omega_{n_k}^1(j)}
\def\onq{\Omega_{n_k}^2(j)}

\def\G{\mathcal{G}}
\def\F{\mathcal{Z}}
\def\rhoo{\varrho}
\def\nat{\mathbb{N} \cup \{0\}}
\def\uz{\underline{\z}}
\def\uo{\underline{\o}}
\def\uout{\Upsilon_{\text{out}}}
\def\uoutd{\Upsilon_{\text{out}}^{\D_0}}
\def\el{\mathcal{L}}
\def\els{\el_{\Sigma}}
\def\op{\mathfrak{o}}
\def\out{\mathrm{out}}
\def\inn{\mathrm{in}}
\def\AA{\mathfrak{A}^m}

\def\wt{\widetilde}
\def\xout{X_{\mathrm{out}}}

\def\A{\mathfrak{A}}
\def\tv{\wt{V}}
\def\la{\lambda}
\def\a{\alpha}
\def\ol{\overline}
\def\ck{\sqrt{{n_k \choose t}t!}}

\makeatletter
\renewcommand*{\@cite@ofmt}{\hbox}
\makeatother

\begin{document}
\title{Rigidity and Tolerance in point processes: Gaussian zeroes and Ginibre eigenvalues}
\author{
\begin{tabular}{c}
 \href{http://math.berkeley.edu/~subhro/}{Subhroshekhar Ghosh}\\ Department of Mathematics\\ University of California, Berkeley\\
\end{tabular}
\and
\begin{tabular}{c}
\href{http://research.microsoft.com/en-us/um/people/peres/}{Yuval Peres}\\ Microsoft Research\\ Redmond\\  
\end{tabular}
}
\date{}
\maketitle
\begin{abstract}
Let  $\Pi$ be a translation invariant point process on the complex plane $\C$ and let $\D \subset \C$ be a bounded open set. We ask what does the point configuration $\Pi_{\out}$ obtained by taking the points of $\Pi$ outside $\D$ tell us about the point configuration $\Pi_{\inn}$ of $\Pi$ inside $\D$? We show that for the Ginibre ensemble, $\Pi_{\out}$ determines the number of points in $\Pi_{\inn}$. For the translation-invariant zero process of a planar Gaussian Analytic Function, we show that $\Pi_{\out}$ determines the number as well as the centre of mass of the points in $\Pi_{\inn}$. Further, in both models we prove that the outside says ``nothing more'' about the inside, in the sense that the conditional distribution of the inside points, given the outside, is mutually absolutely continuous with respect to the Lebesgue measure on its supporting submanifold.

\end{abstract}


\newpage
\section{Introduction}
\label{intro}
A point process $\Pi$ on $\C$ is a random locally finite point configuration on the two dimensional Euclidean plane. We will consider simple point processes, 
namely, those in which no two points are at the same location. A simple point process can, equivalently, be looked upon as a random discrete measure $[\Pi]=\sum_{z \in \Pi} \del_{z}$. The probability distribution of the point process $\Pi$ is a probability measure $\P[\Pi]$ on the Polish space of locally finite point sets on the plane. The relevant topology on the space of locally finite point sets is that of weak convergence of measures, when we identify a point configuration with the counting measure it induces on $\C$. This explains the equivalence of the two ways of looking at a point process. Point processes on the plane have been studied extensively. For a lucid exposition on general point processes one can look at \cite{DV}. The group of translations of $\C$ acts in a natural way on the space of locally finite point configurations on $\C$. Namely, the translation by $z$, denoted by $T_z$, takes the point configuration $\Lambda$ to the configuration $T_z(\Lambda):=\{x+z \big| x \in \Lambda\}$. A point process $\Pi$ is said to be translation invariant if $\Pi$ and $T_z(\Pi)$ have the same distribution for all $z \in \C$. The (first) intensity measure $\rho_1$ of a point process on $\C$ is the measure on $\C$ given by $\rho_1(B)=\E[|\Pi \cap B|]$ for every Borel set $B \subset \C$. Higher order intensity measures of a point process can also be defined in a similar manner, see e.g. Definition 1.2.2 in \cite{HKPV}  .  In this work we will focus primarily on translation invariant point processes on $\C$ whose intensities are absolutely continuous with respect to the Lebesgue measure. 

Let $\D$ be a bounded open set in $\C$. Let $\S$ denote the Polish space of locally finite point configurations on $\C$. The decomposition $\C=\D \cup \D^c$ induces a factorization $\S=\S_{\inn} \times \S_{\out}$, where $\S_{\inn}$ and $\S_{\out}$ are respectively the spaces of  finite point configurations on $\D$ and locally finite point configurations on $\D^c$.  This immediately leads to the natural decomposition $\Upsilon=(\Upsilon_{\inn},\Upsilon_{\out})$ for any $\Upsilon \in \S$, and  consequently a decomposition of the point process $\Pi$ as $\Pi=(\Pi_{\inn},\Pi_{\out})$.

In this paper, we ask the following question:  if we know the configuration $\Pi_{\out}$, what can we  conclude about $\Pi_{\inn}$? We consider with regard to this question the main natural examples of translation invariant point processes on the plane, 
and provide a complete answer  in each case.  

The (homogeneous) Poisson point process is the most natural example of a translation-invariant point process on the plane. One of its characterising properties is that the point configurations in two disjoint measurable sets are independent of each other. Therefore, our question is a triviality for the Poisson process: the points outside $\D$ do not provide  any information about the points inside $\D$.

The two main natural examples of translation invariant point processes in the plane that have non-trivial spatial correlations  are the Ginibre ensemble and the zeroes of the planar Gaussian Analytic Function. We refer the reader to \cite{HKPV} for a detailed study of these ensembles. The Ginibre ensemble (in the finite dimensional setting) was introduced  in the physics literature by Ginibre \cite{Gin} as a model based on non-Hermitian random matrices; the infinite Ginibre ensemble is obtained as a weak limit of these finite dimensional point processes. Like the Poisson process, it is translation-invariant and ergodic under rigid motions of the plane. In fact, it is a determinantal point process with the determinantal kernel ${ \displaystyle K(z,w)=\sum_{j=0}^{\infty} \frac{(z \overline{w})^j}{j!}}=e^{z \ol{w}}$ and the background measure $\frac{1}{\pi}e^{-|z|^2}\mathrm{d}\mathcal{L}(z)$. Here $\mathcal{L}$ denotes the Lebesgue measure on $\C$. Krishnapur (\cite{Kr}, Theorem 3.0.5) has shown that the Ginibre ensemble is the unique determinantal point process on $\C$ that is isometry invariant, has a sesqui-holomorphic kernel (i.e., the determinantal kernel is holomorphic in the first variable and 
anti-holomorphic in the second), and is normalized to have unit intensity.
 
The standard planar Gaussian Analytic Function (abbreviated henceforth as GAF) is the random entire function defined by the (random) Taylor series ${ \displaystyle f(z)=\sum_{k=0}^{\infty}\frac{\xi_k}{\sqrt{k!}}z^k }$ where $\xi_k$-s are i.i.d. standard complex Gaussians. We are interested in the point configuration on $\C$ given by the zeroes of this GAF. The GAF zero process is translation invariant and ergodic, and exhibits local repulsion. It has been studied intensively by several authors including Nazarov, Sodin, Tsirelson, and others (see, e.g., \cite{FH}, \cite{STs1},\cite{STs2},\cite{STs3},\cite{NSV},\cite{NS}). Zelditch, Shiffman, and others have studied their counterparts from a geometric perspective (see, e.g., \cite{SZ1},\cite{SZ2},  \cite{SZ3},  \cite{SZ4},\cite{ZZ},\cite{Ze}). 

Sodin \cite{Sod} has shown that in the class of Gaussian analytic functions (i.e., analytic functions whose finite dimensional marginals are Gaussian distributions), the standard planar GAF is the only one to have a translation invariant zero-set (up to scaling and  multiplication by a deterministic entire function with no zeroes). 

For further details on the Gaussian zero and Ginibre ensembles, we refer the reader to Section \ref{setmod}.

Given a pair of random variables $(X,Y)$ which has a joint distribution on a product of Polish spaces $\S_1 \times \S_2$, we can define the \textsl{regular conditional distribution} $\gamma$ of $Y$ given $X$ by the family of probability measures $\gamma(s_1,\cdot)$ parametrized by the elements $s_1 \in \S_1$ such that for any Borel sets $A \subset \S_1$ and $B \subset \S_2$ we have 
\[ \P \l( X \in A, Y \in B \r) = \int_{A} \gamma(s_1,B) \, \mathrm{d}\,\P[X](s_1) \]
where $\P[X]$ denotes the marginal distribution of $X$.  For details on regular conditional distributions, see, e.g., \cite{Pa} (Chapter 5, Section 8) or \cite{Du} (Section 5.1.3).

Recall that $\S_{\inn}$ and $\S_{\out}$ are Polish spaces. Hence, by the general theory of conditional measures, there exists a \textsl{regular conditional distribution} $\rhoo$ of $\Pi_{\inn}$ given $\Pi_{\out}$. Clearly,  $\rhoo$ can be seen as the distribution of a point process on $\D$ which depends on $\Upsilon_{\out}$.

Let $\uz$ be the vector whose coordinates are the points of $\Pi_{\inn}$ taken in uniform random order (the length of this vector is a random variable). We will denote the conditional distribution of $\uz$ given $\Pi_{\out}$ by $\rho$. Formally, it is a family of probability measures $\rho(\uout,\cdot)$ on $\bigcup_{m=0}^{\infty}\D^m$ parametrized by $\uout \in \S_{\out}$.  

There is a simple relationship between $\rho({\uout},\cdot)$ and $\rhoo({\uout},\cdot)$. Consider the natural map $\phi$ from $\bigcup_{m=0}^{\infty}\D^{m}$ to $\S_{\inn}$ which makes a point configuration $\Upsilon_{\inn}$ of size $m$ from a vector $\uz$ in $\D^m$ by forgetting the order of the coordinates of $\uz$. It is easy to see that $\P[\Pi_{\out}]$-a.s. we have $\phi_*\rho({\uout},\cdot)=\rhoo({\uout},\cdot)  $. The main reason for using $\rho$ is that it allows us to easily compare our conditional measures with other standard measures, which can be thought of as appropriate  reference measures, e.g. in the setting of Theorems \ref{gin-2} and \ref{gaf-2}.


For two measures $\mu_1$ and $\mu_2$ defined on the same measure space $\O$ with $\mu_1 \ll \mu_2$ (meaning $\mu_1$ is absolutely continuous with respect to $\mu_2$), we will denote by $\frac{\mathrm{d} \mu_1}{\mathrm{d} \mu_2}(\o)$ the Radon Nikodym derivative of $\mu_1$ with respect to $\mu_2$ evaluated at $\o \in \O$. By $\mu_1 \equiv \mu_2$, we mean that the measures $\mu_1$  and $\mu_2$ are mutually absolutely continuous, that is $\mu_1 \ll \mu_2$ and $\mu_2 \ll \mu_1$. 

In Theorems \ref{gin-1}-\ref{gaf-2} we denote the Ginibre ensemble by $\G$ and the GAF zero ensemble by $\F$. As before, $\D$ is a bounded open set in $\C$. 

In the case of the Ginibre ensemble, we prove that a.s.  the points outside $\D$ determine the number of points inside $\D$, and ``nothing more''. 

\begin{theorem}
 \label{gin-1}
For the Ginibre ensemble, there is a measurable function $N:\S_{\out} \to \nat$ such that a.s. \[ \mathrm{ \ Number \ of \ points \ in \ } \G_{\inn} = N(\G_{\out})\hspace{3 pt}.\]
\end{theorem}

Since a.s. the length of $\uz$ equals $N(\G_{\out})$, we can assume that each measure $\rho(\uout,\cdot)$ is supported on $\D^{N({\uout})}$.

\begin{theorem}
\label{gin-2}
For the Ginibre ensemble, $\P$-a.s. the measure $\rho({\G_{\out}},\cdot)$ and the Lebesgue measure $\el$ on $\D^{N(\G_{\out})}$ are mutually absolutely continuous.
\end{theorem}

In the case of the GAF zero process, we prove that the points outside $\D$ determine the number as well as the centre of mass (or equivalently, the sum) of the points inside $\D$, and ``nothing more''. 

\begin{theorem}
 \label{gaf-1}
For the GAF zero ensemble, \newline
\noindent
(i)There is a measurable function $N:\S_{\out} \to \nat$ such that a.s.  \[ \mathrm{ \ Number \ of \ points \ in \ } \F_{\inn} = N(\F_{\out}).\]
(ii)There is a measurable function $S:\S_{\out} \to \C$ such that a.s.  \[ \mathrm{ \ Sum \ of \ the \ points \ in \ } \F_{\inn} = S(\F_{\out}).\]
\end{theorem}

For a possible value $\uout$ of $\F_{\out}$, define the set of admissible vectors of inside points (obtained by considering all possible orderings of such inside point configurations) \[ \Sigma_{S({\uout})} := \{ \uz \in \D^{N({\uout})} : \sum_{j=1}^{N({\uout})} \z_j = S({\uout}) \}\] where $\uz=(\z_1,\cdots,\z_{N({\uout})})$.

Since a.s. the length of $\uz$ equals $N({\uout})$, we can assume that each measure $\rho(\uout,\cdot)$ gives us the distribution of a random vector in  $\D^{N({\uout})}$ supported on $\Sigma_{S({\uout})}$.

\begin{theorem}
\label{gaf-2}
For the GAF zero ensemble, $\P$-a.s. the measure $\rho(\F_{\out},\cdot)$ and the Lebesgue measure $\els$ on $\Sigma_{S(\F_{\out})}$ are mutually absolutely continuous.
\end{theorem}

We view the conservation laws as the \textbf{``rigidity''} properties of the respective point processes. The absolute continuity (with respect to the Lebesgue measure on the conserved submanifold) of the conditional distribution of the vector of inside points  can be viewed as \textbf{``tolerance''}. The heuristic is  that due to such mutual absolute continuity, the inside points can form  (almost)  any configuration on this conserved submanifold.

Notice that the question whether $\P[\Pi_{\D}] \ll \P[\Pi]$ for a given $\D$ can be phrased in terms of the conditional distribution $\rhoo$, and the same holds for the main questions addressed in this paper. Therefore, in more general terms, we are interested in the support and the regularity properties of $\rhoo$. 

For a finite point process of fixed size $n$, the phenomenon of outside points determining the number of inside points is a triviality.
However, the behaviour of infinite point processes can be quite different, even when they arise as  distributional limits of finite point processes. For example, let us take the disk of area $n$ centred at the origin and consider the point process $\Pi_n$ given by $n$ uniform points inside it. Let $\D$ be the unit disk centred at the origin. The process $\Pi_n$ clearly has the property that the point configuration outside  $\D$ determines the number of points inside $\D$. However, the distributional limit of $\Pi_n$-s, as $n \to \infty$, is the  Poisson point process, which has no such property. Hence, the rigidity of numbers for the Ginibre or GAF zero ensembles cannot be understood as a limit of the simple  rigidity of numbers for finite ensembles. The reasons behind this phenomenon, therefore, lie in the stronger spatial correlations of these ensembles.

In addition to answering our central question mentioned in the beginning, Theorems \ref{gin-1}-\ref{gaf-2} also provide information on the relative strength of spatial correlations in the Ginibre and the GAF zero ensembles. While a simple visual inspection suffices to (heuristically) distinguish a sample of the Poisson process from that of either the Ginibre or the GAF zero process (of the same intensity), the latter two are hard to set apart between themselves. It is therefore an interesting question to devise mathematical statistics that distinguish them.  The qualitative idea is that the spatial correlation is much stronger in the GAF zero process than in the Ginibre ensemble. There can be several possible approaches to quantify this heuristic observation. One important feature to look at, for instance, is the rate of decay of the hole probabilities. However, it turns out that both the Ginibre ensemble  and the Gaussian zero process behave similarly in this respect. For more details, one can refer to \cite{HKPV}. Our results  demonstrate that the GAF zeroes have  greater spatial dependence, in the sense that the point configuration in the exterior of a open set dictates  more about the one in its interior.

In \cite{STs1}, Sodin and Tsirelson compared the GAF zero process (CAZP in their terminology) with various models of perturbed lattices. They noticed that the lattice process 
\[  \{ \sqrt{3\pi}(k+li) + c e^{2\pi im/3} \eta_{k,l} : k,l \in \mathbb{Z},m=0,1,2 \}, \]
(where $\eta_{k,l}$ are i.i.d. standard complex Gaussians, $c\in (0,\infty)$ is a parameter and $i$ denotes the imaginary unit) achieves ``asymptotic similarity'' with the GAF zero process (in the sense that the variances of scaled linear statistics have similar asymptotic behaviour to those for the GAF zeroes).
Sodin and Tsirelson further observed that the above perturbed lattice model satisfied two conservation laws: one pertaining to the ``mass'' and another pertaining to the ``centre of mass''. They predicted similar conservation laws for the GAF zero process, although the sense in which such laws would hold was left open to interpretation.  Theorem \ref{gaf-1} establishes two conservation laws, one of which preserves the ``mass'' (i.e., the number of points), and the other one preserves the ``centre of mass''. Moreover, Theorem \ref{gaf-2} says that these are the only conservation laws  for the GAF zero process. We further note that among the perturbed lattice models in \cite{STs1}, the one that achieves ``asymptotic similarity'' with the Ginibre ensemble is the process
\[ \{ \sqrt{\pi}(k+li) + c \eta_{k,l} : k,l \in \mathbb{Z}  \}\] 
where $\eta_{k,l}$ and $c$ are as before. In this model, we have one conserved quantity (namely, the ``mass''). In Theorems \ref{gin-1} and \ref{gin-2}, we obtain a conservation law for the ``mass'' (i.e., the number of points) in the Ginibre ensemble, and further, show that there are no other conserved quantities.

En route to proving the main theorems mentioned above, we obtain  results that are interesting in their own right. For example we prove that the harmonic sum $\l(\sum_{z \in \Pi} \frac{1}{z} \r)$, for the Ginibre ensemble as well as for the GAF zero process, is a.s. finite (in a precise technical sense specified in Propositions \ref{invginest-tail} and \ref{invgafest-tail} and the remarks thereafter). In fact, we show that this sum has a finite first moment for both processes. It is not hard to see that the corresponding sum for the Poisson process does not converge in any reasonable sense. Even for the Ginibre or the GAF zero ensembles, the corresponding sum does not converge absolutely. The underlying reason for the conditional convergence is the mutual cancellation arising from the higher degree of symmetry exhibited by a typical point configuration in the Ginibre or the GAF zero process. This is yet another manifestation of the fact that the Gaussian zeros or the Ginibre eigenvalues exhibit a much more 
regular arrangement (which indicates greater rigidity) than, say, the Poisson process. 

In a more precise sense,  we can consider finite dimensional approximations $\G_n$ and $\F_n$ to $\G$ and $\F$ respectively (for details see Sections \ref{ginens} and \ref{gafens}) and define the finite sums $\a_k(n)=\l(\sum_{z \in \Pi} {1}/{z^k}\r)$ when $\Pi = \G_n$ or $\F_n$. Our results, as in Proposition \ref{conv-gin1} and Proposition \ref{conv-gaf1},  establish that both for the Ginibre and the GAF zeroes, these sums converge in probability as $n \to \infty$. The limit $\a_k$ can be justifiably taken to be an analogue of the sum $\l(\sum_{z\in \Pi} 1/z^k \r)$ for the respective limiting process $\G$ or $\F$.
 
We also provide a new proof of a reconstruction theorem for the planar GAF from its zeroes, which essentially says that the GAF zeroes determine a.s. the GAF itself, up to a factor of modulus 1. In what follows, $\a_k$ will denote the random variable introduced above for GAF zeroes, $P_k$ will be the $k$-th Newton polynomial (for details, see Section \ref{vie}). Define \[ a_k=(-1)^kP_k(\a_1,\cdots,\a_k)\] and \[ \chi= \lim_{k \to \infty} k^{1/2}\l(  \sum_{j=0}^{k-1} j! \l|a_j \r|^2   \r)^{-1/2} . \] The existence of the limit will be proved in the course  of proving Theorem \ref{reconstruction}.   We state the reconstruction theorem as: 
\begin{theorem}
\label{reconstruction} 
Consider the random analytic function ${\displaystyle g(z)=\sum_{k=0}^{\infty} \chi a_k z^k  }$, which is  measurable with respect to the GAF zeroes.
There is a random variable $\z$ with uniform distribution on $\mathbb{S}^1$ and independent of the GAF zeroes, such that a.s. we have $f(z)=\z g(z)$, where $f$ is the standard planar Gaussian analytic function.
\end{theorem}

It is interesting to compare Theorem \ref{reconstruction} with the Hadamard Factorization Theorem (see, e.g., \cite{Rud} Chapter 15) for reconstructing analytic functions from their zeroes, where we exploit the fact that the standard planar GAF is a.s. of order 2. In the case of the Hadamard factorization, the key problem is that there is a (random) analytic function of the form $\exp(p)$ (where $p$ is a polynomial of degree at most 2) that occurs as a factor in front of the canonical product formed from the Gaussian zeroes, and a priori there is no concrete information about this function. E.g., it can, in principle, depend on the Gaussian zeroes. However, in Theorem \ref{reconstruction} there is a concrete description of the factor $\z$ and also the precise dependence of $g$ on the GAF zeroes.

The description in Theorem \ref{reconstruction} is optimal in the sense that the factor $\z$ cannot be done away with. This can be seen from the fact that if $\theta$ is a random variable that is uniform in $\mathbb{S}^1$ and independent of the $\xi_i$-s (the $\xi_i$-s are i.i.d. standard complex Gaussians), then the random analytic functions $\theta f$ and $f$ are both distributed as planar GAF-s and have the same zeroes; hence from the zeroes of $f$ we can hope to recover the coefficients of $f$ only up to such a phase $\theta$. Here we recall the definition of the planar GAF $f(z):=\sum_{k=0}^{\infty}\frac{\xi_k}{\sqrt{k !}} z^k$.

A reconstruction theorem for the planar GAF from its zeroes has been obtained by Ben Hough in \cite{Bh} (Section 3, Lemma 9), using methods which closely parallel the proof of a similar reconstruction result for the zeroes of the Gaussian analytic function on the hyperbolic plane (Theorem 6 in \cite{PV}). 
Our approach here is, however, very different from \cite{Bh} and  \cite{PV},  and the reconstruction result follows in a brief and simple manner from the estimates we already obtain in Section \ref{EIPZ}. We thereby obtain a new expression for the GAF in terms of its zeroes, and en route we construct analogues of the classic Vieta formulas relating the zero set of the GAF to its coefficients. It would be interesting to compare the expression that we obtain to that obtained in \cite{Bh}; this seems to lead to new and intriguing identities satisfied a.s. by the zeroes of the planar GAF (for instance, just by comparing the value of $|f(0)|$).

Theorems \ref{gin-2} and \ref{gaf-2} naturally lead to the question whether something more detailed or more quantitative can be said about the conditional measures in question. E.g., can we conclude something more explicit  about the probability density function which is shown to exist in Theorems \ref{gin-2} and \ref{gaf-2}? This question is addressed in \cite{G12-2}, where it is shown that the conditional densities are, in fact, comparable to squared Vandermonde densities. In particular, this shows that even under spatial conditioning, the Gaussian zeroes and the Ginibre eigenvalues exhibit mutual repulsion. 

Theorems \ref{gin-1} and \ref{gaf-1} can be seen as statements on the conservation of moments. To be more precise, the number of points in a domain $\D$ can be thought of as the zeroth moment of the points in $\D$, and their sum can be thought of as their first moment. Thus, Theorems \ref{gin-1} and \ref{gin-2} correspond to the statement that the points outside $\D$ determine the zeroth moment of the inside points for the Ginibre ensemble, and nothing more. In the same spirit, Theorems \ref{gaf-1} and \ref{gaf-2} correspond to the statement that the points outside $\D$ determine the zeroth and the first moments of the inside points for the GAF zero ensemble, and nothing more. It is natural to ask whether there exist point processes where the points outside $\D$ determine the first $k$ moments (and nothing more) of the points in $\D$, for $k \ge 3$. In the forthcoming paper  \cite{GK}, a sequence of such point processes is constructed, which correspond to the zeroes of a natural family of random analytic functions with Gaussian coefficients. In the same paper, it is also established that a determinantal point process is rigid only if its determinantal kernel, considered as an integral operator, is a projection.

The central question of this paper is partially motivated by the concepts of insertion and deletion tolerance of a point process.
Consider a point process $\Pi$. Fix a bounded open set $\D$, and add to $\Pi$ a random point uniformly distributed in $\D$. Let us call this perturbed point process $\Pi_{\D}$. The point process $\Pi$ is said to be \textbf{insertion tolerant} if $\P[\Pi_{\D}] \ll \P[\Pi]$  for all bounded open sets $\D$. \textbf{Deletion tolerance} is similarly defined by deleting a point inside $\D$ (if one exists) uniformly at random. Insertion and deletion tolerance have been investigated by Burton and Keane (\cite{BK}) in the context of percolation, by Holroyd and Peres (\cite{HP}) for studying invariant allocation on the plane, and  by Heicklen and Lyons (\cite{HL}) in the setting of random spanning forests. They have also been studied as topics of their own importance by Holroyd and Soo (\cite{HS}). 

Questions on rigidity and tolerance of point processes can be studied in other settings, for instance, in dimensions other than 2 and in geometries other than the Euclidean setting. 
For the continuum sine kernel process on $\R$ (which arises as the universal bulk limit of Wigner random matrices), rigidity of the number of points in an interval is proved in \cite{G12}. In the same work, rigidity for the number of points is also established for a wide class of stationary determinantal point processes on $\Z$.  For i.i.d. Gaussian perturbations of $\Z^d$, deletion intolerance has been proved in \cite{HS} for $d=1,2$. In the forthcoming paper \cite{PS}, it is shown that for i.i.d. Gaussian perturbations (with mean 0 and variance $\sigma^2$) of $\Z^d$ for $d \ge 3$, there is a critical value $\sigma_c(d)$ such that the point process is deletion intolerant for $\sigma < \sigma_c(d)$ and deletion tolerant for $\sigma > \sigma_c(d)$. In other geometries, it has been shown in \cite{HS} that for the Gaussian analytic function on the hyperbolic disk, the point process of zeroes is both insertion and deletion tolerant. 

The rigidity and tolerance properties of point processes can be exploited to study various other phenomena on random point sets. For example, in \cite{GKP12}, we study applications of rigidity and tolerance to continuum percolation on the Ginibre and GAF zero ensembles. In \cite{G12}, we explore the relationship between rigidity phenomena in point processes and the completeness properties of certain systems of random exponential functions.

\section{Plan of the paper}
\label{overview}
In Section \ref{setmod}, we begin with a detailed description of the models that we study, and also provide an abstract framework in which other models having similar characteristics can be investigated. Further, we show in Section \ref{general} that for proving the main Theorems \ref{gin-1} - \ref{gaf-2}, it suffices to establish them in the case where $\D$ is a disk.  

In order to  study rigidity phenomena, we devise a unified approach in Section \ref{rigidityphen}, where Theorem \ref{rig} gives general criteria for a function (of the inside points) to be rigid with respect to a point process. We complete the proofs of Theorems \ref{gin-1} and \ref{gaf-1} by establishing the relevant criteria for the Ginibre and the GAF zero processes.

In Section \ref{limcond}, we study tolerance properties in the general setup introduced in Section \ref{genset}. Theorem \ref{abs} lays down conditions under which certain tolerance behaviour of a point process can be established. Proving Theorems \ref{gin-2} and \ref{gaf-2}, therefore, amounts to showing that the relevant conditions hold for our models. However, unlike the rigidity phenomena, this requires substantially more work, and is carried out in two stages for each process. First, we obtain some estimates for the point processes $\G_n$ and $\F_n$, which are finite approximations to $\G$ and $\F$ respectively   (see Sections \ref{ginens} and \ref{gafens} for definitions). For the Ginibre ensemble, this is done in Section \ref{gintolerance}, and for the GAF zeroes this is done in Section \ref{tolgafz}. Finally, we apply these estimates to deduce that the relevant conditions for tolerance behaviour hold for our models; this is carried out for the Ginibre ensemble  in Section \ref{limgin}  and for the 
GAF zeroes in Section \ref{limgafz}.

\section{Models and Setup}
\label{setmod}
\subsection{The Ginibre Ensemble}
\label{ginens}
Let us consider an $n \times n$ matrix $X^n, n\ge 1$ whose entries are i.i.d. standard complex Gaussians. 
 The vector of its eigenvalues, in uniform random order, has the joint density (with respect to the Lebesgue measure on $\C^n$) given by
\[
 \label{gindef}
p(z_1,\cdots,z_n)= \frac{1}{\pi^n\prod_{k=1}^n k!}  e^{-\sum_{k=1}^n |z_k|^2} \prod_{i<j} |z_i-z_j|^2.
\]
Recall that a determinantal point process on the Euclidean space $\R^d$ with kernel $K$ and background measure $\mu$ is a point process on $\R^d$ whose $k$-point intensity functions with respect to the measure $\mu^{\otimes k}$ are given by \[ \rho_k(x_1,\cdots,x_k)= \text{det} \bigg[  \l( K(x_i,x_j) \r)_{i,j=1}^k  \bigg]. \]
Typically, $K$ has to be such that the integral operator defined by $K$ is a non-negative trace class contraction mapping $L^2(\mu)$ to itself. For a detailed study of determinantal point processes, we refer the reader to \cite{HKPV}, \cite{Bo} and \cite{Sos}. More details on random matrices can be found in \cite{DG} and \cite{AGZ}. A simple calculation involving Vandermonde determinants shows that the eigenvalues of $X^n$
(considered as a random point configuration) form a determinantal point process on $\C$. Its kernel is given by $ K_n(z,w)=\sum_{k=0}^{n-1} \frac{(z \bar{w})^k}{k!}$ with respect to the background measure $ \, \mathrm{d}\gamma (z) = \frac{1}{\pi}e^{- |z|^2} \mathrm{d}\mathcal{L}(z)$ where $\mathcal{L}$ denotes the Lebesgue measure on $\C$. 
This point process is the Ginibre ensemble (of dimension $n$), which we will denote by $\G_n$.
As $n\rightarrow \infty$, these point processes converge, in distribution, to a determinantal point process given by the kernel $ K(z,w)=e^{z\bar{w}}=\sum_{k=0}^{\infty} \frac{(z \bar{w})^k}{k!}$ with respect to the same background measure $\gamma$. This limiting point process is the infinite Ginibre ensemble $\G$. It is known that $\G$ is ergodic under the natural action of the translations of the plane, see e.g. \cite{Sos} Section 3 and Theorem 7 therein. 

For more details on the Ginibre ensemble, we refer the reader to \cite{HKPV} Chapter 6.

\subsection{The GAF zero process}
\label{gafens}
Let $\{\xi_k \}_{k=0}^{\infty}$ be a sequence of i.i.d. standard complex Gaussians. Define 
\[
f_n(z) =\sum_{k=0}^{n}\xi_k \frac{z^k}{\sqrt{k!}} \,(\text{ for }n \ge 0) \, , \qquad \qquad f(z) =\sum_{k=0}^{\infty}\xi_k \frac{z^k}{\sqrt{k!}} \, \, .
\]
These are Gaussian processes on $\C$ with covariance kernels given by $K_n(z,w)=\sum_{k=0}^{n} \frac{(z \bar{w})^k}{k!}$ and $K(z,w)=\sum_{k=0}^{\infty} \frac{(z \bar{w})^k}{k!} = e^{z \bar{w}}$ respectively. A.s. $f_n$ and $f$ are  entire functions and the functions $f_n$ converge  to $f$ (in the sense of the uniform convergence of holomorphic functions on compact sets). It is not hard to see (e.g., via Rouche's theorem) that this implies that the corresponding point processes of zeroes, denoted by $\F_n$, converge a.s. to the zero process $\F$ of the GAF (in the sense of locally finite point configurations converging on compact sets). It is known that $\F$ is ergodic under the natural action of the translations of the plane, see e.g. \cite{HKPV} Chapter 2 and Proposition 2.3.7 therein.

For more details on the GAF zero process, we refer the reader to \cite{HKPV} Chapter 2.
\subsection{The General Setup}
\label{genset}
Fix a Euclidean space  $\Eu$ equipped  with a non-negative regular Borel measure $\mu$. Let $\S$ denote the space of  countable locally finite point configurations on $\Eu$.  Endow $\S$ with its canonical topology, namely, the topology on $\S$ is that of weak convergence of measures, under the identification of a locally finite point configuration with the counting measure it induces on $\C$ (which gives $\S$ a canonical Borel $\sigma$-algebra). Equivalently, the topology of $\S$ is that of the convergence (in the Hausdorff sense) of finite configurations obtained by restriction to open balls in $\Eu$.  It is known that $\S$, with this topology, is a Polish space. Fix a bounded open set $\D \subset \Eu$. Corresponding to the decomposition $\Eu=\D \cup \D^c$, we have ${\S} =\S_{\inn} \times \S_{\out}$, where $\S_{\inn}$ and $\S_{\out}$  denote the spaces of finite point configurations on $\D$ and  locally finite point configurations on $\D^c$ respectively. 

Let $\Xi$ be a measure space equipped with a probability measure $\tP$. For a random variable $Z:\Xi \to \mathcal{X}$ (where $\mathcal{X}$ is a Polish space), we define the push forward $Z_*\tP$ of the measure $\tP$ by  $Z_*\tP(A)=\tP(Z^{-1}(A))$ where $A$ is a Borel set in $\xo$. Also, for a point process $Z':\Xi \to \S$, we can define point processes $Z_{\inn}':\Xi \to \S_{\inn}$ and $Z_{\out}':\Xi \to \S_{\out}$ by restricting the random configuration $Z'$ to $\D$ and $\D^c$ respectively. 

Let $X,X^n: \Xi \rightarrow \S$ be random variables such that $\tP$-a.s., we have $X^n \rightarrow X$ (in the topology of $\S$). We demand that the point processes $X,X^n$ have their first intensity measures absolutely continuous with respect to $\mu$. We can identify $X_{\inn}$ (by taking the points in uniform random order)  with the random vector $\uz$ which lives in $\bigcup_{m=0}^{\infty}\D^m$. The analogous quantity for $X^n$ will be denoted by $\uz^n$.

For our models we can take $\Eu$ to be $\C$, $\mu$ to be the Lebesgue measure, and $\D$ to be a bounded open set.  

In the case of the Ginibre ensemble, we can define the processes $\G_n$ and $\G$ on the same underlying probability space so that a.s. we have $\G_n \subset \G_{n+1}\subset \G$ for all $n \ge 1$. For a reference, see \cite{Go}, in particular Theorem 3 therein. A bit of explanation is in order as to how we obtain this assertion. The $\G_n$-s and $\G$ are determinantal point processes, whose kernels satisfy the Loewner ordering condition in Theorem 3, \cite{Go}. More precisely, if $K_n$ is the kernel corresponding to $\G_n$ and $K$  is the kernel corresponding to $\G$, then we have \begin{equation} \label{Loew} K_n \preceq K_{n+1} \preceq K,\end{equation} where $\preceq$ is the Loewner order. By $L_1 \preceq L_2$ for two operators $L_1$ and $L_2$ acting on a Hilbert space $H$, we mean that $L_1 - L_2$ is a non-negative definite operator acting on $H$. We extend this notion in the obvious way to the case when $L_1$ and $L_2$ are instead the kernels of integral operators acting on $H$. The inequality (\ref{Loew}) is true because the $K_n$-s correspond to projections onto increasingly large subspaces of the Fock-Bargmann space (increasing with $n$). Thus, Theorem 3 in \cite{Go}  implies that  for every $n$, the process $\G_n$ is stochastically dominated by $\G_{n+1}$, which in turn is stochastically dominated by $\G$. Applying Strassen's Theorem (see, e.g., \cite{Strassen}) for point processes, we get the desired coupling.   We take $(\Xi,\tP)$ to be this underlying probability space (where all the processes are coupled together), $X^{n}=\G_n$ and $X=\G$.

In the case of the Gaussian zero process, we take $(\Xi, \tP)$ to be a measure space on which we have countably many standard complex Gaussian random variables denoted by $\{ \xi_k \}_{k=0}^{\infty}$. Then $X^n$ is the zero set of the polynomial $f_n(z)=\sum_{k=0}^{n}\xi_k \frac{z^k}{\sqrt{k!}}$, and $X$ is the zero set of the entire function $f(z)=\sum_{k=0}^{\infty}\xi_k \frac{z^k}{\sqrt{k!}}$. The fact that $X^n \to X$ $\tP$-a.s. follows from Rouche's theorem. 

\section{A broad outline of the proof strategy}

 We devise a general approach to prove rigidity and tolerance of point processes. In our specific models of interest, we execute this approach, which principally involves obtaining precise estimates on certain quantities related to the finite dimensional approximations of our models.
 
 The proof of the rigidity theorems involves statistics with small variance. Let $A$ be a statistic (of the point process $X$ living on a space $\mathcal{E}$) whose rigidity we want to prove, in the context of a domain $\D \subset \mathcal{E}$ . Generally, $A=A(X_{\inn})$ will be a function of the points $X_{\inn}$ inside the domain $\D$. Suppose we can construct a family of statistics $H_{\eps}$ of the point process $X$ such that $\mathrm{Var}[H_{\eps}] \to 0$ as the parameter $\eps \to 0$. Suppose also that $H_{\eps}$ can be written as \[H_{\eps}(X)=A(X_{\inn}) + \wt{H}_{\eps}(X_{\out}),\] where $\wt{H}_{\eps}$ is a statistic which depends only on the points $X_{\out}$ outside $\D$. Since $\mathrm{Var}[H_{\eps}] \to 0$ as $\eps \to 0$, for small $\eps$ we have, roughly speaking, $H_{\eps}\approx \E[H_{\eps}]$ (in the sense that the difference between the two quantities is small with high probability).
 
 But $\E[H_{\eps}]$ can be computed, in principle, using the one and multi point intensity measures of the point process $X$. This means that \[A(X_{\inn}) + \wt{H}_{\eps}(X_{\out}) \approx \E[H_{\eps}],\] which is known for each $\eps$. The quantity $\wt{H}_{\eps}(X_{\out})$ can also be computed exactly, because the point configuration $X_{\out}$ is known. This suggests that for each $\eps$, we get an approximation to $A(X_{\inn})$ (in the form of $\E[H_{\eps}] -  \wt{H}_{\eps}(X_{\out}) $), which a.s. converges to its actual value as $\eps \to 0$. Typically, the statistic $H_{\eps}$ would be a linear statistic, meaning that there is a function $\Phi_{\eps}: \mathcal{E} \to \C$ such that $H_{\eps}(X)= \int \Phi_{\eps} \mathrm{d}[X]$. In this case, we have $\E[H_{\eps}]= \int \Phi_{\eps} \mathrm{d}\rho_1(x)$, where $\rho_1$ is the first intensity measure of the point process. Often (though not always), the family of functions $\Phi_{\eps}$ can be taken to be the scalings of a single function $\Phi$ (chosen to have some specific properties); e.g. $\Phi_{\eps}(x)=\Phi(\eps x)$.
 
 Our approach to tolerance proceeds via finite dimensional approximations $X^n$ to the (infinite size) point process $X$ (which are coupled together).  The idea of tolerance pertains to proving the absolute continuity of certain conditional measures with respect to a standard reference measure. Let us consider a scenario where we want to show that a certain measure $\mu$  is mutually absolutely continuous with respect to the Lebesgue measure on an open set $E$ (in our problem $\mu$ will usually correspond to the conditional distribution of the vector of points corresponding to $X_{\inn}$ given $X_{\out}$). This would mean that $\mu$ has a density $f$ with respect to the Lebesgue  measure on $E$ (which we denote by $\la$), that is positive almost everywhere. Let us further simplify to the case when we need to prove that the density $f$ is, in fact,  bounded away from 0 and $\infty$. While we might not know a priori that $\mu$ even has a density, it would suffice to prove (in this case) that for any small ball $B \subset E$, we have $\mu(B)/\la(B)$ is bounded between two constants, independent of $B$. 
 
 If we want to do this when $\mu$ is a conditional measure coming from a point process $X$ of infinite size, then obtaining explicit estimates on $\mu(B)$ directly from a description of $X$ can be an intractable problem.
 However, for a point process of finite size, we can often write down concrete expressions for the conditional measures, resulting from the fact that often there is a closed form expression for the joint density of the points.  Then, we can estimate the ratio  $\mu_n(B)/\la(B)$, where $\mu_n$ is the conditional distribution of the (vector of points corresponding to) $X^n_{\inn}$ , given $X^n_{\out}$. If we can establish upper and lower bounds on $\mu_n(B)/\la(B)$ which are uniform in $n$, then we can try to take limits and claim that similar bounds hold for $\mu(B)/\la(B)$. But as $n$ changes, the conditioning variable (i.e., $X^n_{\out}$) also changes, in a way that is determined by the coupling of $X^n$-s and $X$ (ensuring that $X^n_{\out} \to X_{\out}$). Thus, the convergence of the (random) measures $\mu_n$ to $\mu$  (random because of their dependence on $X^n_{\out}$ and $X_{\out}$) poses a significant technical challenge. In the event that appropriate probabilistic estimates can be obtained from the finite dimensional approximations, this problem can be tackled and the limit pushed through; this is the content of Theorem \ref{abs}. 
 
 We approach the question of tolerance in a modular fashion. In Section \ref{limcond}, we work in the general setup of coupled point processes and state Theorem \ref{abs}. Then in separate sections, we obtain the necessary finite dimensional estimates for the models of our interest, and show that these indeed satisfy the conditions laid out in Theorem \ref{abs}. This enables us to carry out the approximation procedure outlined above.
 
\section{Reduction from a general $\D$ to a disk}
\label{general}

In this Section we intend to prove that in order to obtain Theorems \ref{gin-1}-\ref{gaf-2}, it suffices to consider the case where $\D$ is an open disk centred at the origin. We will demonstrate the proof for Theorems \ref{gaf-1} and \ref{gaf-2}, the arguments for Theorems \ref{gin-1} and \ref{gin-2} are on  similar lines. 
 
Let $\D$ be a bounded open set in $\C$. By the translation invariance of the GAF zero ensemble, we take the origin to be in the interior of $\D$. Let $\D_0$ be a disk (centred at the origin) which contains $\overline {\D}$ in its interior (where $\overline{\D}$ is the closure of $\D$).

Suppose we know the point configuration $\F_{\out}$ to be equal to $\uout$. Further, suppose that we can show that the point configuration $\uoutd$ outside  $\D_0$ determines the number $N_0$ and the sum $S_0$ of the points inside $\D_0$ a.s. Since we also know the number and the sum of the points inside $\D_0 \setminus \D$, we can determine the number $N$ as well as the sum $S$ of the points  in $\D$.  This proves the rigidity theorem  for the GAF zero ensemble (Theorem \ref{gaf-1}) for a general $\D$.

Now suppose we have the tolerance Theorem \ref{gaf-2} for a disk.  To obtain Theorem \ref{gaf-2} for $\D$ , we appeal to the tolerance Theorem \ref{gaf-2} for the disk $\D_0$. Define 
\[\Sigma:=\bigg\{(\la_1,\cdots,\la_{N}):\sum_{j=1}^{N}\la_j=S, \la_j \in \D \bigg\} \]
and
\[\Sigma_0 := \bigg\{(\la_1,\cdots,\la_{N_0}):\sum_{j=1}^{N_0}\la_j=S_0, \la_j \in \D_0 \bigg\}.\]


The conditional distribution of the vector of points inside $\D_0$, given $\uoutd$, lives on $\Sigma_0$, in fact it has a density $f_0$ which is positive a.e. with respect to Lebesgue measure on $\Sigma_0$. Let there be $k$ points in $\D_0 \setminus \D$  and let their sum be $s$, clearly we have $N=N_0-k$ and $S=S_0-s$. We parametrize $\Sigma$ by the last $N-1$ coordinates. Note that the set $U:=\{(\la_2,\cdots,\la_N): (S-\sum_{j=2}^N \la_j, \la_2,\cdots,\la_N) \in \Sigma\}$ is an open subset of $\D^{N-1}$. Further, we define the set $V:=\{(\la_1,\cdots,\la_k): \la_i \in \D_0 \setminus \D, \sum_{i=1}^k \la_i =s \}$.

Let the points in $\D_0\setminus \D$, taken in uniform random order, form the vector $\underline{\mathbf{z}}=(z_1,\cdots,z_k)$. Then we can condition the vector of points in $\D_0$ to have its last $k$ coordinates equal to $\underline{\mathbf{z}}$, to obtain the following formula for the conditional density of the vector of points in $\D$ at $(\z_1,\cdots,\z_N) \in \Sigma$ (with respect to the Lebesgue measure on $\Sigma$):
\begin{equation} 
\label{condng}
f(\z_1,\z_2,\cdots,\z_N)=\frac{f_0(\z_1, \z_2,\cdots,\z_N,z_1,\cdots,z_k)}{\int_U f_0(s-(\sum_{j=2}^N{w_j}), w_2,\cdots,w_N,z_1,\cdots,z_k)dw_2 \cdots dw_N}.
\end{equation}
It is clear that for a.e. $\underline{\mathbf{z}} \in V$, we have $f$ is strictly positive a.e. with respect to Lebesgue measure on $\Sigma$, because the same is true of $f_0$ on $\Sigma_0$.

\section{The Rigidity Phenomenon}
\label{rigidityphen}
We begin by giving a precise definition of rigidity. Recall the general setup in Section \ref{genset}.
\begin{definition}
 \label{defrig}
A measurable function $f_{\inn} : \S_{\inn} \rightarrow \C $ is said to be \textbf{rigid} with respect to the point process $X$ on $\S$ if there is a measurable function  $f_{\out}:\S_{\out} \rightarrow \C$ such that a.s. we have $f_{\inn}(X_{\inn})=f_{\out}(X_{\out})$.
\end{definition}

In this section, we prove that the number of points in $\D$ in the case of the Ginibre ensemble and the number as well as the sum of the  points in $\D$ for the GAF zero process are rigid. In fact, we will state some general conditions that ensure such rigid behaviour, and then show that the Ginibre and the GAF satisfy the relevant conditions. 

From here on in this section, we will consider point processes on $\C$. We will use linear statistics of point processes as the main tool that will enable us to obtain rigidity results for such point processes. 

\begin{definition}
 \label{linearstat}
Let $\Pi$ be a point process on $\C$ and let $\varphi$ be a compactly supported continuous function on $\C$.  The \textbf{linear statistic} corresponding to $\varphi$ is the random variable $\int \varphi \, \mathrm{d}[\Pi] = \sum_{z\in \Pi} \varphi(z) $.
\end{definition}
By a  $C_c^k$ function on  $\C$, we mean a compactly supported $C^k$ function on $\C$.

We can now state:
\begin{theorem}
 \label{rig}
Let $\Pi$ be a point process on $\C$ whose first intensity is absolutely continuous with respect to  Lebesgue measure, and let $\D$ be a bounded open set in $\C$. Let $\varphi$ be a continuous function on $\C$. Suppose that for any $1>\eps>0$, we have a $C_c^2$ function $\Phi^{\eps}$ such that $\Phi^{\eps}=\varphi$ on $\D$, and $\mathrm{Var}\l( \int_{\C} \Phi^{\eps} \mathrm{d}[\Pi]  \r) < \eps$. Then the function $T:\S_{\inn} \to \C$ given by $T(\Upsilon_{\inn})=\int_{\D} \varphi \hspace{5 pt} \mathrm{d}[\Upsilon_{\inn}]=\sum_{z \in \Upsilon \cap \D} \varphi(z)$ is rigid with respect to $\Pi$.
\end{theorem}

\begin{proof}
Consider the sequence of $C_c^2$ functions $\Phi^{2^{-n}}, n \ge 1$. Note that $\E \l[ \int_{\C} \Phi^{2^{-n}} \mathrm{d}[\Pi] \r] = \int_{\C} \Phi^{2^{-n}} \rho_1 \mathrm{d}\mathcal{L}$ where $\rho_1(z)$ is the one point intensity function of $\Pi$.  It follows from Chebyshev's inequality that \[ \P \l(  \l| \int_{\C} \Phi^{2^{-n}} \mathrm{d}[\Pi] - \E \l[ \int_{\C} \Phi^{2^{-n}} \mathrm{d}[\Pi] \r] \r| >  2^{-n/4} \r) \le 2^{-n/2}.\]  The Borel Cantelli lemma implies that with probability 1, as $n \to \infty$ we have \[\label{rig-1} \l| \int_{\C} \Phi^{2^{-n}} \mathrm{d}[\Pi] - \E \l[ \int_{\C} \Phi^{2^{-n}} \mathrm{d}[\Pi] \r] \r| \to 0. \]
But  \[ \int_{\C} \Phi^{2^{-n}} \mathrm{d}[\Pi] =  \int_{\D} \Phi^{2^{-n}} \mathrm{d}[\Pi] +  \int_{\D^c} \Phi^{2^{-n}} \mathrm{d}[\Pi]. \]
Thus we have, as $n \to \infty$  
\begin{equation} 
\label{rig-2}
\l| \int_{\D} \Phi^{2^{-n}} \mathrm{d}[\Pi] +  \int_{\D^c} \Phi^{2^{-n}} \mathrm{d}[\Pi] - \int_{\C} \Phi^{2^{-n}} \rho_1 \mathrm{d}\mathcal{L} \r| \to 0. 
\end{equation}
If we know $\Pi_{\out}$, we can compute $ \int_{\D^c} \Phi^{2^{-n}} \mathrm{d}[\Pi]$ exactly, also $\rho_1$ is known explicitly; in case of a translation invariant point process $\Pi$ it is, in fact, a constant $c(\Pi)$. Hence, from the limit in (\ref{rig-2}), a.s. we can obtain $\int_{\D} \Phi^{2^{-n}} \mathrm{d}[\Pi] = \int_{\D} \varphi \mathrm{d}[\Pi]$   as the limit
\[ \lim_{n \to \infty} \l( \int_{\C}  \Phi^{2^{-n}} \rho_1 \mathrm{d}\mathcal{L} - \int_{\D^c} \Phi^{2^{-n}} \mathrm{d}[\Pi] \r) . \]
\end{proof}

We now use Theorem \ref{rig} to establish Theorems \ref{gin-1} and \ref{gaf-1}.

\begin{proof} [\textbf{Proof of Theorems \ref{gin-1} and \ref{gaf-1}}]
We have already seen in Section \ref{general} that it suffices to take $\D$ to be a disk centred at the origin. We intend to construct functions $\Phi^{\eps}$ as in Theorem  \ref{rig}.  

Let $r_0$ be the  radius of $\D$. Fix $\eps >0$.

We begin with the continuous function $\widetilde{\Psi}$ on $\R_{+} \cup \{0\}$ such that 
\[ \widetilde{\Psi}(r) = \begin{cases}
                           1 & \text{ for $0\le r \le 2r_0$}, \\
                           -\frac{\eps}{2} \log r + \frac{\eps}{2} \log 2r_0 +1  & \text{ for $2r_0 \le r \le 2r_0\exp(2/\eps)$}, \\
                           0 & \text{ for $r \ge 2r_0\exp(2/\eps)$ }.
                          \end{cases}
  \]

Notice that  ${\widetilde{\Psi}}'(r)=-\eps/2r$ and ${\widetilde{\Psi}}''(r)=\eps/2r^2$ for $2r_0 \le r \le 2r_0\exp(2/\eps)$.  
We can then smooth the function ${\widetilde{\Psi}}$ at $ 2r_0$ and $ 2r_0\exp (2/\eps)$ such that the resulting function $\Psi_1$ is $C^2$ on the positive reals, takes the values $1$ for $r \le r_0$ and $0$ for $r \ge 2r_0e^{2/\eps}+1$, and satisfies the inequalities $|\Psi_1'(r)|\le \eps/r$ and $|\Psi''(r)|\le \eps/r^2$ for all $r > 0$. Finally, we define the radial $C_c^2$ function $\Psi$ on $\C$ as $\Psi(z)=\Psi_1(|z|)$.

For $\G$, we know (see \cite{RV} Theorem 11) that there exists a constant $C_1>0$ such that for every radial $C_c^2$ function $\Psi$ we have   \[\mathrm{Var} \l( \int_{\C} \Psi \, \mathrm{d}[\G] \r) \le C_1 \int_{\C} | \nabla \Psi (z) |^2 \, \mathrm{d}\mathcal{L}(z). \] But from the definition of $\Psi$ it is clear that $\int_{\C} | \nabla \Psi (z) |^2 \, \mathrm{d}\mathcal{L}(z) \le C_2\eps$.We apply Theorem \ref{rig} with $\varphi \equiv 1$, and choose $\Phi^{C_1C_2\eps}=\Psi$ as defined above; recall that $\Psi \equiv 1$ on $\D$. 

For $\F$, we know (see \cite{NS1} Theorem 1.1) that there exists a constant $C_3>0$ such that every $C_c^2$ function $\vartheta$ satisfies \[\mathrm{Var} \l( \int_{\C} \vartheta \, \mathrm{d}[\F] \r) \le C_3 \int_{\C} | \Delta \vartheta (z) |^2 \, \mathrm{d}\mathcal{L}(z). \] For the rigidity of the number of points of $\F$ in $\D$ we make exactly the same choice as we did for $\G$, and note that $\int_{\C} | \Delta \Psi (z) |^2 \, \mathrm{d}\mathcal{L}(z) \le C_4\eps^2$. For the rigidity of the sum of points of $\F$ in $\D$ we intend to apply Theorem \ref{rig} with $\varphi(z)=z$. We consider the function $\theta(z)=z\Psi(z)$, $\Psi$ as before. Observe that $\Delta \theta(z)=4\frac{\partial \Psi}{\partial \bar{z}}(z) + z \Delta \Psi(z)$. Using this, for $\eps < 1$, we get $\int_{\C} |\Delta (z\Psi(z))|^2 \, \mathrm{d}\mathcal{L}(z) \le C_5\eps$. It remains to note that for $z \in \D$ we have $\Psi(z)=1$ and $\theta(z)=z$.
\end{proof}

\section{Tolerance: Limits of Conditional Measures}
\label{limcond}
The tolerance properties are established for both models by obtaining explicit bounds on conditional probability measures for finite  approximations (finite matrices in case of Ginibre and polynomials for GAF) and then passing to the limit. In this Section, we state and prove some general conditions (in the context of the abstract setup considered in Section \ref{genset}), which will enable us to make the transition from the finite ensembles to the infinite one. 

We start with the following general proposition:

\begin{proposition}
\label{meas1}
Let $\Gamma$ be a second countable topological space. Let $\Sigma$ be a countable basis of open sets in $\Gamma$ and let $\A:=\{ \cup_{i=1}^k \sigma_i : \sigma_i \in \Sigma, k \ge 1 \}$. Let $c>0$. To verify that two non-negative  Borel  measures $\mu_1$ and $\mu_2$ on   $\Gamma$ satisfy $\mu_1(B) \le c \mu_2(B)$ for all Borel sets $B$ in $\Gamma$ where $\mu_2$ is known to be a regular Borel measure,  it suffices to verify the inequality for all sets in $\A$. 
\end{proposition}

\begin{proof} Any open set $U \subset \Gamma$ is a countable union $\bigcup_{i=1}^{\infty}\sigma_i , \sigma_i \in \Sigma$ because the sets in $\Sigma$ form a basis for the topology on $\Gamma$. If we have $\mu_1(\bigcup_{i=1}^n\sigma_i) \le c \mu_2(\bigcup_{i=1}^n\sigma_i)$, then we can let $n \to \infty$ to obtain $\mu_1(U) \le c \mu_2(U)$. Once we have the inequality for all open sets $U$, we can extend it to all Borel sets as follows. Let $B$ be a Borel set and $U$ be an open set containing $B$. Then we have \[ \mu_1(B) \le \mu_1(U) \le c\mu_2(U). \] But for a regular Borel measure $\mu$ and for any Borel set $B \subset \Gamma$, we have $\mu(B)=\text{inf}_{B \subset U } \mu(U)$ where the infimum is taken over all open sets $U$ containing $B$. Therefore, taking infimum over all open sets $U$ containing $B$ on the right hand side of the last display, we obtain \[\mu_1(B) \le c \mu_2(B), \] as desired. 
\end{proof}

In this Section, we will work in the setup of Section \ref{genset}, specifying $\D$ to be an open ball, and requiring that the first intensity of our point process $X$ is absolutely continuous with respect to the Lebesgue measure on $\Eu$. We further assume that $X$ exhibits rigidity of the number of points. In other words, 
\begin{definition}
 \label{nopoints}
There is a measurable function $N:\S_{\out} \to \nat$ such that a.s. we have \[\text{ Number of points in } X_{\inn} = N(X_{\out}). \]
\end{definition}

In such a situation, we can identify $X_{\inn}$ (by taking the points in uniform random order)  with a random vector $\uz$   taking values in $\D^{N(X_{\out})}$.  Studying the conditional distribution $\rhoo(X_{\out},\cdot)$ of $X_{\inn}$ given $X_{\out}$ is then the same as studying the conditional distribution of this random vector given $X_{\out}$. We will denote the latter distribution by $\rho(X_{\out},\cdot)$. Note that it is supported on $\D^{N(X_{\out})}$ (see Section \ref{intro} for details).    

For $m >0$, let $\mathfrak{W}_{\inn}^m$ denote the countable basis for the topology on $\D^m$ formed by open balls contained in $\D^m$ and having rational centres and rational radii. We define the collection of sets $\AA:=\{ \bigcup_{i=1}^k A_i: A_i \in \mathfrak{W}_{\inn}^m , k \ge 1\}$.

Fix an integer $n \ge 0$, a closed annulus $B \subset \D^c$ whose centre is at the origin and which has a rational inradius (or inradius equal to the radius of $\D$) and a rational outradius, and a collection of $n$ disjoint open balls $B_i$ with rational radii and centres having rational coordinates such that $\{ B_i \cap \D^c \}_{i=1}^n \subset B$. Let $\Phi(n,B,B_1, \cdots, B_n)$ be the Borel subset of $\S_{\out}$ defined as follows:
\[ \Phi(n,B,B_1,\cdots,B_n)= \{ \Upsilon \in \S_{\out} : | \Upsilon \cap B |=n, | \Upsilon \cap {\B_i} | =1 \}.\] 
Then the countable collection $\Sigma_{\out}=\{ \Phi(n,B,B_1,\cdots,B_n): n,B, B_i \text{ as above} \}$ is a basis for the topology of $\S_{\out}$. This follows from the characterization of the topology (on the space of locally finite point configurations) as that of the convergence (in the Hausdorff sense) of the restrictions to balls. 
Define the collection of sets $\mathcal{B}:=\{\bigcup_{i=1}^k \Phi_i: \Phi_i \in \Sigma_{\out}, k\ge 1 \}$.

We will denote by $\Omega^m$ the event that $|X_{\inn}|=m$, and by $\Omega^m_n$ we will denote the event $\l|X^n_{\inn}\r|=m$; we recall here that $X^n$ is an approximation of $X$. 

\begin{definition}
Let $P$ and $Q$ be two sets, and let $\alpha(p,q)$ and $\beta(p,q)$ be non-negative functions defined on $P \times Q$. We write $\alpha(p,q) \asymp_{q} \beta(p,q)$ if there exist positive functions $k_1(q),k_2(q)$ defined on $Q$ such that \[ k_1(q)\alpha(p,q) \le \beta(p,q) \le k_2(q)\alpha(p,q) \text{ for all } p,q. \] The main point is that $k_1,k_2$ in the above inequalities are uniform on $P$.  
\end{definition}
We will also use the notation introduced in Section \ref{genset}.
We will define an ``exhausting'' sequence of events as: 
\begin{definition}
A sequence of events $\{\Omega(j)\}_{j \ge 1}$ is said to \textsl{exhaust} another event $\Omega$ if  $\Omega(j) \subset \Omega(j+1) \subset \Omega$  for all $j $ and  $\tP(\Omega \setminus \Omega(j)) \rightarrow 0$ as $j \to \infty$.
\end{definition}

Let $\mathcal{M}(\D^{m})$ denote the space of all probability measures on $\D^m$. Let $\mathfrak{B}(\D^m)$ denote the Borel sigma-algebra on $\D^m$.
For a mapping $Y: S \to T$,  by $Y^{-1}(B)$ (where $B \subset T$) we mean the set $\{\o \in S: Y(\o) \in B\}$. We recall the notion of two measures $\mu_1$ and $\mu_2$ being mutually absolutely continuous, denoted $\mu_1 \equiv \mu_2$, from Section \ref{intro}. Finally, recall the definition of $\AA$ and $\mathcal{B}$ from the beginning of this section.

Now we are ready to state the following important technical reduction:

\begin{theorem}
\label{abs}
Fix $m\ge 0$ such that $\tP(\Omega^m)>0$. 
Suppose that: 

(a) There is a map $\nu: \S_{\out} \times \mathfrak{B}(\D^m) \rightarrow [0,1] $ such that for each $\uout \in \S_{\out}$ we have $\nu(\uout,\cdot) \in \mathcal{M}(\D^{m})$ and for each Borel set $A \subset \D^m$, the map $ \uout \rightarrow \nu(\uout,A)$ is measurable. 

(b) For each fixed $j$ we have a sequence $\{ n_k \}_{k \ge 1}$ (which might depend on $j$) and corresponding events $\Omega_{n_k}(j)$     such that: 


$\mathrm{(i)} \hspace{4pt} \Omega_{n_k}(j) \subset \Omega^m_{n_k}$. 

$\mathrm{(ii)} \hspace{4pt} \Omega(j) := \varliminf_{k \to \infty} \Omega_{n_k}(j) $ exhaust $\Omega^m$ as $j \uparrow \infty$.

$\mathrm{(iii)}$ For all $A \in \mathfrak{A}^m$ and $B \in \mathcal{B}$ we have
\begin{equation} 
\label{abscond}
\tP [  (\xin^{n_k} \in A) \cap  ( X_{\out}^{n_k} \in B) \cap \Omega_{n_k}(j) ] \asymp _j \int_{(X_{\out}^{n_k})^{-1}(B) \cap \Omega_{n_k}(j)  } \nu(\xout({\xi}),A) \mathrm{d}\tP(\xi) + \vartheta(k;j,A,B). \tag{C3}  \end{equation}                              
where $\lim_{k \to \infty} \vartheta(k;j,A,B) = 0$ for each fixed $j,A$ and $B$ .
                                                                                                                                                                         
Then  a.s. on the event $\Omega^m$ we have 
\begin{equation}
\label{targeteq}
 \rho(\xout,\cdot) \equiv  \nu(\xout,\cdot). \tag{C4} \end{equation}
\end{theorem}
\vspace{5 pt}

We defer the proof of Theorem \ref{abs} to Section \ref{techproof}.

We conclude this section with the following simple observations:
\begin{remark}
If Theorem \ref{abs} holds for all $m \ge 0$, then we can conclude that (\ref{targeteq}) holds a.e. $\xi \in \Xi$. 
\end{remark}
\begin{remark}
\label{equivcond}
 The relation (\ref{abscond}) is equivalent to the following statement: for each fixed $j$, there exist positive quantities $C_j,c_j$  such that for any  $A \in \mathfrak{A}^m$, $B \in \mathcal{B}$ we have: 
 \[ \varlimsup_{k \to \infty} \tP [  (\xin^{n_k} \in A) \cap  ( X_{\out}^{n_k} \in B) \cap \Omega_{n_k}(j) ] \le C_j \varlimsup_{k \to \infty} \int_{(X_{\out}^{n_k})^{-1}(B) \cap \Omega_{n_k}(j)  } \nu(\xout({\xi}),A) \mathrm{d}\tP(\xi) \]
 and \[ \varliminf_{k \to \infty} \tP [  (\xin^{n_k} \in A) \cap  ( X_{\out}^{n_k} \in B) \cap \Omega_{n_k}(j) ] \ge c_j \varliminf_{k \to \infty} \int_{(X_{\out}^{n_k})^{-1}(B) \cap \Omega_{n_k}(j)  } \nu(\xout({\xi}),A) \mathrm{d}\tP(\xi).  \]
We emphasize that $C_j,c_j$ in the above pair of inequalities do not depend on $A$ or $B$. 
\end{remark}

\begin{remark}
 In Theorem \ref{abs}, one should think of $\{n_k\}_k$ as the index of an approximating subsequence of point processes $\{X^{n_k}\}_k$. The parameter $j$, on the other hand, enumerates an exhaustion of the relevant part $\O^m$ of the sample space (as a toy example, one can think of the subsets $\{|Y|\le j\}$ for some random variable $Y$, which exhaust the entire sample space as $j \to \infty$).  
\end{remark}

\section{Tolerance of the Ginibre Ensemble }
\label{gintolerance}
In this section we obtain several estimates necessary to prove Theorem \ref{gin-2} in the case where $\D$ is a fixed disk centred at the origin. Fix an integer $m \ge 0$. 
\subsection{Estimates for $\G_n$ }
\label{ratcondgin}
Let $\del\in (0,1)$. We consider the event $\Omega_n^{m,\delta}$ which entails that $\mathcal{G}_n$ has exactly $m$ points inside $\D$, and there is a $\delta$ separation between $\partial \D$ and $(\G_n)_{\out}$; the analogous event for $\mathcal{G}$ will be denoted by $\Omega^{m,\delta}$.  Notice that $\Omega^{m,\delta}$ has positive probability (which is bounded away from 0 uniformly in $\delta$), so by the convergence of $\mathcal{G}_{n}$-s to $\mathcal{G}$, we obtain that $\Omega_n^{m,\delta}$-s have positive probability that is uniformly bounded away from 0 for large enough $n$.

A brief explanation is in order for the last assertion regarding $\P(\Omega^{m,\delta})>0$. First, note that the events $\Omega^{m,\delta} \uparrow \Omega^m$, so it suffices to prove $\P(\Omega^m)>0$. But $\G$ being a determinantal point process with the kernel discussed in Section \ref{setmod}, the number of points of $\G$ in $\D$ is a sum of independent Bernoulli random variables, each with success probability strictly between 0 and 1. This implies that for each $m \ge 0$, we have $\P(\Omega^m)>0$. 

We denote the  points of $\G_n$ inside  $\D$ (in uniform random order) by $\uz=(\zeta_1,\cdots, \zeta_m)$ and those outside $\D$ (in uniform random order) by ${\uo}=(\omega_1,\omega_2,\cdots,\omega_{n-m})$. Following the notation introduced in Section \ref{intro}, for a vector $(\uz,\uo)$ as above, we denote $\Upsilon_{\inn}=\{\z_i\}_{i=1}^m$ , $\Upsilon_{\out} =\{\o_j \}_{j=1}^{n-m}$, and $\Upsilon=\Upsilon_{\inn}\cup \Upsilon_{\out}$.  For a vector $\underline{\gamma}=(\gamma_1,\cdots,\gamma_N)$ of $N$ points in $\C$, we denote by $\Delta(\underline{\gamma})$  the Vandermonde determinant $\prod_{i < j}(\gamma_i -\gamma_j)$. For two vectors $\underline{\gamma}_1,\underline{\gamma}_2$ we set $\Delta(\underline{\gamma}_1,\underline{\gamma}_2)=\Delta(\underline{\gamma})$ where $\underline{\gamma}$ denotes the concatenated vector $(\underline{\gamma}_1,\underline{\gamma}_2)$.

Then the conditional distribution $\rho(\Upsilon_{\out},\uz)$ of $\uz$ given $\Upsilon_{\out}$ has the  density
\begin{equation}
  \label{condgin}
\rho^n_{\uo}(\uz)=C(\uo)\left| \Delta(\uz,\uo) \right|^2 \exp \l( -\sum_{k=1}^{m}|\zeta_k|^2 \r)
\end{equation}
with respect to the Lebesgue measure on $\D^m$, where $C(\uo)$ is the normalizing constant (which depends on $\uo$).
Let $(\uz,\uo)$ and $(\uz',\uo)$ correspond to two configurations such that the event $\Omega_{n}^{m,\del}$ occurs in both cases. Then the ratio of the conditional densities at these two points is given by
\begin{equation}
 \label{condgin1}
\frac{\rho^n_{\uo}(\uz')}{\rho^n_{\uo}(\uz)}=\frac{|\Delta(\uz',\uo)|^2}{|\Delta(\uz,\uo)|^2} \frac{\exp(-\sum_{k=1}^{m}|\zeta_k^{'}|^2)}{\exp(-\sum_{k=1}^{m}|\zeta_k|^2)}\, .
\end{equation}

Our goal in this section is to obtain upper and lower bounds on the ratio of conditional densities $\frac{\rho^n_{\uo}(\uz')}{\rho^n_{\uo}(\uz)}$. We will take up, one after the other, the issues of bounding the quantities appearing in the expression (\ref{condgin}) for this ratio.

Clearly, ${\displaystyle \exp \l( -\sum_{k=1}^{m}|\zeta_k^{'}|^2 \r)} \bigg/ {\displaystyle \exp \l( -\sum_{k=1}^{m}|\zeta_k|^2 \r)}$ is bounded above and below by constants which are functions of $m$ and $\D$.

To study the ratio of the Vandermonde determinants, we define \[ \Gamma(\uz,\uo)=\prod_{1\le i \le m, 1 \le j \le n-m}(\zeta_i-\omega_j).\] Then we have
\begin{equation}
\label{condgin2}
\frac{|\Delta(\uz',\uo)|^2}{|\Delta(\uz,\uo)|^2} = \frac{|\Delta(\uz')|^2}{|\Delta(\uz)|^2}\frac{|\Gamma(\uz',\uo)|^2}{|\Gamma(\uz,\uo)|^2}.
\end{equation}
In order to bound $\frac{\displaystyle |\Gamma(\uz',\uo)|^2}{\displaystyle |\Gamma(\uz,\uo)|^2}$ from above and below  uniformly in $\uz,\uz' \in \D^m$, it suffices to bound $\frac{\displaystyle |\Gamma(\uz,\uo)|}{\displaystyle |\Gamma(\underline{0},\uo)|}$ from above and below uniformly in $\uz \in \D^m$. Here $\underline{0}$ is the vector of all 0-s in $\D^m$.
We observe that \[ \frac{|\Gamma(\uz,\uo)|}{|\Gamma(\underline{0},\uo)|}= \prod_{i=1}^m \l( \prod_{j=1}^{n-m} \l|\frac{\zeta_i- \omega_j}{\omega_j}\r| \r). \]
Since $m$ is fixed, it suffices to bound ${\displaystyle \prod_{j=1}^{n-m} \l|\frac{\zeta_0 - \omega_j}{\omega_j}\r|}$ from above and below uniformly in $\zeta_0 \in \D$. To this end we prove
\begin{proposition}
\label{gin1}
Let $\zeta_0 \in \D$ and $\omega_1,\cdots,\omega_{n-m}$ be such that $\mathrm{dist}(\omega_j,\D)>\delta$ for all $j$. Then we have
\[ \l| \log \l( \prod_{j=1}^{n-m} \l|\frac{\zeta_0 - \omega_j}{\omega_j}\r| \r) \r| \le K_1(\D)\l|\sum_{j=1}^{n-m}\frac{1}{\omega_j}\r|+K_2(\D)\l|\sum_{j=1}^{n-m} \frac{1}{\omega_j^2} \r|+ K_3(\D,\delta) \l(\sum_{j=1}^{n-m}\frac{1}{|\omega_j|^3} \r). \]
\end{proposition}
Here $K_1(\D)$, $K_2(\D)$ and $K_3(\D,\delta)$ are constants depending on $\D$ and $\delta$.
\begin{proof}
We begin with
$ \log \l| \frac{\zeta_0 -\omega_j}{\omega_j} \r|  = \log \l| 1-\frac{\zeta_0}{\omega_j}   \r|.$
Due to the $\delta$-separation between $\partial \D$ and $\omega$, the ratio $\theta_j := \frac{\zeta_0}{\omega_j}$ satisfies $|\theta_j| \le \frac{r}{r+\delta} < 1$. Let $\log$ be the branch of the complex logarithm given by the power series development $\log(1-z)=-\l( \sum_{k=1}^{\infty} \frac{z^k}{k} \r) $ for $|z|<1$.
Then we have $\log \l| 1-\theta_j   \r| =  \Re \log (1 - \theta_j) = - \Re (\theta_j) - \frac{1}{2}\Re (\theta_j^2) + h(\theta_j)  $
where $\l| h(\theta_j)\r| \le K_3'(\D,\delta)|\theta_j|^3$, where $K_3'$ is a constant depending on $\D$ and $\delta$.  Hence,
\[ \l| \log \l( \prod_{j=1}^{n-m} \l|\frac{\zeta_0- \omega_j}{\omega_j}\r| \r) \r|
 = \l| \sum_{j=1}^{n-m} \log \l(  \l|\frac{\zeta_0- \omega_j}{\omega_j}\r| \r) \r|
 = \l|-\Re \l( \sum_{j=1}^{n-m}\theta_j \r) - \frac{1}{2} \Re \l( \sum_{j=1}^{n-m} \theta_j^2 \r) + \sum_{j=1}^{n-m} h(\theta_j) \r|. \]
Recall that $\theta_j=\frac{\zeta_0}{\omega_j}$ and $|\zeta_0| \le r$ and $\l| h(\theta_j)\r| \le K_3'(\D,\delta)|\theta_j|^3$. The triangle inequality applied to the above gives us the statement of the Proposition with $K_1(\D)=r$, $K_2(\D)=\frac{1}{2}r^2$ and $K_3(\D,\delta)= K_3'(\D,\delta)r^3$.
\end{proof}
As a result, we have
\begin{proposition}
\label{gin2}
Conditionally on $\Omega_{n}^{m,\delta}$, there exists a constant $K(\D,\delta)>0$ such that
\[\exp \bigg(- 4mK(\D,\delta)\X_n \bigg)\frac{|\Delta(\uz')|^2}{|\Delta(\uz)|^2}\le \frac{\rho^n_{\uo}(\uz')}{\rho^n_{\uo}(\uz)}  \le \exp \bigg( 4mK(\D,\delta)\X_n \bigg) \frac{|\Delta(\uz')|^2}{|\Delta(\uz)|^2},\]
where ${\displaystyle \X_n = \l|\sum_{\omega \in \mathcal{G}_n \cap \D^c}\frac{1}{\omega} \r|+ \l|\sum_{\omega \in \mathcal{G}_n \cap \D^c}\frac{1}{\omega^2} \r|+ \l( \sum_{\omega \in \mathcal{G}_n \cap \D^c}\frac{1}{|\omega|^3} \r)}$ and $\E [\X_n]\le c_1(\D,m)<\infty .$
\end{proposition}
\begin{proof}
Clearly, it suffices to bound $ {|\Delta(\uz',\uo)|^2} \big/ {|\Delta(\uz,\uo)|^2}$ from above and below.\\
From Proposition \ref{gin1}, on $\Omega_{n}^{m,\delta}$ we have $\l| \log \l( \prod_{j=1}^{n-m} \l|\frac{\zeta_0 - \omega_j}{\omega_j}\r| \r) \r| \le  K(\D,\delta)\X_n$ for any $\zeta_0 \in \D$, where $K(\D,\delta)=\max \{K_1,K_2,K_3\}$ .
Considering this for each $\zeta_i,i=1,\cdots,m$, exponentiating and then taking product over $i=1,\cdots,m$, we get
\[ \exp \bigg(-mK(\D,\delta)\X_n\bigg) \le \frac{|\Gamma(\uz,\uo)|}{|\Gamma(\underline{0},\uo)|} \le \exp \bigg(mK(\D,\delta)\X_n\bigg). \]
The same estimate holds for $\zeta'$. Since $\frac{\displaystyle |\Gamma(\uz',\uo)|}{\displaystyle |\Gamma(\uz,\uo)|}= \frac{\displaystyle |\Gamma(\uz',\uo)|}{\displaystyle |\Gamma(\underline{0},\uo)|} \bigg/ \frac{\displaystyle |\Gamma(\uz,\uo)|}{\displaystyle |\Gamma(\underline{0},\omega)|}$, we have
\[  \exp \bigg(-2mK(\D,\delta)\X_n\bigg) \le \frac{|\Gamma(\uz',\uo)|}{|\Gamma(\uz,\uo)|} \le \exp \bigg(2mK(\D,\delta)\X_n\bigg).  \]
In view of (\ref{condgin1}) and (\ref{condgin2}), this leads to the desired bound $\frac{\rho^n_{\uo}(\uz')}{\rho^n_{\uo}(\uz)} $.
In Proposition \ref{invginest-tail}, we will see that each of the three random sums defining $\X_n$ has finite expectation. Moreover, those expectations are bounded uniformly (in $n$) by quantities depending only on $\D$ and $m$. This yields the statement $\E [\X_n]\le c_1(\D,m)<\infty$.
\end{proof}

\begin{corollary}
\label{ginc2.1}
Given  $M>0$, we can replace $\X_n$ in Proposition \ref{gin2} by a uniform bound $M$ except on an event of probability less than $ c_1(m,\D)/M$.
\end{corollary}

\subsection{Estimates for Inverse Powers}
\label{invpowersgin}
Our aim in this section is to estimate  the sums of inverse powers of the points in $\mathcal{G}$ and $\mathcal{G}_n$ outside a disk containing the origin.
To this end, we first discuss certain estimates on the variance of linear statistics, which are uniform in $n$. Let $B(0;r)$ denote the disk of radius $r$ centred at the origin.
\begin{proposition}
\label{est}
Let $\varphi$ be a compactly supported Lipschitz function, supported inside the disk $B(0;r)$ with Lipschitz constant $\kappa(\varphi)$. Let $\varphi_R(z):=\varphi(z/R)$. Then
$\mathrm{Var} \l(\int\varphi_R(z)\, d[\G_n](z)\r)\le C(\varphi)$, where $C(\varphi)$ is a quantity that is independent of $n$, and depends on $\varphi$ only through $\kappa(\varphi)$.
The same conclusion holds for $\mathcal{G}$ in place of $\mathcal{G}_n$.
\end{proposition}

To prove Proposition \ref{est}, we will make use of a general fact about determinantal point processes:
\begin{lemma}
\label{varlinst}
Let $\Pi$ be a determinantal point process with Hermitian kernel $K$. Let $K$ be a reproducing kernel with respect to its background measure $\gamma$, which means $K(x,y)= \int K(x,z)K(z,y)\, \mathrm{d}\gamma(z)$ for all $x,y$. Let $\varphi, \psi$ be  compactly supported continuous functions.  Then we have
\[\mathrm{Cov} \l[ \int \varphi \, d[\Pi], \int \psi \, d[\Pi] \r] = \frac{1}{2}\iint (\varphi(z) - \varphi (w)) \overline{(\psi(z)-\psi(w))}|K(z,w)|^2 \, \mathrm{d}\gamma(z) \, \mathrm{d}\gamma(w).\]
\end{lemma}
\begin{proof}

Denote by $\rho_1$ and $\rho_2$ respectively the one and two-point correlation functions of $\Pi$. Then we can write
\begin{align*} \mathrm{Cov} \l[ \int \varphi \, d[\Pi], \int \psi \, d[\Pi] \r] && = \int \varphi(x)\ol{\psi(y)} \rho_2(x,y) \mathrm{d}\gamma(x)\mathrm{d}\gamma(y) +  \int \varphi(x) \ol{ \psi(x)} \rho_1(x) \mathrm{d}\gamma(x) \\  && - {\l(\int \varphi(x) \rho_1(x) \mathrm{d}\gamma(x)\r)} \ol{\l(\int \psi(y) \rho_1(y) \mathrm{d}\gamma(y)\r)}. \end{align*}
But $\rho_1(x)=K(x,x)$ and $\rho_2(x,y)=K(x,x)K(y,y) - |K(x,y)|^2$. Using this, the expression for the variance in the display above reduces to 
\begin{equation} \label{vlst2}   \iint \varphi(x) \ol{\psi(x)} K(x,x) \mathrm{d}\gamma(x)  - \iint \varphi(x)\ol{\psi(y)} |K(x,y)|^2 \mathrm{d}\gamma(x)\mathrm{d}\gamma(y). \end{equation}
But since $K$ is the integral kernel corresponding to a projection operator and $K(y,x)=\ol{K(x,y)}$, we have $K(x,x)=\iint |K(x,y)|^2 \mathrm{d}\gamma(y)$. 
Using this, the expression for the variance in (\ref{vlst2}) reduces to 
\[ \iint \l( \varphi(x) \ol{\psi(x)} - \varphi(x)\ol{\psi(y)}  \r) |K(x,y)|^2 \mathrm{d}\gamma(x) \mathrm{d}\gamma(y)\]\[ =\frac{1}{2} \iint  \l( \varphi(x) - \varphi(y) \r) \ol{\l(\psi(x)-\psi(y)\r)} |K(x,y)|^2 \, \mathrm{d}\gamma(x) \mathrm{d}\gamma(y),  \] as desired. In the last step, we have used  symmetry in $x$ and $y$.
\end{proof}

\begin{proof} [\textbf{Proof of Proposition \ref{est}}]
We give the proof when $r=1$, from here the general case is obtained by scaling, because any function $\varphi$ supported on $B(0;r)$ is equal to the scaling $\Phi_r$ of some function $\Phi$ supported on $B(0;1)$ and having the same continuity and differentiability properties as $\varphi$. Notice that $(\Phi_r)_R=\Phi_{rR}$, so the result for $\varphi$ can be deduced from the result for $\Phi$. In what follows we deal with $\G_n$, the result for $\G$ follows, for instance, from taking limits as $n \to \infty$ for the result for $\G_n$.

Using Lemma \ref{varlinst}, we have \[\text{Var} \l(\int \varphi_R(z)\, d[\G_n](z) \r)= \frac{1}{2} \iint {|\varphi_R(z)-\varphi_R(w)|^2|K_n(z,w)|^2\, \mathrm{d}\gamma(z)\, \mathrm{d}\gamma(w)}\]
where $\gamma$ is the standard complex Gaussian measure.
Now, \[|\varphi_R(z)-\varphi_R(w)|^2 = |\varphi(z/R)-\varphi(w/R)|^2 \le \frac{1}{R^2}\kappa(\varphi)^2|z-w|^2. \]
Therefore, it suffices to bound the integral $\int_{A(R)} |z-w|^2 |K_n(z,w)|^2 \, \mathrm{d}\gamma(z) \, \mathrm{d}\gamma(w)$ on the set \[A(R):= \{ (z,w) : \min\{|z|, |w|\} \le R \} \] because outside $A(R)$, we have $\varphi_R(z)=\varphi_R(w)=0$.  We begin with 
\begin{align*} 
\int_{A(R)} |z-w|^2 |K_n(z,w)|^2 \, \mathrm{d}\gamma(z) \, \mathrm{d}\gamma(w) \le & \int_{A_1(R)} |z-w|^2 |K_n(z,w)|^2 \, \mathrm{d}\gamma(z) \, \mathrm{d}\gamma(w) \\ & + \int_{A_2(R)} |z-w|^2 |K_n(z,w)|^2 \, \mathrm{d}\gamma(z) \, \mathrm{d}\gamma(w)
\end{align*} 
where \[A_1(R)=\{ |z|\le 2R, |w| \le 2R \} \text{ and } A_2(R)=\{ |z|\le R, |w| \ge 2R \} \cup \{ |w|\le R, |z| \ge 2R \}. \]

We first address the case of $A_2(R)$. By symmetry, it suffices to bound the integral over the region $\{ |z|\le R, |w| \ge 2R \}$. In this region, $||z|-|w|| \ge R$, and $|K_n(z,w)|^2 e^{-|z|^2-|w|^2} \le e^{-||z|-|w||^2}$. The integral is bounded from above by \[\frac{1}{\pi^2}\int_{|z| \le R} \l( \int_{|w| \ge 2R} (|z|+|w|)^2 e^{-(|z|-|w|)^2}  \, \mathrm{d}\mathcal{L}(w)\r)  \, \mathrm{d}\mathcal{L}(z).\]
It is not hard to see that the inner integral is $O(R^2e^{-R^2})$ Integrating over $|z|\le R$ gives another factor of $R^2$, so the total contribution is $o(1)$ as $R \to \infty$.

For the integral over $A_1(R)$, we proceed as follows:
\begin{align*}
&\int_{A_1(R)} |z-w|^2 |K_n(z,w)|^2\, \mathrm{d}\gamma(z) \mathrm{d}\gamma(w)
\\&=\frac{1}{\pi^2}\int \bigg( |z|^2-z\bar{w}-\bar{z}w+|w|^2 \bigg) \l(\sum_{j=0}^{n-1}(z\bar{w})^j/j!\r) \l(\sum_{j=0}^n(\bar{z}w)^j/j!\r)e^{-|z|^2-|w|^2} \, \mathrm{d}\mathcal{L}(z) \, \mathrm{d}\mathcal{L}(w).
\end{align*}
Now, we integrate the $|z-w|^2$ part term by term. Due to radial symmetry, only some specific terms from $|K_n(z,w)|^2$ contribute. For example, when we integrate the $|z|^2$ term in $|z-w|^2$, only the ${\displaystyle \frac{(z\bar{w})^j}{j!}\frac{(\bar{z}w)^j}{j!}}$,$0\le j\le n-1$ terms in the expanded expression for $|K_n(z,w)|^2$ contribute. When we integrate $z\bar{w}$, only the ${\displaystyle \frac{(z\bar{w})^{j}}{j!}\frac{(\bar{z}w)^{j+1}}{(j+1)!}}$,$0\le j\le{n-2}$ terms provide non-zero contributions. Due to symmetry between $z$ and $w$, it is enough to bound the contribution from $(|z|^2 - z \overline{w})$ by  $O(R^2)$.

$\bullet$ \underline{$|z|^2$ term: }

Let us introduce the change of variables  by $x=|z|^2$ and $y=|w|^2$. Then the contribution in the above integral coming from \[ \frac{(z\bar{w})^j}{j!}\frac{(\bar{z}w)^j}{j!} \] can be written as \[ \int_0^{4R^2} \int_0^{4R^2}\frac{x^{j+1}}{j!}e^{-x}\frac{y^{j}}{j!}e^{-y}dxdy. \] So, the total contribution due to all such terms, ranging from $j=0,\cdots,n-1$ is \[ \sum_{j=0}^{n-1}  \l( \int_0^{4R^2} \frac{x^{j+1}}{j!}e^{-x}dx \r) \l( \int_0^{4R^2}\frac{y^{j}}{j!}e^{-y}dy \r). \]

$\bullet$
 \underline{$z\bar{w}$ term: } 

As above, the contribution coming from the \[ \frac{(z\bar{w})^{j}}{j!}\frac{(\bar{z}w)^{j+1}}{(j+1)!} \] is given by \[ \int_0^{4R^2} \int_0^{4R^2}\frac{x^{j+1}}{j!}e^{-x}\frac{y^{j+1}}{(j+1)!}e^{-y}dxdy.  \] Therefore the total contribution from $0\le j\le{n-2}$ is
\[\sum_{j=0}^{n-2} \l( \int_0^{4R^2} \frac{x^{j+1}}{j!}e^{-x}dx \r) \l( \int_0^{4R^2}  \frac{y^{j+1}}{(j+1)!}e^{-y}dy \r).\]

We interpret ${\displaystyle \frac{x^j}{j!}e^{-x}dx}$ as a gamma density, the corresponding random variable being denoted by $\Gamma_{j+1}$.

The contribution due to the $|z|^2$ term is ${ \displaystyle \sum_{j=0}^{n-1} \E[\Gamma_{j+1}{1\hskip-4pt{\rm 1}}_{(\Gamma_{j+1}\le 4R^2)}]\P[\Gamma_{j+1}\le 4R^2]}$ and that due to the $z\bar{w}$ term is
${ \displaystyle \sum_{j=0}^{n-2} \E[\Gamma_{j+1}{1\hskip-4pt{\rm 1}}_{(\Gamma_{j+1}\le 4R^2)}]\P[\Gamma_{j+2}\le 4R^2] }$.

The difference between the above two terms can be written as:
\begin{equation}
 \label{qty}
 \E[\Gamma_{n}{1\hskip-4pt{\rm 1}}_{(\Gamma_{n}\le 4R^2)}]\P[\Gamma_{n}\le 4R^2]+\sum_{j=1}^{n-1}\E[\Gamma_j{1\hskip-4pt{\rm 1}}{(\Gamma_j\le 4R^2)}]\bigg(\P[\Gamma_j\le 4R^2]-\P[\Gamma_{j+1}\le 4R^2] \bigg).
\end{equation}
All the expectations in the above are $\le 4R^2$, and $\P[\Gamma_j\le 4R^2] \ge \P[\Gamma_{j+1}\le 4R^2]$ because $\Gamma_{j+1}$ stochastically dominates $\Gamma_j$ (this can be seen, e.g., by considering a sequence $\{X_i\}_{i=1}^{\infty}$ of i.i.d. exponentials with mean 1 and using a natural coupling of the $\Gamma_j$-s by setting $\Gamma_j=X_1+\cdots+X_j$). Therefore the absolute value of \eqref{qty}, by a telescopic sum, is $\le 4R^2 \P[\Gamma_1\le 4R^2]\le 4R^2$.
Combining all of these, we see that  $\text{Var} \l(\int \varphi_R(z)\, d[\G_n](z)\r)$ is bounded by  ${ \displaystyle \l(1/R^2 \r)\kappa(\varphi)^2 8R^2=C(\varphi)}$. It is clear that $C(\varphi)$ depends on $\varphi$ only through $\kappa(\varphi)$.
\end{proof}

Let $r_0$ be the radius of $\D$.   Let $\varphi$ be a non-negative radial $C_c^{\infty}$ function supported on complex numbers with radius in $[r_0,3r_0]$ such that $\varphi=1$ on $[\frac{3r_0}{2},2r_0]$ and $\varphi(r_0 + r)=1-\varphi(2r_0+2r)$, for $0\le r \le \frac{1}{2}r_0$. In other words, $\varphi$ is a test function supported on the annulus between $r_0$ and $3r_0$ and its ``ascent'' to 1 is twice as fast as its ``descent''. Let $\widetilde{\varphi}$ be another radial test function with the same support as $\varphi$, satisfying $\widetilde{\varphi}(r_0+xr_0)=1$ for $ 0 \le  x\le \frac{1}{2}$ and $\widetilde{\varphi}=\varphi$ otherwise. Recall that for a test function $\psi$ and $L>0$ we denote by $\psi_L$ the scaled function $\psi_L(z)=\psi(z/L)$. 

Observe that the functions $\widetilde{\varphi}$ and $\varphi_{2^j}$ for $ j\ge 1$ form a partition of unity on $\D^c$.

For the rest of the paper, we introduce the following notation for the dilations of a set:
\begin{notation}
For a  set $\Lambda \subset \C$ and $R>0$, we denote $R \cdot \Lambda = \{Rz: z\in \Lambda \}$.
\end{notation}

\begin{proposition}
\label{invginest-tail}
Let $r_0$ be the radius of $\D$ . Let $\varphi$ and $\wt{\varphi}$ be defined as above.   

$\mathrm{(i)}$ The random variables \[ S_l(n):= \int \frac{\widetilde{\varphi}({z})}{z^l} \, d[\G_n](z) + \sum_{j=1}^{\infty} \int \frac{\varphi_{2^j}({z})}{z^l} \, d[\G_n](z) = \sum_{\o \in \G_n \cap \D^c} \frac{1}{\o^l} \quad (\text{ for }l \ge 1) \] and \[
 \widetilde{S}_l(n):= \int \frac{\widetilde{\varphi}({z})}{|z|^l} \, d[\G_n](z) + \sum_{j=1}^{\infty} \int \frac{\varphi_{2^j}({z})}{|z|^l} \, d[\G_n](z) = \sum_{\o \in \G_n \cap \D^c} \frac{1}{|\o|^l} \quad (\text{ for }l \ge 3)\]
 have finite first moments which, for every fixed $l$, are bounded above by quantities which do not depend on $n$. 

$\mathrm{(ii)}$ There exists $k_0=k_0(\varphi)\ge 1$,   not depending on $n$ or $l$, such that for $k \ge k_0$ the ``tails'' of $S_l(n)$ and $\widetilde{S}_l(n)$ beyond the disk $2^k \cdot \D $, given by \[\tau_l^{n}(2^k):= \sum_{j=k}^{\infty} \int \frac{\varphi_{2^j}({z})}{z^l} \, d[\G_n](z) \, \,(\text{ for }l \ge 1) \quad \text{ and } \quad \widetilde{\tau}_l^{n}(2^k):= \sum_{j=k}^{\infty} \int \frac{\varphi_{2^j}({z})}{|z|^l} \, d[\G_n](z) \, \,(\text{ for }l \ge 3)\] satisfy the estimates \[ \E \l[ \l| \tau_l^{n}(2^k) \r|\r]\le C_1(\varphi,l)/2^{kl} \quad \text{ and } \quad  \E \l[ \l| \widetilde{\tau}_l^{n}(2^k) \r|\r]\le C_2(\varphi,l)/2^{k(l-2)}. \]

All of the above remain true when $\G_n$ is replaced by $\G$, for which we use the notations ${S}_l$ , $\widetilde{S}_l$, $\tau_l(R)$ and $\widetilde{\tau}_l(R)$ to denote the quantities corresponding to $S_l(n)$, $\widetilde{S}_l(n)$, ${\tau}_l^{n}(R)$ and $\widetilde{\tau}_l^{n}(R)$ respectively.
 \end{proposition}

\begin{remark}
For $\G$, by the sum ${\displaystyle  \l( \sum_{\o \in \G \cap \D^c}\frac{1}{\o^l} \r)}$ we denote the quantity \[S_l=\int \frac{\widetilde{\varphi}({z})}{z^l} \, d[\G](z) + \sum_{j=1}^{\infty} \int \frac{\varphi_{2^j}({z})}{z^l} \, d[\G](z)\] due to the obvious analogy with $\G_n$, where the corresponding  sum $S_l(n)$ is indeed equal to ${\displaystyle  \l( \sum_{\o \in \G_n \cap \D^c}\frac{1}{\o^l} \r)}$ with its usual meaning.
\end{remark}

\begin{proof}
Recall that the functions $\widetilde{\varphi}$ and $\varphi_{2^j}$ for $ j\ge 1$ form a partition of unity on $\D^c$, hence we have the identities appearing in part (i). 

Set $\psi_k^n(l)= \int \frac{\varphi_{2^k}({z})}{z^l} \, d[\G_n](z)$ for $k \ge 1$, and $\psi_0^n(l)=\int \frac{\widetilde{\varphi}({z})}{z^l} \, d[\G_n](z)$.
When $l \ge 3$ we also define $\gamma_k^n(l)= \int \frac{\varphi_{2^k}({z})}{|z|^l} \, d[\G_n](z)$ for $k \ge 1$, and $\gamma_0^n(l)=\int \frac{\widetilde{\varphi}({z})}{|z|^l} \, d[\G_n](z)$. Let ${\Psi}_k(l)$ and ${\Gamma}_k(l)$ denote the analogous quantities defined with respect to $\G$ instead of $\G_n$.

We begin with the observation that for $k \ge 1$ we have $\E[\psi_k^n(l)]=0$ (owing to the rotational symmetry of the distributions $\G_n$-s). This implies that \[\E[|\psi_k^n(l)|]\le \l(\E[|\psi_k^n(l)|^2]\r)^{1/2} =\sqrt{\mathrm{Var}[\psi_k^n(l)]}.\]
We then apply Proposition \ref{est} to the function $\varphi(z)/z^l$ and $R=2^k$ to obtain $\E[|\psi_k^n(l)|]\le C(\varphi,l)/(2^k)^l$. We also note that \[ \E[|\psi_0^n(l)|] \le \frac{1}{\pi}\int_{3\cdot \D \setminus \D} \frac{K_n(z,z)e^{-|z|^2}}{|z|^l}\,\mathrm{d}\el(z) \le \frac{1}{\pi} \int_{3\cdot \D \setminus \D} \frac{1}{|z|^l}\mathrm{d}\el(z)=c(l) .\]
This implies that for $l \ge 1$ \[\E[|S_l(n)|]\le \sum_{k=0}^{\infty}\E[|\psi_k^n(l)|]<\infty .\]
The desired bound for  $\E[\widetilde{S}_l(n)]$ follows from a direct computation of the expectation using the first intensity, and noting that the first intensity of $\G_n$ (with respect to Lebesgue measure) is $K_n(z,z)e^{-|z|^2}\le 1$ for all $z$. 

The estimates for $\tau$ and $\widetilde{\tau}$ follow by using the above argument for the sums $\sum_{j=k}^{\infty}\psi_j^n(l)$ and $\sum_{j=k}^{\infty}\gamma_j^n(l)$.

%

\end{proof}

\begin{corollary}
\label{invginest-tailcor}
For $R=2^k$ for  $k \ge k_0$ (as in Proposition \ref{invginest-tail}), we have $\P[|\tau_{ \, l}^{{n}}(R)|>R^{-l/2}]\le c_1(\varphi,l)R^{-l/2}$ and $\P[|\widetilde{\tau}_{\, l}^{n}(R)|>R^{-(l-2)/2}]\le c_2(\varphi,l)R^{-(l-2)/2}$, and these estimates remain true when $\G_n$ is replaced with $\G$.
\end{corollary}

\begin{proof}
We use the estimates on the expectation of $|\tau_l^{n}(R)|$ and  $|\widetilde{\tau}_l^{n}(R)|$ from Proposition \ref{invginest-tail} and apply Markov's inequality.
\end{proof}

With notations as above, we have
\begin{proposition}
\label{conv}
For each $l \ge 1$ we have  $S_l(n) \rightarrow S_l$ in probability, and for each $l \ge 3$ we have $\widetilde{S}_l(n) \rightarrow \widetilde{S}_l$ in probability, and hence  we have such convergence a.s.  along some subsequence, simultaneously for all $l$. 
\end{proposition}

\begin{proof}
Fix $\delta>0$. Given $\eps>0$, we choose $R=2^k$ large enough such that $c_1R^{-l/2}<\eps/4$ (for $l \ge 1$) and $c_2R^{-(l-2)/2}<\eps/4$ for $l \ge 3$ (as in Corollary \ref{invginest-tailcor}), as well as $R^{-l/2}<\del/4$ for $l=1,2$ and $R^{-(l-2)/2}<\del/4$ for $l \ge 3$. We now recall notations from the beginning of the proof of Proposition \ref{invginest-tail}.  By definition, we have $S_l(n)=\sum_{j=0}^k \psi_{j}^n(l) + \tau_l^n(2^k)$. On the disk of radius $R$, we have $\G_n \to \G$ a.s.  Now choose $n$ large enough so that we have  $|\sum_{j=0}^k \psi_j^n(l) - \sum_{j=0}^k \Psi_j(l) |<\del/3 $ except on an event of probability $\eps/2$. By choice of $R$, we have $|\tau_l^n(2^k)|< \del/3$ and $|\tau_l(2^k)|<\del/3$ except on an event of probability $<\eps/2$.  Combining all these, we have $\P(|S_l(n) - S_l|>\del) \le \eps$, proving that $S_l(n) \to S_l$ in probability.

For each $l$, given any sequence we can find a subsequence along which this convergence is a.s. A diagonal argument now gives us a subsequence for which a.s. convergence holds simultaneously for all $l$. 

When $l \ge 3$, the argument for $\widetilde{S}_l$ is similar.   
\end{proof}

Define $S_l(\D,n)=\sum_{z \in \G_n \cap \D}1/z^l$ and $S_l(\D)=\sum_{z \in \G \cap \D} 1/z^l$. Set $\a_l(n)=S_l(\D,n)+S_l(n)$ and $\a_l=S_l(\D)+S_l$.  Observe that $\a_l(n)=\sum_{z \in \G_n} 1/z^l$. Notice that a.s. there is no point at the origin, so each of these random variables is well defined. With these definitions, we have:

\begin{proposition}
 \label{conv-gin1}
For each $l$, $\a_l(n) \to \a_l$ in probability as $n \to \infty$. Hence, there is a subsequence such that $\a_l(n) \to \a_l$ a.s. when $n \to \infty$ along this subsequence, simultaneously for all $l$. 
\end{proposition}

\begin{proof}
Since the finite point configurations given by $\G_n|_{ \D} \to \G|_{\D}$ a.s. and a.s. there is no point at the origin, therefore $S_l(\D,n) \to S_l(\D)$ a.s.  
This, combined with Proposition \ref{conv}, gives us the desired result. The last assertion follows essentially from the fact that  the convergence $\G_n|_{ \D} \to \G|_{\D}$ (for a fixed realization of the random variables in question) is weak convergence of counting measures. Since the point configurations in a given realization of the $G_n$-s are bounded away from 0 (by an amount which depends on the realization), therefore the convergence of $S_l$-s can be obtained by integrating the counting measures against the function $\frac{1}{z^l}$ (mollified near the origin and away from the point configurations to make it a continuous function).
\end{proof}

\section{Limiting Procedure for the Ginibre Ensemble}
\label{limgin}

The aim of this section is to use the estimates derived in Section \ref{gintolerance} to verify the conditions laid out in Theorem \ref{abs}, so that the limiting procedure outlined in Section \ref{limcond} can be executed. This will lead us to a proof of Theorem \ref{gin-2} for a disk $\D$. We have already seen in Section \ref{general} that this implies Theorem \ref{gin-2} for general $\D$.

\begin{proof}  [\textbf{Proof of Theorem \ref{gin-2} for a disk}] 

We will appeal to Theorem \ref{abs} with $\D$  an open disk. We already know from Theorem \ref{gin-1} that the number of points in  $\D$ is rigid. In terms of the notation used in Section \ref{limcond}, we set $X=\G$ and $X^n=\G_n$. 

Fix an integer $m \ge 0$. Consider the event that $m$ points of $\G$ are inside $\D$. The result in Theorem \ref{gin-2} is trivial for $m=0$. Hence, we can assume that $m >0$. 
For $\o \in \S_{\out}$, our candidate for $\nu({\o},\cdot)$ (refer to Theorem \ref{abs}) is the probability measure $Z^{-1}|\Delta(\uz)|^2\mathrm{d}\el(\uz)$ on $\D^m$, where $Z$ is the normalizing constant. Heuristically, this is in tune with the fact that the conditional measure is supported on $\D^m$ and with the estimate in Proposition \ref{gin2} on conditional densities. It is also a mathematical expression of the notion that $\o$ determines only the number of points in $\D$, and ``nothing more''.  Notice that $\nu$ is a constant when considered as a function mapping $\S_{\out}$ to $\mathcal{M}(\D^m)$, in the sense that $\nu$ does not depend on the outside configuration. Since a.s. ${\nu}(\xout,\cdot)$ is mutually absolutely continuous with respect to the Lebesgue measure on $\D^m$, Theorem \ref{abs} would imply that the same holds true for the conditional distribution $\rho(\xout,\cdot)$ of the points inside $\D$ (treated as a vector in the uniform random order). 


Now we construct, for each $j$, the sequences $\{n_k(j)\}_{k \ge 1}$ and the events $\Omega_{n_k}(j)$. We proceed as follows. From Proposition \ref{conv}, we get a subsequence $n_k$ such that a.s. $S_1({n_k}) \to S_1$, $S_2({n_k}) \to S_2$ and $\widetilde{S}_3({n_k}) \to \widetilde{S}_3$. This is going to be our subsequence $n_k(j)$ for all $j$. 

Let $M_j \uparrow \infty$ be a sequence of positive numbers such that none of them is an atom of the distributions of $|S_1|,|S_2|$ and $|\widetilde{S}_3|$ (in case such atom at all exists).

We will first \textbf{define the events} $\Omega_{n_k}(j)$ by the following conditions:
\begin{itemize}
\item[] (i) There are exactly $m$ points of $\mathcal{G}_{n_k}$  inside $\D$. 
\item[] (ii) $\mathrm{Dist}(\D,\mathcal{G}_{n_k} \cap \D^c) \ge 1/M_j$ (so $\Omega_{n_k}(j) \subset \Omega_{n_k}^{m,\del}$ with $\del=1/M_j$, defined at the beginning of section \ref{ratcondgin}).
\item[] (iii) $|S_1(n_k)|<M_j,|S_2(n_k)|<M_j,|\widetilde{S}_3(n_k)|<M_j$.
\end{itemize}

Recall that for an event $\mathcal{A}$ and a random variable $U$ on the same probability space, we say that $\mathcal{A}$ is measurable with respect to $U$ if $\mathcal{A}$ is measurable with respect to the sigma-algebra generated by $U$. Clearly, each $\Omega_{n_k}(j)$ is measurable with respect to $(\G_{n_k})_{\out}$. On the event $\Omega^m_{n_k}$ (which entails that $\l|({\G_{n_k}})_{\inn}\r|=m$; refer Section \ref{limcond}), the points in $(\G_{n_k})_{\out}$, considered in uniform random order, yield a vector $\uo$ in $(\D^c)^{n_k-m}$. Denoting the conditional distribution of $\uz$ given $\uo$ to be $\rho^{n_k}_{\uo}(\uz)$ we recall from Proposition \ref{gin2} that on the event $\Omega^{m,\del}_{n_k}$ we have
\[\exp \bigg(- 4mK(\D,\delta)\X_{n_k} \bigg)\frac{|\Delta(\uz')|^2}{|\Delta(\uz)|^2}\le \frac{\rho^{n_k}_{\uo}(\uz')}{\rho^{n_k}_{\uo}(\uz)}  \le \exp \bigg( 4mK(\D,\delta)\X_{n_k} \bigg) \frac{|\Delta(\uz')|^2}{|\Delta(\uz)|^2}.\]
Therefore, the bounds on $S_1({n_k}),S_2({n_k})$ and $\widetilde{S}_3({n_k})$ as in condition 
(iii) above imply that on $\Omega_{n_k}(j)$ we have, with $\del=1/M_j,$ \[\exp \bigg(- 12mK(\D,\delta)M_j \bigg)\frac{|\Delta(\uz')|^2}{|\Delta(\uz)|^2}\le \frac{\rho^{n_k}_{\uo}(\uz')}{\rho^{n_k}_{\uo}(\uz)}  \le \exp \bigg( 12mK(\D,\delta)M_j \bigg) \frac{|\Delta(\uz')|^2}{|\Delta(\uz)|^2} \hspace{3 pt}.\]   
Let $A \in \mathfrak{A}^m$ and $B \in \B$. Considering the last inequality on the right hand side, and
cross multiplying, we get \[ \rho^{n_k}_{\uo}(\uz') |\Delta(\uz)|^2 \le  \rho^{n_k}_{\uo}(\uz) \exp \bigg( 12mK(\D,\delta)M_j \bigg) |\Delta(\uz')|^2 .  \]
Observe that $\Omega_{n_k}(j)$ is measurable with respect to $(\G_{n_k})_{\out}$, hence $\uo$ determines whether the above inequality holds. In particular, it holds for all $\uz,\uz'$ for any $\uo$ is such that $\Omega_{n_k}(j)$ occurs.
We now integrate the above inequalities, first with respect to the Lebesgue measure in the variable $\uz' \in A$, yielding
\[ |\Delta(\uz)|^2  \int_A  \rho^{n_k}_{\uo}(\uz') \mathrm{d}\mathcal{L}(\uz') \le \exp \bigg( 12mK(\D,\delta)M_j \bigg) \rho^{n_k}_{\uo}(\uz)  \int_{A}  |\Delta(\uz')|^2  \mathrm{d}\mathcal{L}(\uz').  \]
Then we integrate with respect to the Lebesgue measure in the variable $\uz \in \D^m$ to obtain 
\[  \int_{\D^m} |\Delta(\uz)|^2 \mathrm{d}\mathcal{L}(\uz) \int_A  \rho^{n_k}_{\uo}(\uz') \mathrm{d}\mathcal{L}(\uz')  \le \exp \bigg( 12mK(\D,\delta)M_j \bigg) \int_{A}   |\Delta(\uz')|^2  \mathrm{d}\mathcal{L}(\uz'). \]
Finally, recalling again that $\Omega_{n_k}(j)$ is measurable with respect to $(\G_{n_k})_{\out}$, we integrate with respect to the distribution of $\uo$ on the event \[\{\uo \in (\D^c)^{n_k-m}:(\G_{\out}^{n_k})^{-1} \l( \{ \o_1,\cdots,\o_{n_k-m} \} \r) \in  (\G_{\out}^{n_k})^{-1} ( B ) \cap \Omega_{n_k}(j)\}\] to obtain
\[  \l( \int_{\D^m} |\Delta(\uz)|^2 \mathrm{d}\mathcal{L}(\uz) \r) \tP [  (\G_{\mathrm{in}}^{n_k} \in A) \cap  ( \G_{\out}^{n_k} \in B) \cap \Omega_{n_k}(j) ] \]\[\le \exp \bigg( 12mK(\D,\delta)M_j \bigg) \int_{(\G_{\out}^{n_k})^{-1}(B) \cap \Omega_{n_k}(j)  } \l( \int_{A}   |\Delta(\uz')|^2  \mathrm{d}\mathcal{L}(\uz') \r) \mathrm{d}\tP(\uo). \]
We can carry out the same procedure with the inequality on the left hand side. Combined with our definition of $\nu$, it can be seen that together they give us condition (\ref{abscond}) in Theorem \ref{abs}.

All that remains now is to show that events $\Omega(j):= \varliminf_{k \to \infty}\Omega_{n_k}(j)$ exhausts $\Omega^m$.

Since the $M_j$-s are increasing, we automatically have $\Omega(j) \subset \Omega(j+1)$. Moreover, recall that $\Omega^m$ is the event that there are $m$ points of $\G$ inside $\D$. Each $\Omega_{n_k}(j)$ satisfies this condition for $\G_{n_k}$. And finally, $\G_{n_k} \to \G$ a.s. on $\D$. These three facts together imply that $\Omega(j) \subset \Omega^m$ for each $j$. It only remains to check that  $\P(\Omega^m \setminus \Omega(j)) \to 0$ as $j \to \infty$. This is the goal of the next proposition, which will complete the proof of Theorem \ref{gin-2} for a disk.  
\end{proof}

\begin{proposition}
\label{sigma}
Let $\O(j)$ be as defined above. Then there exists $\O^{\mathrm{corr}}(j)$ measurable with respect to $\G_{\out}$ such that $\P(\O(j) \Delta \O^{\mathrm{corr}}(j))=0$. Further, $\P(\Omega^m \setminus \O(j)) \to 0$ as $j \to \infty$.
\end{proposition}

\begin{proof}
We first show the existence of $\O^{\mathrm{corr}}(j)$ with the desired properties. Accordingly, until the beginning of the last paragraph of this proof, $j$ is fixed; in particular $j$ does not depend on the parameter $\eps$ to be introduced below. 

Let $\eps >0$. We will construct an event $A_{\eps}$ (which depends on $j$) such that $A_{\eps}$ is measurable with respect to $\G_{\out}$ and there is a bad set $\Omega_{\mathrm{bad}}^{\eps}$ (which also depends on $j$) of  probability $<\eps$ such that $\O(j) \setminus \Omega_{\mathrm{bad}}^{\eps} = A_{\eps} \setminus \Omega_{\mathrm{bad}}^{\eps}$. As a result, $\P(\O(j) \Delta A_{\eps})<\eps$.
Then $\O^{\mathrm{corr}}(j)=\varliminf_{n \to \infty}{A_{2^{-n}}}$ will give us the desired event. It is easy to check that $\P(\Omega(j) \Delta \Omega^{\mathrm{corr}}(j))=0$. Let $\delta=\delta(\eps)<1$ be a small number, depending on $\eps$, to be chosen later. 

We \textbf{define the event} $A_{\eps}$ by the following conditions: 
\begin{itemize}
\item[] (i) $N(\G_{\out})=m$, where $N$ is as in Theorem \ref{gin-1}.
\item[] (ii) There is a $\frac{1}{M_j}$ separation between $\partial \D$ and the points $\o$ of $\mathcal{G}$ outside $\D$.
\item[] (iii) $|S_1|< M_j-\delta, |S_2|< M_j-\delta,|\widetilde{S}_3|< M_j-\delta$.
\end{itemize}
It is clear from the definition of $A_{\eps}$ that it is measurable with respect to $\G_{\out}$.

We can compare the definition of the event $A_\eps$ with a similar description for the event $\O(j)$. Since $\O(j)$ is the liminf of the events $\onk$ (as $k \to \infty$),   $\O(j)$ occurs if and only if each of the conditions (i), (ii) and (iii) in the definition of $\onk$ are true for $\G_{n_k}$ for all large enough $k$. The reader may notice the close similarity of the defining conditions of $\onk$ with those of $A_\eps$. In order to show that $\O(j) \setminus \Omega_{\mathrm{bad}}^{\eps}= A_{\eps} \setminus \Omega_{\mathrm{bad}}^{\eps}$, we need to establish that on $(\Omega_{\mathrm{bad}}^{\eps})^c$, condition (i) 
in the definition of $\onk$ being true for all large enough $k$ is equivalent to condition (i) in the definition of $A_\eps$ being true, and similarly with conditions (ii) and (iii) respectively in the two definitions.

By Proposition \ref{conv}, $S_i({n_k}) \to S_i, i=1,2$ and $\widetilde{S}_3({n_k}) \rightarrow \widetilde{S}_3$ a.s. along our chosen subsequence $\{n_k\}_k$.  By Egorov's Theorem, there is a bad event $\Omega^1$ of probability $<\eps/4$ such that outside $\Omega^1$, this convergence is uniform. Therefore, on $(\Omega^1)^c$ there is a large $k_0$ such that $|S_i({n_k})-S_i|< \del, i=1,2$ and $|\wt{S}_3({n_k})-\wt{S}_3|< \del$ for all $k \ge k_0$. 

By the a.s. convergence of $\mathcal{G}_{n_k}$-s to $\mathcal{G}$ on the compact set $2 \cdot \overline{\D}$, there is a large $k_1$ and a small event $\Omega^{2}$ of probability $< \eps/4$ outside which condition (i) in the definition of $\Omega_{n_k}(j)$ are simultaneously true or simultaneously false (simultaneity in $k$) for all  $\mathcal{G}_{n_k},k\ge k_1$. By the coupling of $\G_{n_k}$-s and $\G$ (which entails that $\G_n \subset \G_{n+1} \subset \G $ for all $n$), condition (i) in the definition of $\Omega_{n_k}(j)$ being simultaneously true for $\mathcal{G}_{n_k},k\ge k_1$ would imply that the condition (i) in the definition of $A_{\eps}$ is also true. The same statement holds with ``true'' replaced by ``false''. Thus, on the event $(\Omega^{2})^c$, exactly one of the following two possibilities hold : \newline 
(a) condition (i) defining $\onk$ is true for all $k \ge k_1$ and condition (i) defining $A_\eps$ is true,\newline
(b) condition (i) defining $\onk$ is false for all $k \ge k_1$ and condition (i) defining $A_\eps$ is false.

Similarly, the a.s. convergence of $\mathcal{G}_{n_k}$-s to $\mathcal{G}$ on the compact set $2 \cdot \overline{\D}$ guarantees that there is a large $k_2$ and a small event $\Omega^3$ of probability $<\eps/4$ such that on $(\Omega^4)^c$, exactly one of the two possibilities hold: \newline 
(a) condition (ii) defining $\onk$ is true for all $k \ge k_1$ and condition (ii) defining $A_\eps$ is true,\newline
(b) condition (ii) defining $\onk$ is false for all $k \ge k_1$ and condition (ii) defining $A_\eps$ is false.

Since $M_j$ is  not an atom of the distribution of $|S_1|,|S_2|$ or $|\widetilde{S}_3|$ , there is an event $\Omega^4$ (of probability $<\eps/4$) outside which each $|S_i|$ or $|\widetilde{S}_3|$ is either $>M_j+2\delta$ or $<M_j-2\delta$ ($\delta<1$ is chosen based on $\eps$ so that this condition is satisfied).

Define $\Omega_{\mathrm{bad}}^{\eps}= \Omega^1 \cup \Omega^2 \cup \Omega^3 \cup \Omega^4$; clearly $\P(\Omega_{\mathrm{bad}}^{\eps})<\eps$.
We note that on $(\Omega_{\mathrm{bad}}^{\eps})^c$, the condition (i)  in the definition of $\onk$ being true for all $k$ large enough is equivalent to condition (i) in the definition of $A_\eps$ being true. Similarly, on $(\Omega_{\mathrm{bad}}^{\eps})^c$, the condition (ii)  in the definition of $\onk$ being true for all $k$ large enough is equivalent to condition (ii) in the definition of $A_\eps$ being true. 

In order to show  that $\O(j) \setminus \Omega_{\mathrm{bad}}^{\eps}= A_{\eps} \setminus \Omega_{\mathrm{bad}}^{\eps}$ we are left to verify the only one remaining point: on $(\Omega_{\mathrm{bad}}^{\eps})^c$, the condition (iii) in the definition of $\onk$ being true for all large enough $k$ is equivalent to the condition (iii) in the definition of $A_\eps$ being true. 

Suppose the event $\O(j) \setminus \Omega_{\mathrm{bad}}^{\eps}$ occurs. Then we have $|S_i({n_k})|<M_j$ (with $i=1,2$) and $|\wt{S}_3(n_k)|<M_j$, for all large enough $k$. Hence $|S_i(n_k)-S_i|<\del$ implies $|S_i|<M_j+\del$ for $i=1,2$, similarly $|\wt{S}_3(n_k)-\wt{S}_3|<\del$ implies $|\wt{S}_3|<M_j+\del$. But we are on $(\Omega^4)^c$, so for $i=1,2$, we have $|S_i| \in (M_j-2\del,M_j+2\del)^c$, hence $|S_i|<M_j+\del$ implies $|S_i|<M_j-2\del$. Similarly, we get $|\wt{S}_3|<M_j-2\del$. Hence,  the event $A_{\eps}$ also occurs. 

Conversely, suppose the event$A_{\eps} \setminus \Omega_{\mathrm{bad}}^{\eps}$ occurs. Then for $i=1,2$ we have $|S_i|< M_j-\del$. But $|S_i({n_k})-S_i|<\del$ for all  large enough $k$, because we are on $(\O^1)^c$. This implies $|S_i({n_k})|<M_j$ for all large enough $k$. Similarly, $|\wt{S}_3(n_k)|<M_j$ for all large enough $k$. These facts together  mean that the event $\O(j)$ also occurs.

Hence $\O(j) \setminus \Omega_{\mathrm{bad}}^{\eps}= A_{\eps} \setminus \Omega_{\mathrm{bad}}^{\eps}$. As a result, we have $\l(\O(j) \Delta A_{\eps}\r) \subset \Omega_{\mathrm{bad}}^{\eps}$, and $\P(\O(j) \Delta A_{\eps})<\eps$.

To show that $\P(\O^m \setminus \O(j)) \to 0$ as $j \to \infty$, we define events $A'(j)\subset A_{\eps}$ above by replacing $\delta$ by 1 in the condition (iii) in the definition of $A_{\eps}$. Clearly, $A'(j) \subset A_{\eps}$ for each $0<\eps<1$, and therefore $ A'(j) \subset \varliminf_{n \to \infty}A_{2^{-n}} =\Omega(j)^{\mathrm{corr}}$. It is also clear that $\P(\O^m \setminus A'(j)) \to 0$ as $j \to \infty$, because $S_1,S_2$ and $\widetilde{S}_3$ are well defined random variables (with no mass at $\infty$) and the separation of $1/M_j$ between $D$ and the points of $\G_{\out}$ converges to 0. These two facts imply that $\P(\O^m \setminus \O(j)) \to 0$ as $j \to \infty$.

\end{proof}
This completes the proof of Theorem \ref{gin-2} in the case of $\D$ being a disk.

\section{Tolerance of the GAF zeros for a disk }
\label{tolgafz}

In this section we obtain the estimates necessary to prove Theorem \ref{gaf-2} in the case where $\D$ is a disk. By translation invariance of $\F$, we can take $\D$ to be centred at the origin. 

A key goal in this section is to quantitatively compare the conditional densities of the (vector of) inside zeroes (for finite degree polynomials) at two possible locations, along the lines of Proposition \ref{unifgaf}. Let $(\uz,\uo)$ be a randomly chosen vector of points corresponding to the ensemble $\mathcal{Z}_n$, with $\uz$ corresponding to the zeroes inside $\D$ and $\uo$ corresponding to those outside $\D$. Let $(\uz',\uo)$ be another possible vector of zeroes such that $\sum_i \z'_i = \sum_i \z_i$. We want to compare the conditional densities at these two locations. In Subsection \ref{polapprox} we introduce the conditional densities $\rho^n_{\uo,s}(\uz)$. We observe that the ratio $\rho^n_{\uo,s}(\uz')/\rho^n_{\uo,s}(\uz)$ of conditional densities (at two different locations $\uz$ and $\uz'$) splits neatly into a ratio of Vandermonde determinants and a ratio  of an explicit function of the elementary symmetric polynomials of the zeroes, as in (\ref{condgaf1}).  In Subsection \ref{ratvand}, we deal with the ratio of the Vandermonde determinants, on similar lines to the Ginibre case. In Subsection \ref{EIPZ}, we prove estimates on inverse power sums of the zeroes, which are necessary to complete the study of the Vandermonde ratio.

The procedure to estimate the ratio involving the elementary symmetric functions is more elaborate. As is evident from (\ref{condgaf1}), it is of relevance to compare the quantities $\s_k(\uz',\uo)$ and $\s_k(\uz,\uo)$, where $\s_k$ denotes the elementary symmetric polynomial of order $k$. Since the vectors of zeroes differ only in the co-ordinates corresponding to the zeroes inside $\D$, we begin with an equation that splits the difference between $\s_k(\uz',\uo)$ and $\s_k(\uz,\uo)$  into  quantities involving only the inside zeroes and quantities involving only the outside zeroes; see (\ref{sig2}). 

Since we need to estimate the ratio of certain functions of the $\s_k$-s, the next step is to apply the above technique to express these quantities  in terms of $\s_k(\uo)$ and $\s_k(\uz,\uo)$; see (\ref{upper}), (\ref{exp2}) and (\ref{exp3}). While $\s_k(\uz,\uo)$ is a convenient quantity to deal with probabilistically (because it is essentially a ratio of two independent Gaussians), the same cannot be said of $\s_k(\uo)$. Hence, the next step is to express $\s_k(\uo)$ in terms of a series in $\s_j(\uz,\uo)$ (for $j \le k$). This is done in Proposition \ref{expansion}, exploiting a certain recursive structure in the algebraic relationship between $\s_k(\uz,\uo)$ and $\s_k(\uo)$. In Proposition \ref{techgaf3}, we quantitatively establish that in this series expansion, it is essentially the leading term that dominates. 

All that remains now is to put  the pieces together, and obtain the desired estimates by controlling the leading term discussed above,  along with bounds on the error accrued in the process. This is done in Proposition \ref{gafestprep}. Finally, all the above ingredients are combined to formally state a quantitative comparison result between the conditional densities at the two locations in Proposition \ref{unifgaf}. 

\subsection{Polynomial Approximations}
\label{polapprox}
We focus on the event $\Omega_n^{m,\delta}$ which entails that $f_n$ has exactly $m$ zeroes inside $\D$, and there is a $\delta$ separation between $\partial \D$ and the outside zeroes. The corresponding event for the GAF zero process has positive probability, so by the distributional convergence  $\F_{n} \to \F$, we have that $\Omega_n^{m,\delta}$ has positive probability  (which is bounded away from 0 as $n \to \infty$).

Let us denote the  zeroes of $f_n$ (in uniform random order) inside $\D$  by $\uz=(\zeta_1,\cdots, \zeta_m)$ and those outside $\D$ by $\uo=(\omega_1,\omega_2,\cdots,\omega_{n-m})$. Let $s=\sum_{j=1}^m \z_j$. and ${\displaystyle \Sigma_S:=\{(\z_1,\dots,\z_m) \in \D^m: \sum_{j=1}^m \z_j =s\}}$ Then the conditional density $\rho^n_{\uo,s}(\uz)$ of $\uz$ given $\uo,s$ is of the form (see, e.g., \cite{FH})
\begin{equation}
  \label{condgaf}
\rho^n_{\uo,s}(\uz)=C(\uo,s) {\left| \Delta(\zeta_1,\dots,\zeta_{m},\omega_1,\dots,\omega_{n-m}) \right|^2 }{\displaystyle \left( \sum_{k=0}^n \left| \frac{{\sigma_k}{(\uz,\uo)}}{\sqrt{{n\choose k} k!} } \right|^2    \right)^{\! \! \! -n-1}}
\end{equation}
where $C(\uo,s)$ is the normalizing factor, and for a vector $\underline{v}=(v_1,\cdots,v_{N})$ we define the elementary symmetric polynomial of order $k$ as \[\sigma_k(\underline{v})=\sum_{1 \le i_1 < \cdots <i_k \le N} v_{i_1}\cdots v_{i_k}\] and for two vectors $\underline{u}$ and $\underline{v}$, $\s_k(\underline{u},\underline{v})$ is defined to be $\s_k(\underline{w})$ where the vector $\underline{w}$ is obtained by concatenating the vectors $\underline{u}$ and $\underline{v}$. As a matter of convention, for any vector $\underline{v}$ we set $\sigma_k(\underline{v})=0$ for all $k \le 0$ or $ \ge \mathrm{ Length \ of \ } \underline{v}$. 

Throughout this Subsection \ref{polapprox}, the zeroes will be those of $f_n$ with $n$ fixed. Define the quantity \[D(\uz,\uo)= \left( \sum_{k=0}^n \left| \frac{\sigma_k(\uz,\uo)}{\sqrt{{n\choose k} k!} } \right|^2    \right).\]

Let $(\uz,\uo)$ and $(\uz',\uo)$ be two vectors of points (under $\F_n$). Then the ratio of the conditional densities at these two vectors is given by
\begin{equation}
 \label{condgaf1}
\frac{\rho^n_{\uo,s}(\uz')}{\rho^n_{\uo,s}(\uz)}= \frac {|\Delta(\uz',\uo)|^2}  {|\Delta(\uz,\uo)|^2}   {\displaystyle \left( \sum_{k=0}^n \left| \frac{\sigma_k(\uz,\uo)}{\sqrt{{n\choose k} k!} } \right|^2    \right)^{\! \! n+1}} \bigg/ {\displaystyle \left( \sum_{k=0}^n \left|\frac {\sigma_k(\uz',\uo)}{\sqrt{{n\choose k} k!} } \right|^2    \right)^{\! \! n+1}}= \frac {|\Delta(\uz',\uo)|^2} {|\Delta(\uz,\uo)|^2} \l(\frac{D(\uz',\uo)} {D(\uz,\uo)}\r)^{-(n+1)}.
\end{equation}
The expression (\ref{condgaf1}) has two distinct components: the ratio of two Vandermonde determinants and the ratio of certain expressions involving the elementary symmetric functions of the zeroes. We will consider these two components in two separate sections.

\subsubsection{Ratio of Vandermondes}
\label{ratvand}
Here we consider the quantity ${|\Delta(\uz',\uo)|^2} \big/ {|\Delta(\uz,\uo)|^2}$. We proceed exactly as in the case of the Ginibre ensemble. We refer the reader to section \ref{ratcondgin}. The estimates here are valid for all pairs $(\uz,\uz') \in \D^m \times \D^m$.

\begin{proposition}
\label{nr2}
On the event $\Omega_{n}^{m,\delta}$, there exist quantities $K(\D,\delta)>0$ and $\X_n(\uo)>0$ such that for any $(\uz,\uz')  \in \D^m \times \D^m$ we have
\[\exp \bigg(- 2mK(\D,\delta)\X_n(\uo) \bigg) \frac{|\Delta(\uz')|^2}{|\Delta(\uz)|^2}\le \frac{|\Delta(\uz',\uo)|^2}{|\Delta(\uz,\uo)|^2} \le \exp \bigg( 2mK(\D,\delta)\X_n(\uo) \bigg) \frac{|\Delta(\uz')|^2}{|\Delta(\uz)|^2}\]
where ${\displaystyle \X_n(\uo) = \l|\sum_{\omega_j \in \mathcal{G}_n \cap \D^c}\frac{1}{\omega_j} \r|+ \l|\sum_{\omega_j \in \mathcal{G}_n \cap \D^c}\frac{1}{\omega_j^2} \r|+ \l( \sum_{\omega_j \in \mathcal{G}_n \cap \D^c}\frac{1}{|\omega_j|^3} \r)}$ and  $\E [\X_n(\uo)]\le c_1(\D,m)<\infty .$
\end{proposition}
\begin{remark}
The estimates on ${ \displaystyle \l|\sum \frac{1}{\o_j} \r| }$, ${ \displaystyle \l| \sum \frac{1}{\o_j^2} \r| }$ and ${\displaystyle \l( \sum \frac{1}{|\o_j|^3} \r) }$ which are necessary for Proposition \ref{nr2} are proved in section \ref{EIPZ}.
\end{remark}
\begin{corollary}
\label{nrc2.1}
Given  $M>0$, we can replace the $\X_n(\uo)$ in Proposition \ref{nr2} by a bound $M$ except on an event of probability less than $ c(m,\D)/M$.
\end{corollary}

\subsubsection{Ratio of Symmetric Functions}
\label{ratcond}
In this section we will restrict $\uz$ and $\uz'$ to lie in the same constant-sum hyperplane. For $s \in \C$, define \[\Sigma_s := \{\uz \in \D^m : \sum_{i=1}^m \z_i' = s \}. \]  
We want to bound the ratio $\l({D(\uz',\uo)} \big/ {D(\uz,\uo)}\r)^{n+1}$ from above and below. Our main goal is:
\begin{proposition}
\label{gafest}
Suppose $(\uz,\uo)$ is randomly generated from $\F_n$, let $s = \sum_{i=1}^m \z_i $, and define $\Sigma_s$ as above. 
Given $M>0$ large enough, $\exists n_0$ such that for all $n \ge n_0$ the following is true: with probability $\ge 1- C/M$  we have,  on the event $\Omega_n^{m,\delta}$ , the inequality \[ e^{-2K(m,\D)M\log M} \le \l( {D(\uz',\uo)} \bigg/ {D(\uz,\uo)} \r)^{n+1} \le e^{2K(m,\D)M\log M}  \] for all $\uz' \in \Sigma_s$.
\end{proposition}

We will first prove several auxiliary propositions.

We begin with the observation
\begin{equation}
\label{sig1}
\s_k(\uz,\uo) = \sum_{i=0}^{m}\s_i(\uz)\s_{k-i}(\uo).
\end{equation}
Note that $\s_0(\uz)=1$, $\s_1(\uz)=s=\s_1(\uz')$ and $|\s_i(\uz)|<{m \choose i}r_0^i$ for all $i \le m$, where $r_0$ is the radius of $\D$. Since both $\uz$ and $\uz' \in \Sigma_s$, we have
\begin{equation}
\label{sig2}
\s_k(\uz',\uo)=\s_k(\uz,\uo)+\sum_{i=2}^{m}[\s_i(\uz')-\s_i(\uz)]\s_{k-i}(\uo).
\end{equation}
Taking modulus squared on both sides and using the identity \[ \l| \sum_j a_j \r|^2 = \sum_j \l|a_j \r|^2 + 2 \sum_{j \ne k} \Re(a_j \ol{a_k}), \] we have $|\s_k(\uz',\uo)|^2=$
\begin{align*}  && |\s_k(\uz,\uo)|^2+\sum_{i=2}^{m}|\s_i(\uz')-\s_i(\uz)|^2|\s_{k-i}(\uo)|^2 + \sum_{i=2}^{m} 2\Re \l((\s_i(\uz')-\s_i(\uz)) \overline{\s_k(\uz,\uo)} \s_{k-i}(\uo) \r) \\ &&+\sum_{i,j=2}^{m}\sum_{i \ne j} 2\Re \l(\overline{(\s_i(\uz')-\s_i(\uz))}(\s_j(\uz')-\s_j(\uz)) \overline{\s_{k-i}(\uo)} \s_{k-j}(\uo) \r). \end{align*}
Dividing throughout by $\l( \c \r)^2$ and then summing the above over $k=0,\cdots,n$  we get
\begin{align*}
&&\sum_{k=0}^{n}\l|\frac{\s_k(\uz',\uo)}{\c}\r|^2 = \sum_{k=0}^{n}\l|\frac{\s_k(\uz,\uo)}{\c}\r|^2 + \sum_{i=2}^{m}|\s_i(\uz')-\s_i(\uz)|^2\l(\sum_{k=0}^{n}\l|\frac{\s_{k-i}(\uo)}{\c}\r|^2 \r)\\
&& + \sum_{i=2}^{m} 2\Re \l((\s_i(\uz')-\s_i(\uz)) \l( \sum_{k=0}^n \frac{\overline{\s_k(\uz,\uo)}}{\c} \frac{\s_{k-i}(\uo)}{\c}\r) \r)\\
&&+ \sum_{i,j=2}^{m}\sum_{i \ne j} 2\Re \l(\overline{(\s_i(\uz')-\s_i(\uz))}(\s_j(\uz')-\s_j(\uz)) \l( \sum_{k=0}^{n} \frac{\overline{\s_{k-i}(\uo)}}{\c} \frac{\s_{k-j}(\uo)}{\c} \r) \r).
\end{align*}
Applying the triangle inequality to the above, we obtain  
\begin{equation}
\label{upper}
D(\uz,\uo) -A(\uz,\uz',\uo)  \le  D(\uz',\uo) \le D(\uz,\uo) + A(\uz,\uz',\uo) 
\end{equation}
where we recall the fact that $D(\uz,\uo)= \left( \sum_{k=0}^n \left| \frac{\sigma_k(\uz,\uo)}{\sqrt{{n\choose k} k!} } \right|^2    \right)$
and 
\begin{align*} A(\uz,\uz',\uo):= &&\sum_{i=2}^{m}|\s_i(\uz')-\s_i(\uz)|^2\l(\sum_{k=0}^{n} \l| \frac{\s_{k-i}(\uo)}{\c} \r|^2 \r) \notag
+ \sum_{i=2}^{m} 2 |\s_i(\uz')-\s_i(\uz)| \l| \sum_{k=0}^n \frac{\overline{\s_k(\uz,\uo)}}{\c} \frac{\s_{k-i}(\uo)}{\c} \r| \notag  \\
&& + \sum_{i\ne j=2}^{m} 2 |\s_i(\uz')-\s_i(\uz)||\s_j(\uz')-\s_j(\uz)| \l| \sum_{k=0}^{n} \frac{\overline{\s_{k-i}(\uo)}}{\c} \frac{\s_{k-j}(\uo)}{\c} \r|. \end{align*}
Divide throughout by $D(\uz,\uo)$ in the above inequalities, and observe that \\
(a) $i$ and $j$ vary between 2 and $m$ (which is fixed and finite),\\
(b) $|\s_i(\uz')-\s_i(\uz)|$ is bounded for each $i$ as discussed immediately before (\ref{sig2}).\\
In view of these facts, we conclude that in order to bound $\l({D(\uz',\uo)} \big/ {D(\uz,\uo)}\r)^{n+1} $ between two constants, it suffices to show that the quantities
\begin{equation}
\label{exp2}
{\displaystyle D^{(1)}_i(\uz,\uo):= \l| \sum_{k=0}^n \frac{\overline{\s_k(\uz,\uo)}}{\c} \frac{\s_{k-i}(\uo)}{\c} \r|} \bigg/ {\displaystyle \l( \sum_{k=0}^{n}  \l| \frac{\s_{k}(\uz,\uo)}{\c} \r|^2 \r) }
\end{equation}
and
\begin{equation}
\label{exp3}
{ \displaystyle D^{(2)}_{i,j}(\uz,\uo):= \l| \sum_{k=0}^n \frac{\overline{\s_{k-i}(\uo)}}{\c} \frac{\s_{k-j}(\uo)}{\c} \r|} \bigg/  {\displaystyle \l( \sum_{k=0}^{n}  \l| \frac{\s_{k}(\uz,\uo)}{\c} \r|^2 \r) }
\end{equation}
for $m \ge i,j \ge 2$ are bounded above by $M/n$ with high probability (depending on $M$) where $M > 0$ is a constant.

Observe, by Vieta's formula, that  ${  {\s_k(\uz,\uo)} \big/ {\c}= (-1)^k {\xi_{n-k}} \big/ {\xi_n}}$, where the $\xi_j$-s are standard complex Gaussians. On the other hand, we do not yet have good control over $\s_k(\uo)$. In order for us to obtain such control, the idea is to obtain a convenient expansion of $\s_k(\uo)$ in terms of $\s_i(\uz,\uo)$. To this end, we formulate:
\begin{proposition}
\label{expansion}
On the event $\O^{m,\del}_n$ we have, for $0 \le k \le n-m $,
\[ \s_k(\uo)=\s_k(\uz,\uo)+\sum_{r=1}^k g_r \s_{k-r}(\uz,\uo)\]
where a.s. the random variables $g_r$ are $O(K(\D,m)^r)$ as $r \to \infty$, for a deterministic quantity $K(\D,m)$ and the constant in  $O$ being deterministic and not depending on $n$ and $\del$.
\end{proposition}
\begin{proof}
We begin with the observation that
\begin{equation}
\label{exp4}
\s_k(\uz,\uo)=\sum_{r=0}^m \s_r(\uz)\s_{k-r}(\uo).
\end{equation}
From this, we have
\begin{equation}
\label{exp5}
\s_k(\uo)=\s_k(\uz,\uo)-\sum_{r=1}^{m} \s_r(\uz)\s_{k-r}(\uo).
\end{equation}
Notice that the coefficients $ \s_r(\uz)$ in the above expansion are independent of $k$. This suggests a recursive pattern in the expression of $\s_k(\uo)$ in terms of $\s_j(\uz,\uo)$ with $j \le k$.  Proceeding inductively in (\ref{exp5}) we can similarly expand each of the lower order $\s_{k-r}(\uo)$ in terms of $\s_j(\uz,\uo)$ and obtain an expansion of $\s_k(\uo)$ in terms of $\s_j(\uz,\uo), j=1,\cdots,k$:
\begin{equation} \label{exp6} \s_k(\uo)=\s_k(\uz,\uo)+\sum_{r=1}^k g_r \s_{k-r}(\uz,\uo).\end{equation}
Consequently, the coefficient of $\s_k(\uz,\uo)$ is 1 and the rest of the coefficients  are polynomials in $\s_j(\uz),j=1,\cdots,m$. Due to the inductive structure, it is easy to see that $g_r$-s satisfy a recurrence relation:
\begin{equation} \label{rec1}
 \left(
 \begin{array}{c}
 g_i\\
 g_{i-1}\\
\dots\\
g_{i-m+1}
 \end{array}
 \right)
 =
 A
 \left(
 \begin{array}{c}
 g_{i-1}\\ g_{i-2}\\ \dots \\ g_{i-m}
 \end{array}
 \right)
\end{equation}
where \[ A:= \left(
 \begin{array}{cccc}
 -\s_1(\uz) & -\s_2(\uz) & \dots & -\s_m(\uz) \\
 1 & 0 & \dots & 0 \\
\dots & \dots & \dots & \dots \\
0 & 0 & 1 & 0 \\
 \end{array}
 \right),\]
with the boundary conditions $g_i=-\s_i(\uz)$ for $i=0,\cdots,m-1$ and $g_i=0$ for $i<0$.

We observe that the eigenvalues of $A$ are precisely the negatives of the inside zeroes $\zeta_1.\cdots,\zeta_m$. Due to the recursive structure in (\ref{rec1}), we note that $g_k$ is an element of $A^k$ applied to the vector $(-\s_{m-1}(\uz),\cdots,-\s_0(\uz))$. This implies that $g_k$ is a linear combination of the entries of $A^k$ (with the coefficients in the linear combination being independent of $k$, but depending on $\D$ and $m$).

From Proposition \ref{matest}, we see that if the eigenvalues of a matrix $A$ have modulus $< \rho'$,  then the entries of the matrix $A^k$ are  $o({\rho'}^k)$. Now, the eigenvalues of $A$ are $\zeta_i$-s and for each $i$ we have $|\zeta_i| \le \text{radius}(\D)=K(\D)$. Combining the last two paragraphs, we have $g_r=O((K(\D)+1)^r)$  Crucially, the constant in the $O$ above does not depend on $n$ and $\del$ (although it can depend on $m$).
\end{proof}
We complete this discussion with the following result on the growth of matrix entries:
\begin{proposition}
\label{matest}
Let $A$ be a $d \times d$ matrix. Let $\rho'>\rho(A) := \text{max}\{|\lambda| : \lambda \text{ an eigenvalue of } A \}$. Then $(A^r)_{ij}=o({\rho'}^r)$ as $r \to \infty$, for all $i,j$.
\end{proposition}
\begin{proof}
This follows from the well known result (Gelfand's spectral radius theorem): $\rho(A)= \lim_{r\to \infty} \|A^r \|^{1/r}$ where $\| \cdot \|$ denotes the Frobenius norm of the matrix, and the fact that \newline $\max_{1\le i,j \le d}|B_{ij}|\le  \| B \|$  for any $d \times d$ matrix $B$.
\end{proof}

The aim of the next proposition is to show that in the expansion of $\s_k(\uo)$ in terms of $\s_r(\uz,\uo)$ with $r \le k $, it is essentially the leading term $\s_k(\uz,\uo)$ that matters. For clarity, we will switch to a different indexing of the $\s$-s and involve the Gaussian coefficients of the polynomial more directly. Recall that ${\displaystyle  \frac{\s_{n-r}(\uz,\uo)}{{\sqrt{{n \choose n-r}(n-r)!}}}=(-1)^{n-r}\frac{\xi_r}{\xi_n}}$ . Divide both sides of the expansion proved in Proposition \ref{expansion} by $\c$ and note that $\c=\sqrt{{n \choose {k-r}}(k-r)!}\sqrt{(n-k+1)\cdots(n-k+r)} \hspace{5pt}$. 

\begin{notation}
 We denote the  falling factorial (of order $k$) of an integer $x$ by \[ (x)_k:=x(x-1)\cdots (x-k+1)\hspace{3pt}. \] For $k=0$ we set $(x)_0=1$.
\end{notation}
 
Switching variables to $l=n-k$, we can rewrite the expansion in Proposition \ref{expansion} as
\[ \frac{\s_{n-l}(\uo)}{\sqrt{{n \choose {n-l}}(n-l)!}}= (-1)^{n-l}\frac{\xi_l}{\xi_n}+\sum_{r=1}^{n-l}(-1)^{n-l-r}g_r\frac{\xi_{l+r}}{\xi_n}\frac{1}{\sqrt{(l+r)_r}}  \]
for $l \ge m$.
Define for any integer $l \ge 1$
\begin{equation}
\label{etan} \eta_l^{(n)}=\sum_{r=1}^{n-l} (-1)^{n-l-r} g_r \xi_{l+r}\frac{1}{\sqrt{(l+r)_r}} \hspace{5pt}.
\end{equation}
As a result, for $l \ge m$ we have 
\begin{equation}
 \label{expansion5}
\frac{\s_{n-l}(\uo)}{\sqrt{{n \choose {n-l}}(n-l)!}}=  \frac{1}{\xi_n} \l[(-1)^{n-l}\xi_l + \eta_l^{(n)}   \r] .
\end{equation}

\begin{proposition}
\label{techgaf3}
Let $\eta_l^{(n)}$ be as in (\ref{etan}). Then there exist positive random variables $\eta_l$ (independent of $n$) such that a.s. $\l|\eta_l^{(n)}\r|\le \eta_l $, and for fixed $l_0 \in \mathbb{N}$ and large enough $M >0$ we have
 \[(i)\P \l[ \eta_l>\frac{M}{l^{1/8}} \mathrm{\ for\  some \ } l \ge 1  \r] \le e^{-c_1M^2} \]
 \[(ii)\P \l[ \eta_l>\frac{M}{l^{1/8}} \mathrm{\ for\  some \ } l \ge l_0  \r] \le e^{-c_2M^2l_0^{\frac{1}{4}}}  \]
where $c_1,c_2$ are constants that depend on the domain $\D$ and on $m$.
\end{proposition}

\begin{proof}
We begin by recalling from Proposition \ref{expansion} that a.s. $|g_r|\le K^r$ for some $K=K(\D,m)$. Moreover, $\sqrt{(l+p)}\ge \sqrt{2}l^{\frac{1}{4}}p^{\frac{1}{4}}$
Hence, \[ |\eta_l^{(n)}|\le \sum_{r=1}^{n-l} \frac{K^r}{2^{r/2}l^{r/4}(r!)^{1/4}} \l| \xi_{l+r} \r| . \]
Let $B$ be such that ${ \displaystyle \sup_{r\ge 1} \frac{K^r}{2^{r/2}(r!)^{1/8}} \le B}$ . Then ${\displaystyle  |\eta_l^{(n)}|\le \sum_{r=1}^{n-l} \frac{B\l| \xi_{l+r} \r|}{l^{r/4}(r!)^{1/8}} } .$

Define  ${ \displaystyle \eta^{ }_l = \sum_{r=1}^{\infty} \frac{B\l| \xi_{l+r} \r|}{l^{r/4}(r!)^{1/8}} } \hspace{3 pt} .$ Clearly, $|\eta_l^{(n)}|\le \eta^{ }_l$.
Now, $|\xi_{l+r}|^2$ is an exponential random variable with mean 1, we have  \begin{equation}\label{norest}  \P \l( |\xi_{l+r}| \ge \frac{Ml^{r/8}(r!)^{1/16}}{B} \r) = \text{exp}\l(-\frac{M^2l^{r/4}(r!)^{1/8}}{B^2} \r) . \end{equation}
Moreover, \[\sum_{l=1}^{\infty}\sum_{r=1}^{\infty}\text{exp}\l(-\frac{M^2l^{r/4}(r!)^{1/8}}{B^2} \r) \\
\le C_1\text{exp}\l( -M^2/B^2 \r) .\]
So we can have $|\xi_{l+r}| \le \frac{Ml^{r/8}(r!)^{1/16}}{B}$ for all $l \ge1, r\ge0$, except on an event of probability $\le C_1\text{exp}\l( -M^2/B^2 \r)$. The last expression can be made arbitrarily small by fixing $M$ to be large enough. Note that for large enough $M$, we have $C\text{exp}\l( -M^2/B^2 \r) \le \text{exp}\l( -cM^2 \r)$.
On the complement of this small event, we have for all $l \ge 1$  \[ \eta_l \le \sum_{r=1}^{\infty} \frac{M}{l^{r/8}(r!)^{1/16}} \le \sum_{r=1}^{\infty} \frac{M}{l^{r/8}(r!)^{1/16}} \le C\frac{M}{l^{1/8}} \hspace{3pt}.\]
We then absorb this factor of $C$ into $M$ by changing the constant $c_1$ appearing in the exponent of the tail estimate (i) we want to prove.

For proving (ii) we proceed exactly as in the case of (i) above and sum (\ref{norest}) over $r \ge 1$ and $l \ge l_0$.
\end{proof}

We are now ready to prove bounds on the quantities $D^{(1)}_{i,j}(\z,\o)$ and $D^{(2)}_{i,j}(\z,\o)$ in (\ref{exp2}) and (\ref{exp3}). To this end, we first express them in terms of $\xi$-s, using ${\s_{n-l}(\uz,\uo)} \big/ {\sqrt{{n \choose {n-l} }(n-l)!}}=(-1)^{n-l}{\xi_l} \big/ {\xi_n}$ and (\ref{expansion5}). For $2\le i,j \le m$ these expressions, (after canceling out $|\xi_n|^2$ from numerator and denominator) reduce to  :

\begin{equation}
\label{exp2a}
{\displaystyle D^{(1)}_{i}(\z,\o) = \l| \sum_{l=0}^{n-i}\frac{\overline{(-1)^{n-l}\xi_{l}}\l((-1)^{n-l-i}\xi_{l+i}+\eta_{l+i}^{(n)}\r)}{\sqrt{(l+i)_i}} \r|} \bigg/  {\displaystyle  \l( \sum_{l=0}^{n}|\xi_l|^2 \r)   }
\end{equation}
and
\begin{equation}
\label{exp3a}
{\displaystyle D^{(2)}_{i,j}(\z,\o) = \l|\sum_{l=0}^{n-i \wedge n-j}\frac{\overline{\l((-1)^{n-l-i}\xi_{l+i}+\eta_{l+i}^{(n)}\r)}\l((-1)^{n-l-j}\xi_{l+j}+\eta_{l+j}^{(n)}\r)}{\sqrt{(l+i)_i}\sqrt{(l+j)_j}}\r|} \bigg/ {\displaystyle \l( \sum_{l=0}^{n}|\xi_l|^2 \r)  } .
\end{equation}

Define \[ \e_n=\sum_{l=0}^n |\xi_l|^2 , \qquad \L_{ij}^{(n)}=\sum_{l=0}^{n-i \wedge n-j}\frac{(-1)^{i+j}\overline{\xi_{l+i}}\xi_{l+j}}{\sqrt{(l+i)_i (l+j)_j}} \hspace{5 pt} , \]
\[
\M_{ij}^{(n)}=\sum_{l=0}^{n-i \wedge n-j}\frac{|\xi_{l+i}|\eta_{l+j}}{\sqrt{ (l+i)_i (l+j)_j}} \hspace{5 pt} , \qquad 
\N_{ij}^{(n)}=\sum_{l=0}^{n-i \wedge n-j}\frac{\eta_{l+i}\eta_{l+j}}{\sqrt{(l+i)_i (l+j)_j}} \hspace{5 pt} .\]
Recall that for $i=0$  the product $(l+i)_i$ in the denominator is replaced by 1.
Expanding out the products and applying the triangle inequality in (\ref{exp2a}) and (\ref{exp3a}), we observe that in order to obtain upper bounds for $D^{(1)}_{i}(\z,\o)$ and $D^{(2)}_{i,j}(\z,\o)$, it suffices to estimate from above the following quantities  (remember that $|\eta_l^{(n)}|\le \eta_l$ for all $l$):
\begin{equation} \label{rememb}  \frac{\l( \l|\L_{0j}^{(n)} \r|+ \M_{0j}^{(n)} \r) }{\e_n} , \qquad  
\frac{\l( {\l|\L_{ij}^{(n)} \r|}+ {\M_{ij}^{(n)} } + {\M_{ji}^{(n)} } + {\N_{ij}^{(n)} } \r)}{\e_n} \mathrm{ \ for \ } 2 \le i,j \le m \hspace{5 pt}. \end{equation}
To be more specific, by a simple application of the triangle inequality, $D^{(1)}_{i}(\z,\o)$ is bounded above by  ${\l( \l|\L_{0j}^{(n)} \r|+ \M_{0j}^{(n)} \r) }\big/{\e_n}$ and $D^{(2)}_{i,j}(\z,\o)$ is bounded above by  ${\l( {\l|\L_{ij}^{(n)} \r|}+ {\M_{ij}^{(n)} } + {\M_{ji}^{(n)} } + {\N_{ij}^{(n)} } \r)}\big/{\e_n}$.

Let \begin{align*} \Y_n= &&  \sum_{i=2}^m \l|\L_{0i}^{(n)} \r| + \sum_{i=2}^m \M_{0i}^{(n)} +  \sum_{i,j \ge 2}^m \l|\L_{ij}^{(n)} \r| + \sum_{i,j \ge 2}^m \M_{ij}^{(n)} + \sum_{i,j \ge 2}^m \M_{ji}^{(n)}   + \sum_{i,j \ge 2}^m \N_{ij}^{(n)}   \hspace{5 pt}. \end{align*}

Recall  (\ref{upper})  and  that for fixed $m$ and $\D$, we have that $\s_i(\uz)$ are uniformly bounded. Putting all these together, we deduce that: for some constant $K(m,\D)$ we have 
\begin{equation}
\label{upperlower}
1-K(m,\D)\frac{\Y_n}{\e_n} \le \frac{D(\uz',\uo)}{D(\uz,\uo)} \le 1+ K(m,\D)\frac{\Y_n}{\e_n} \hspace{5 pt}.
\end{equation}
Regarding $\Y_n$ and $\e_n$, we have the following estimates:
\begin{proposition}
\label{gafestprep}
Given $M>0$ we have:\\
$\mathrm{(i)}$ $\P[\Y_n\ge M\log M]\le c(m,\D)/M$ ,\\
$\mathrm{(ii)}$ There exists $n_0$ such that for $n \ge n_0$ we have $\P[\frac{n}{2}\le|\e_n|\le 2n] \ge 1-\frac{1}{M} \hspace{5 pt}.$ 
\end{proposition}
\begin{proof} \underline{Part (i):}\\
 We will show that if  $X$ is one of $\L_{ij}^{(n)},\M_{ij}^{(n)}$ or $\N_{ij}^{(n)}$, then it satisfies $\P[|X|>M\log M]<c'(m,\D)/M$. This would imply $\P[\Y_n> M\log M]<c_1(m,\D)/M$ albeit for a different constant $c(m,\D)$. This is easily seen as follows. The number of summands in the definition of $\Y_n$ is a polynomial in $m$, let us call it $p(m)$. Now, for each summand $X$ of $\Y_n$, we have \[\P\l[|X|> \frac{M\log M}{p(m)}\r] \le \P\l[|X|> \frac{M}{p(m)}\log \frac{M}{p(m)}\r] \le c'(m,\D)p(m)/M .\] If $\Y_n > M\log M$ and there are $p(m)$ summands, at least one of the summands must be $>M\log M /p(m)$. By a union bound over the summands, this implies $\P[\Y_n >M\log M]\le c'(m,\D)p(m)^2/M$. Setting $c(m,\D)=c'(m,\D)p(m)^2$ as the new constant, this would give us part (i) of the proposition.

It remains to show that each of the summands in $\Y_n$ satisfies $\P[|X|>M\log M]<c(m,\D)/M$ with the quantity $c(m, \D)$ not depending on $n$. The understanding is that $c$ will depend on the particular summand, and we will take a maximum over the $ p(m) $ summands in order to obtain $c'(m,\D)$ as in the previous paragraph. We consider the cases of $\L,\M$ and $\N$ separately.

\underline{Estimating $\L_{ij}^{(n)}, i \vee j \ge 2$:} 

We have \[ \E \l[ |\L_{ij}^{(n)}|^2 \r] = \sum_{l=0}^{n-i \wedge n-j}\frac{1}{(l+i)_i  (l+j)_j} \le \sum_{l=0}^{\infty}\frac{1}{(l+i)_i (l+j)_j} = c(i,j) \le c <\infty \]
(since either $i$ or $j \ge 2$). An application of Chebyshev's inequality proves the desired tail bound on $\L_{ij}^{(n)}$.

\underline{Estimating $\M_{ij}^{(n)}, j \ge 2$:}

By Proposition \ref{techgaf3} we can assume (after discarding an event $A$ of probability  $\le e^{-c(\log M)^2}$) that $\text{max}_{l \ge 1}  \eta_l  \le \log M/l^{1/8} $.  Applying the triangle inequality to the definition of $\M_{ij}^{(n)}$ and using the upper bound on $\eta_{l+j} $ we have \[ \E \l[ |\M_{ij}^{(n)}|{1\hskip-4pt{\rm 1}}_{A^c} \r] \le \log M \ \E \l[ \sum_{l=0}^{n-i \wedge n-j} \frac{|\xi_{l+i}|}{(l+j)^{1/8}\sqrt{(l+i)_i (l+j)_j}} \r] \le  m(i,j) \log M  \le c \log M < \infty, \]
where ${\displaystyle m(i,j)=\sum_{l=0}^{\infty} \frac{\E[|\xi|]}{(l+j)^{1/8}\sqrt{(l+i)_i (l+j)_j}} \le c < \infty }$, and $\xi$ is a standard complex Gaussian. In the last step, we have used the fact that $ j \ge 2$ in order to upper bound $m(i,j)$ uniformly in $i$ and $j$ . Applying Markov inequality gives us the desired tail estimate ${\displaystyle \P (|M_{ij}^{(n)}|{1\hskip-4pt{\rm 1}}_{A^c}>M\log M) \le {c}/{M}}.$ This, along with ${\displaystyle \P(A)<e^{-c(\log M)^2}}$ completes the proof.

\underline{Estimating $\N_{ij}^{(n)}, i \wedge j \ge 2$:}

We consider the event $A$ as in the case of $\M_{ij}^{(n)}$. On $A^c$, we use $\eta_{l+i}\eta_{l+j}\le (\log M)^2<M\log M $ (for large $M$), and the rest is bounded above by \[\sum_{l=0}^{n-i \wedge n-j}\frac{1}{l^{1/4}\sqrt{(l+i)_i (l+j)_j}} \le \sum_{l=0}^{\infty}\frac{1}{(l+i)^{1/8}(l+j)^{1/8}\sqrt{(l+i)_i (l+j)_j}} =c(i,j) \le c <\infty .\]  This, along with $\P[A]<e^{-c(\log M)^2}$ establishes the desired tail bound on $N_{ij}^{(n)}$.

All these arguments complete the proof of part (i) of Proposition \ref{gafestprep}.

\underline{Part(ii):}
\newline
We simply observe that $|\xi|^2$-s are i.i.d. exponentials and by the strong law of large numbers, $\e_n/n \to 1$ a.s. From this, the statement of part (ii) follows.

\end{proof}

Define
\[ \L_{ij}=\sum_{l=0}^{\infty}\frac{(-1)^{i+j}\overline{\xi_{l+i}}\xi_{l+j}}{\sqrt{(l+i)_i (l+j)_j}} \hspace{5 pt}, \qquad \tau(L_{ij}^{(n)})= \L_{ij} - \L_{ij}^{(n)}, \hspace{5 pt} i \vee j \ge 2  \]
\[
\M_{ij}=\sum_{l=0}^{\infty}\frac{{|\xi_{l+i}|}\eta_{l+j}}{\sqrt{ (l+i)_i (l+j)_j}} \hspace{5 pt} , \qquad  \tau(M_{ij}^{(n)})= \M_{ij} - \M_{ij}^{(n)}, \hspace{5 pt}  j \ge 2 \] 
\[\N_{ij}=\sum_{l=0}^{\infty}\frac{\eta_{l+i}\eta_{l+j}}{\sqrt{(l+i)_i (l+j)_j}} \hspace{5 pt}, \qquad  \tau(N_{ij}^{(n)})= \N_{ij} - \N_{ij}^{(n)}, \hspace{5 pt}   i \wedge j \ge 2 .\]
For $i=0$  the product $(l+i)_i$ in the denominator is replaced by 1.
The above random variables have finite first moments, as can be seen by arguing on similar lines to the estimates in Proposition \ref{gafestprep}. Similar arguments with first moments of $\tau(L_{ij}^{(n)}),\tau(M_{ij}^{(n)}),\tau(N_{ij}^{(n)})$ also show that:
\begin{proposition}
\label{uselat1} 
As $n \to \infty$, each of the random variables $\tau(L_{ij}^{(n)}), \tau(M_{ij}^{(n)}), \tau(N_{ij}^{(n)})$ converges to $0$ in $L^1$, and hence, in probability.
\end{proposition}

Now we are ready to prove Proposition \ref{gafest}.

\begin{proof}[Proof of Proposition \ref{gafest}]
We refer back to equation (\ref{upperlower}):\[1-K(m,\D)\frac{\Y_n}{\e_n} \le \frac{D(\uz',\uo)}{D(\uz,\uo)} \le 1+ K(m,\D)\frac{\Y_n}{\e_n} \hspace{5 pt}.\] By virtue of Proposition \ref{gafestprep} we have, dropping an event of probability $\le c/M$, the inequality $\l| \Y_n/\e_n \r| \le 2 M\log M/n  $ (for $n \ge n_0$, where $n_0$ is large enough such that  Proposition \ref{gafestprep} part(ii) holds). Raising this to the $(n+1)$th power we obtain the desired result. All constants are absorbed in $K(m,\D)$. 
\end{proof}
\begin{corollary}
\label{gafestcor}
Let $(\uz,\uo)$ be picked randomly from the distribution $\P\l[\F_n\r]$ and $s=\sum_{i=1}^m \z_i$. Let $\Sigma_s=\{ \uz': \sum_{i=1}^m \zeta_i' =s  \}$. Given $M>0$ large enough, there exists $ n_0$ such that for all $n \ge n_0$, we have, on $\Omega_n^{m,\delta}$ ( except on an event of probability $\le C/M$ ) \[ e^{-4K(m,\D)MlogM} \le \l( {D(\uz'',\uo)} \bigg/ {D(\uz',\uo)} \r)^{n+1} \le e^{4K(m,\D)MlogM}  \]  for all $(\uz'',\uz') \in \Sigma_s \times \Sigma_s$ .
\end{corollary}
\begin{proof}
We consider the identity \[ \frac{D(\uz'',\uo)}{D(\uz',\uo)} =  \frac{D(\uz'',\uo)}{D(\uz,\uo)} \bigg/  \frac{D(\uz',\uo)}{D(\uz,\uo)}\] for   $(\uz',\uz'')\in \Sigma_s \times \Sigma_s $ where $s=\sum_{i=1}^m \z_i$, and applying the bounds in Proposition \ref{gafest}.

\end{proof}

\subsubsection{Estimate for Ratio of Conditional Densities}
\label{ERCD}
\begin{proposition}
 \label{unifgaf}
Let $(\uz,\uo)$ be picked randomly from the distribution $\P\l[\F_n\r]$ and $s=\sum_{i=1}^m \z_i$. Let $\Sigma_s=\{ \uz': \sum_{i=1}^m \zeta_i' =s  \}$. There exists a constant $K(m,\D,\delta)$ such that given $M>0$ large enough, for $n\ge n_0(m,M,\D)$  the following inequalities hold on $\Omega_n^{m,\delta}$, except on an event of probability $\le c(m,\D)/M$: \[      \exp \bigg({-K(m,\D,\delta)M\log M} \bigg)\frac{|\Delta(\uz'')|^2}{|\Delta(\uz')|^2}  \le \frac{\rho^n_{\uo,s}(\uz'')}{\rho^n_{\uo,s}(\uz')} \le \exp \bigg({K(m,\D,\delta)M\log M}\bigg)\frac{|\Delta(\uz'')|^2}{|\Delta(\uz')|^2}      \]
 for all $(\uz',\uz'') \in \Sigma_s \times \Sigma_s$.
\end{proposition}
\begin{proof}
We simply put together the estimates for the ratios of the Vandermondes as well as the symmetric functions and subsume all relevant constants in $c(m,\D)$ and $K(m,\D,\del)$.
\end{proof}

\subsection{Estimates for Inverse Powers of Zeroes}
\label{EIPZ}
In this section we prove estimates on the (smoothed) sum of inverse powers of GAF zeroes.

\begin{proposition}
\label{invgafestlem}
Let $\Phi$ be a $C_c^{\infty}$ radial function supported on the annulus between $r_0$ and $3r_0$. Let $\Phi_R = \Phi(z/R)$. We have
\[ \mathrm{(i)} \hspace{5 pt} \E \l[ \l| \int \frac{\Phi_R(z)}{z^l} \, \mathrm{d}[\F_n](z) \r| \r] \le C_1(r_0,\Phi,l)/R^l 
 \]
\[ \mathrm{(ii)} \hspace{5 pt} \E \l[\int \frac{\Phi_R(z)}{|z|^l} \, \mathrm{d}[\F_n](z) \r] \le C_2(r_0,\Phi,l)/R^{l-2} \hspace{3 pt}.  \]
The same is true for $\F$ in place of $\F_n$. The quantities $C_1(r_0,\Phi,l)$ and $C_2(r_0,\Phi,l)$ do not depend on $n$.
\end{proposition}

\begin{proof} \underline{ Part (i):}
We begin with \[ \int \frac{1}{z^l}\Phi_R(z) \, \mathrm{d}[\F_n](z) = \int \frac{1}{z^l}\Phi_R(z) \Delta \text{log}(|f_n(z)|)\, \mathrm{d}\el(z) .\]
Now, $\text{log}(\sqrt{\E[|f_n(z)|^2]})$ is a radial function, and Laplacian of a radial function is also radial. Hence,  \[\int \frac{1}{z^l} \Phi_R(z) \Delta \text{log}(\sqrt{\E[|f_n(z)|^2]})\, \mathrm{d}\el(z)=0 \] because for $l \ge 1$, ${1}/{z^l}$ when integrated against a radial function  always gives 0 due to rotational symmetry. Let $\hat{f}_n(z)={f_n(z)} \big/ {\sqrt{\E[|f_n(z)|^2]}}$. Then the above argument implies
\begin{equation}
\label{e1}
\int \frac{1}{z^l}\Phi_R(z) \, \mathrm{d}[\F_n](z) =  \int \frac{1}{z^l}\Phi_R(z) \Delta \text{log}(|\hat{f}_n(z)|)\, \mathrm{d}\el(z) .
\end{equation}
Integrating by parts the right hand side of (\ref{e1}) we have
\begin{equation}
\label{e2}
\E \l[ \l|\int \frac{1}{z^l}\Phi_R(z) \, \mathrm{d}[\F_n](z) \r| \r] \le \int \l| \Delta \l(\frac{1}{z^l}\Phi_R(z)\r) \r|  \E \l[\l|\text{log}(|\hat{f}_n(z)|)\r| \r] \, \mathrm{d}\el(z) .
\end{equation}
Now, the integrand is non-zero only for $Rr_0 \le |z| \le 3Rr_0$. Hence it is easy to see that $\l| \Delta \l(\frac{1}{z^l}\Phi_R(z)\r) \r| \le C(r_0,\Phi,l)/R^{l+2}$. Further, $\E \l[\l|\text{log}(|\hat{f}_n(z)|)\r| \r]$ is a constant because $\hat{f_n}$ is $N_{\C}(0,1)$.

Finally, $\int_{3Rr_0 \cdot D \setminus Rr_0\cdot D} \, \mathrm{d}\el(z) = 8\pi r_0^2R^2$, where $D$ is the unit disk.

Putting all these together, we have \[\int \l| \Delta \l(\frac{1}{z^l}\Phi_R(z)\r) \r|  \E \l[\l|\text{log}(|\hat{f}_n(z)|)\r| \r] \, \mathrm{d}\el(z) \le C_1(r_0,\Phi,l)/R^l\] as desired (all constants are subsumed in $C_1(r_0,\Phi,l)$).

\underline{ Part (ii):} Since $\Phi$ is a radial function on $\C$,  there exists a function $\wt{\Phi}$ on $\R_{\ge 0}$ such that $\Phi(z)=\wt{\Phi}(|z|)$. 
 We have, with $r=|z|$,
\[\E \l[ \int \frac{1}{|z|^l}\Phi_R(z) \, \mathrm{d}[\F_n](z) \r] = c \int \frac{1}{r^l} \wt{\Phi}_{R}(r)\Delta \text{log} (\sqrt{K_n(z,z)}) r\mathrm{d}r   \]
where the integral on the right hand side is over the non-negative reals. Integrating by parts,
\[ \l| \int \frac{1}{r^l} \wt{\Phi}_{R}(r)\Delta \text{log} (\sqrt{K_n(z,z)}) r\mathrm{d}r \r| \le   \int \l| \Delta \l(\frac{1}{r^l}\wt{\Phi}_R(r)\r) \r|  \text{log}(\sqrt{K_n(z,z)}) r\mathrm{d}r . \]
But $\text{log}(\sqrt{K_n(z,z)})  \le \text{log}(\sqrt{K(z,z)})=\frac{1}{2}r^2 \le \frac{9}{2}r_0^2R^2$ and $\l| \Delta \l(\frac{1}{z^l}\Phi_R(z)\r) \r| \le C(r_0,\Phi,l)/R^{l+2}$,
hence \[ \int \l| \Delta \l(\frac{1}{r^l}\wt{\Phi}_R(r)\r) \r|  \text{log}(\sqrt{K_n(z,z)}) r\mathrm{d}r \le C_2(r_0,\Phi,l)/R^{l-2} \]
for another constant $C_2$ (which clearly does not depend on $n$).
\end{proof}

\begin{remark}
\label{invgafestlemrem}
The quantity $C_1(r_0,\Phi,l)$ above is of the form $p(l)\l(\frac{1}{r_0} \r)^{l}C_1(\Phi)$ where $p$ is a fixed polynomial (of degree 2), and $C_1(\Phi)$ is a quantity that depends only on $\Phi$. Both $p$ and $C_1(\Phi)$  do not depend on $n$. Similarly,  $C_2(r_0,\Phi,l)$ is of the form $p_1(l)\l(\frac{1}{r_0} \r)^{l-2}C_2(\Phi)$; where $p_1$ is another degree 2 polynomial and $C_2(\Phi)$ does not depend on $n$.
\end{remark}

Let $r_0$ be the radius of $\D$.   Let $\varphi$ be a non-negative radial $C_c^{\infty}$ function supported on $[r_0,3r_0]$ such that $\varphi=1$ on $[\frac{3r_0}{2},2r_0]$ and $\varphi(r_0 + r)=1-\varphi(2r_0+2r)$, for $0\le r \le \frac{1}{2}r_0$. In other words, $\varphi$ is a test function supported on the annulus between $r_0$ and $3r_0$ and its ``ascent'' to 1 is twice as fast as its ``descent''.

Let $\widetilde{\varphi}$ be another radial test function with the same support as $\varphi$, and $\widetilde{\varphi}(r)=1$ for $ r_0 \le  r \le \frac{3 r_0}{2}$ and $\widetilde{\varphi}=\varphi$ otherwise.

\begin{proposition}
\label{estedge}
\[ \E \l[ \l|\int \frac{\widetilde{\varphi}(z)}{z^l} \, \mathrm{d}[\F_n](z) \r|\r] \le \E \l[ \l|\int \frac{\widetilde{\varphi}(z)}{|z|^l} \, \mathrm{d}[\F_n](z) \r|  \r] \le c(r_0,\widetilde{\varphi},l)\] where $c(r_0,\wt{\varphi},l)$ does not depend on $n$, and the same result remains true when $\F_n$ is replaced by $\F$.
\end{proposition}

\begin{proof}
For fixed $l$, let  ${\displaystyle a(n)= \int \frac{\widetilde{\varphi}(z)}{z^l} \, \mathrm{d}[\F_n](z)}$ and  ${\displaystyle b(n)= \int \frac{\widetilde{\varphi}(z)}{|z|^l} \, \mathrm{d}[\F_n](z)}$. We have, $\E \l[ |a(n)| \r] \le \E \l[ b(n) \r] = C\int \frac{\widetilde{\varphi}(z)}{|z|^l}\Delta \log \l(K_n(z,z)\r) \mathrm{d}\el(z)$ for some constant $C$. Using the uniform convergence  of the continuous functions $\Delta \log \l( K_n(z,z) \r) \to \Delta \log \l( K(z,z) \r)<\infty $ on the (compact) support of $\widetilde{\varphi}$, we deduce that $\E \l[b(n) \r]$-s are bounded by constants that do not depend on $n$ but depend on $r_0,\widetilde{\varphi}$ and $l$. It is obvious from the above argument that the same holds true for $\F$ instead of $\F_n$.
\end{proof}

\begin{proposition}
\label{invgafest-tail}
%
%
%
%
%
Let $r_0$ be the radius of $\D$ . Let $\varphi$ and $\wt{\varphi}$ be defined as above.   

(i) The random variables \[ S_l(n):= \int \frac{\widetilde{\varphi}({z})}{z^l} \, \mathrm{d}[\F_n](z) + \sum_{j=1}^{\infty} \int \frac{\varphi_{2^j}({z})}{z^l} \, \mathrm{d}[\F_n](z) = \sum_{\o \in \F_n \cap \D^c} \frac{1}{\o^l} \quad (\text{ for }l \ge 1) \] and \[
 \widetilde{S}_l(n):= \int \frac{\widetilde{\varphi}({z})}{|z|^l} \, \mathrm{d}[\F_n](z) + \sum_{j=1}^{\infty} \int \frac{\varphi_{2^j}({z})}{|z|^l} \, \mathrm{d}[\F_n](z) = \sum_{\o \in \F_n \cap \D^c} \frac{1}{|\o|^l} \quad (\text{ for }l \ge 3)\]
have a finite first moment which, for every fixed $l$, is bounded above uniformly in $n$. 

(ii) There exists $k_0=k_0(\varphi)\ge 1$,  not depending on $n$ and $l$, such that for $k \ge k_0$ the ``tails'' of $S_l(n)$ and $\widetilde{S}_l(n)$ beyond the disk $2^k \cdot \D $, given by \[\tau_l^{n}(2^k):= \sum_{j=k}^{\infty} \int \frac{\varphi_{2^j}({z})}{z^l} \, \mathrm{d}[\F_n](z) \, \,(\text{ for }l \ge 1) \quad \text{ and } \quad \widetilde{\tau}_l^{n}(2^k):= \sum_{j=k}^{\infty} \int \frac{\varphi_{2^j}({z})}{|z|^l} \, \mathrm{d}[\F_n](z) \, \,(\text{ for }l \ge 3)\] satisfy the estimates \[ \E \l[ \l| \tau_l^{n}(2^k) \r|\r]\le C_1(\varphi,l)/2^{kl/2} \quad \text{ and } \quad  \E \l[ \l| \widetilde{\tau}_l^{n}(2^k) \r|\r]\le C_2(\varphi,l)/2^{k(l-2)/2}. \]

All of the above remain true when $\F_n$ is replaced by $\F$, for which we use the notations ${S}_l$ and $\widetilde{S}_l$ to denote the quantities corresponding to $S_l(n)$ and $\widetilde{S}_l(n)$.

 \end{proposition}

\begin{remark}
For $\F$, by the sum ${\displaystyle  \l( \sum_{\o \in \F \cap \D^c}\frac{1}{\o^l} \r)}$ we denote the quantity \[S_l=\int \frac{\widetilde{\varphi}({z})}{z^l} \, \mathrm{d}[\F](z) + \sum_{j=1}^{\infty} \int \frac{\varphi_{2^j}({z})}{z^l} \, \mathrm{d}[\F](z)\] due to the obvious analogy with $\F_n$, where the corresponding  sum $S_l(n)$  is indeed equal to ${\displaystyle  \l( \sum_{\o \in \F_n \cap \D^c}\frac{1}{\o^l} \r)}$ with its usual meaning.
\end{remark}

\begin{proof}[Proof of Proposition \ref{invgafest-tail}]
The functions $\widetilde{\varphi}$ and $\varphi_{2^j}, j\ge 1$ form a partition of unity on $\D^c$, hence the identities appearing in part (i). That the left hand side in part (i) has finite expectation can be seen from Proposition \ref{invgafestlem} and Proposition \ref{estedge}; it is uniformly bounded in $n$ because so are $C(r_0,\varphi,l)$ in Proposition \ref{invgafestlem} and $c(r_0,\varphi_0)$ in Proposition \ref{estedge}.

Fix $n,l\ge 1 $. Set $\psi_k= \int \frac{\varphi_{2^k}({z})}{z^l} \, \mathrm{d}[\F_n](z)$ for $k \ge 1$, and $\psi_0=\int \frac{\widetilde{\varphi}({z})}{z^l} \, \mathrm{d}[\F_n](z)$.
When $l \ge 3$ we also define $\gamma_k= \int \frac{\varphi_{2^k}({z})}{|z|^l} \, \mathrm{d}[\F_n](z)$ for $k \ge 1$, and $\gamma_0=\int \frac{\widetilde{\varphi}({z})}{|z|^l} \, \mathrm{d}[\F_n](z)$. Let ${\Psi}_k$ and ${\Gamma}_k$ denote the analogous quantities defined with respect to $\F$ instead of $\F_n$.

For $k \ge 1$, we apply Proposition \ref{invgafestlem} part (i) to the function $\Phi=\varphi$ and $R=2^k$ to obtain \begin{equation}\label{reference1}\E[|\psi_k|]\le C(r_0,\varphi,l)/(2^k)^l.\end{equation} We also observe that $\E[|\psi_0|] \le \E[|\gamma_0|]$ which, for fixed $l$, is uniformly bounded above in $n$, using Proposition \ref{estedge}. 

The above arguments imply that for $l \ge 1$ \[\E[|S_l(n)|]\le \sum_{k=0}^{\infty}\E[|\psi_k|]<\infty .\]
The results for $\wt{S}_l(n)$ are similar, utilizing part (ii) of Proposition \ref{invgafestlem}.

To obtain $k_0$ as in part (ii), we recall  Remark \ref{invgafestlemrem} and find $k_0$ such that $p(l)\l(\frac{1}{r_0}\r)^l 2^{-kl/2}C(\varphi) \le 1/2$ for all $k \ge k_0$. Such a $k_0$ can be obtained as follows: there exists $C_1(\varphi)>0$ such that for all $l \ge 1$ we have $( 2\cdot C(\varphi) \cdot p(l) )^{1/l} \le C_1(\varphi)$. We now choose $k_0$ such that $\frac{2^{-k_0/2}}{r_0} \le \frac{1}{C_1(\varphi)}$.
This $k_0$ will clearly be independent of $n$, because so is $C(\varphi)$, and by choice it is independent of $l$. To estimate $\E \l[ \l| \tau_l^{n}(2^k) \r|\r]$, we now simply sum (\ref{reference1}) for $k \ge k_0$.

The result for $\wt{\tau}^n_l(2^k)$ follows from a  similar argument.

The same arguments yield the corresponding results when $F_n$ is replaced by $F$. 
\end{proof}

\begin{corollary}
\label{invgafest-tailcor}
For $R=2^k,k\ge k_0$ as in Proposition \ref{invgafest-tail}. We have $\P[|\tau_l^{n}(R)|>R^{-l/4}]\le R^{-l/4}$ and $\P[|\widetilde{\tau}_l^n(R)|>R^{-(l-2)/4}]\le R^{-(l-2)/4}$, and these estimates remain true when $f_n$ is replaced with $f$.
\end{corollary}

\begin{proof}
We use the estimates on the expectation on $|\tau_l^n(R)|$ and  $|\widetilde{\tau}_l^n(R)|$ from Proposition \ref{invgafest-tail} and apply Markov's inequality.
\end{proof}

With notations as above, we have
\begin{proposition}
\label{conv-gaf}
For each $l \ge 1$ we have  $S_l(n) \rightarrow S_l$ in probability, and for each $l \ge 3$ we have $\widetilde{S}_l(n) \rightarrow \widetilde{S}_l$ in probability, and hence we have such convergence a.s. along some subsequence, simultaneously for all $l$. 
\end{proposition}

\begin{proof}
 We argue on similar lines to the proof of Proposition \ref{conv}.
\end{proof}

Define $S_k(\D,n)=\sum_{z \in \F_n}1/z^k$ and $S_k(\D)=\sum_{z \in \F} 1/z^k$. Set $\a_k(n)=S_k(\D,n)+S_k(n)$ and $\a_k=S_k(\D)+S_k$. Observe that $\a_k(n)=\sum_{z \in \F_n} 1/z^k$. Then we have:

\begin{proposition}
 \label{conv-gaf1}
For each $k$, $\a_k(n) \to \a_k$ in probability as $n \to \infty$. Hence, there is a subsequence such that $\a_k(n) \to \a_k$ a.s. when $n \to \infty$ along this subsequence, simultaneously for all $k$. 
\end{proposition}

\begin{proof}
Since the finite point configurations given by $\F_n|_{ \D} \to \F|_{\D}$ a.s. and a.s. there is no point at the origin, therefore $S_k(\D,n) \to S_k(\D)$ a.s.  
This, combined with Proposition \ref{conv-gaf}, gives us the desired result.  
\end{proof}

\section{Limiting procedure for GAF zeroes}
\label{limgafz}
In this section, we will work in the framework of Section \ref{limcond}. More specifically, we will show that the conditions for Theorem \ref{abs} are satisfied, which will give us the desired conclusion. We will introduce the definitions and check all the conditions here except the fact $\O(j)$ exhausts $\O^m$. This last criterion will be verified in the subsequent Proposition \ref{exhaust}.

In terms of the notation used in Section \ref{limcond}, we have $X^n=\F_n$ and $X=\F$.

\begin{proof}  [\textbf{Proof of Theorem \ref{gaf-2} for a disk}] 
We will invoke Theorem \ref{abs}. We will define the relevant quantities (as in the statement of Theorem \ref{abs}) and verify that they satisfy the required conditions for that theorem to apply.

Following the notation in Theorem \ref{abs}, we begin with  $\O^m$, $m\ge 0$.  If $m=0$, there is nothing to be proved, and if $m=1$ then  from Theorem \ref{gaf-1} the position of the zero inside $\D$ is known, and therefore the result is trivial. So we focus on the case $m \ge 2$. 
 
Our candidate for $\nu(\xout({\xi}),\cdot)$ (refer to Theorem \ref{abs}) is the probability measure $\Theta^{-1}|\Delta(\uz)|^2\mathrm{d}\els(\uz)$ on $\Sigma_{S(X_{\out}(\xi))}$, where $\Delta(\uz)$ is the Vandermonde determinant formed by the coordinates of $\uz $ ,  $\els$ is the Lebesgue measure on $\Sigma_{S(X_{\out}(\xi))}$ and $\Theta$ is the normalizing factor. Note that this form of the candidate measure, heuristically, is in tune with the fact the conditional measure is supported on $\Sigma_{S(X_{\out}(\xi))}$ and with the estimate in Proposition \ref{unifgaf} on conditional densities. It is also a mathematical expression of the notion that $\o$ determines only the number and the sum of the points in $\D$, and ``nothing more''. Here we recall the definition of $S(X_{\out}(\xi))$ from Theorem \ref{gaf-1} and the definition of $\Sigma_{S(X_{\out}(\xi))}$ from Section \ref{ratcond}. Note that as soon as we define $\nu(\xout,\cdot)$ which maps $\xi$ to $\mathcal{M}(\D^m)$, this automatically induces a map from $\S_{\out}$ to $\mathcal{M}(\D^m)$ which satisfies the required measurability properties.
 
To find the sequence $n_k$ (which will be the same for every $j$ in our case), we proceed as follows. Let $N_g(K)$ denote the number of zeroes of the function $g$ in a set $K$, $\Gamma$ denote the closed annulus of width 1 around $\D$, and in the next statement let $Z$ be any of the variables $L,M$ or $N$ as in Proposition \ref{gafestprep} (and the immediately preceding discussion) with $p=0$ or $2 \le p \le m$ and $2 \le q \le m$.  We have the probabilities of each of the following events converging to 0 as $n \to \infty$:
$\{|S_j(n) - S_j|>1\}_{j=1,2}, \{|\wt{S}_3(n) - \wt{S}_3|>1\}, \{|N_f(\overline{\D})\neq N_{f_n}({\overline{\D}})|\},\{|N_f(\Gamma)\neq N_{f_n}(\Gamma)|\},\{|\tau(Z^{(n)}_{pq})|>1\},\{ \frac{1}{n}\l( \sum_{j=0}^n |\xi_j|^2\r) \le 1/2\}$ and $\{\l| |\Delta(\uz(n))|^2 - |\Delta(\uz)|^2 \r| >\frac{1}{M_j}\}$ (where $\uz(n)$ and $\uz$ correspond to the zeroes inside $\D$ of $f_n$ and $f$ respectively). Call the union of these events $\mathfrak{B}_n$. For a given $k \ge 1$, let $n_k'$ be such that $\P(\mathfrak{B}_n) < 2^{-k}$ for all $n \ge n_k'$. From Proposition \ref{conv-gaf} we have a subsequence  such that $S_j(n) \to S_j$ ($j=1,2$) and $\wt{S}_3({n}) \to \wt{S}_3$ a.s. as $n \to \infty$ along that subsequence. For a given $k \ge 1$, we define $n_k$  to be the least integer in that subsequence which is $\ge n_k'$.

Fix a sequence of positive real numbers $M_j \uparrow \infty$, such that $M_1 = \min_j M_j > \sqrt{2}$. 

On the event $\O^m_{n_k}$ (which entails that $\fn$ has $m$ zeroes inside $\D$), let $\zn$ and $\on$  denote the vector of inside and outside zeroes of $\fn$ (taken in uniform random order) respectively. Let $\sn$ denote the sum of  the inside zeroes: $\sn:=\sum_{j=1}^m \zn_j$, where $\zn_j$ are the coordinates of the vector $\zn$. By $\Sigma_{\sn}$ we will denote the (random) set $\{\uz' \in \D^m: \sum_{j=1}^m \uz'_j=\sn  \}$. Also recall the notation  $\rho^{n_k}_{\uo,s}({\uz})$ to be the conditional density (with respect to the Lebesgue measure on $\Sigma_s$) of the inside zeroes (at $\uz \in \D^m$) given the vector of outside zeroes to be $\uo$ and the sum of the inside zeroes to be $s$. 

We \textbf{define our event $\onk$}  by the following conditions on the zero set $(\zn,\on)$ of $\fn$:
 \begin{enumerate}
 \item There are exactly $m$ zeroes of $f_{n_k}$ in $\D$
 \item There are no zeroes of $f_{n_k}$ in the closed annulus of width $1/M_j$ around $\D$.
 \item $|S_1({n_k})|\le M_j, |S_2({n_k})|\le M_j, |\wt{S}_3({n_k})|\le M_j$.
 \item There exists $\uz' \in \Sigma_{\sn}$ such that $(\uz',\on)$ satisfies ($\uo$ here is an abbreviation for $\on$):
\begin{enumerate} 
\item ${\displaystyle \l| \sum_{r=0}^{n_k} \frac{\overline{\s_r(\uz',\uo)}}{\cnew} \frac{\s_{r-i}(\uo)}{\cnew} \r|} \bigg/ {\displaystyle \l( \sum_{r=0}^{n_k}  \l| \frac{\s_{r}(\uz',\uo)}{\cnew} \r|^2 \r) } \le M_j/n_k$ for each $2\le i \le m$.
\item ${ \displaystyle \l| \sum_{r=0}^{n_k} \frac{\overline{\s_{r-i}(\uo)}}{\cnew} \frac{\s_{r-j}(\uo)}{\cnew} \r|} \bigg/  {\displaystyle \l( \sum_{r=0}^{n_k}  \l| \frac{\s_{r}(\uz',\uo)}{\cnew} \r|^2 \r) }\le M_j/n_k$ for all $2 \le i,j \le m$.
\item $ \frac{1}{M_j} \le |\Delta(\uz')|^2 \le M_j  $
\end{enumerate}
\end{enumerate}

%
%

Clearly, $\onk$ depends only on the quantities $\on$ and $\sn$ (and on no other information from $\zn$). In particular, for a vector $\uo$ of outside zeroes of $\fn$, if there exists $\uz \in \Sigma_{s}$ such that $(\uz,\uo)$ satisfies the conditions required in the definition of $\onk$, then all $(\uz',\uo)$  (with $\uz' \in \Sigma_s$) satisfies these conditions. 
From (\ref{condgaf1}), Proposition \ref{nr2}, (\ref{exp2}), (\ref{exp3}) and the discussion therein, it follows that on the event $\onk$ we have \begin{equation} \label{compar} b(j)|\Delta(\uz)|^2 \le {\rho^{n_k}_{\uo,s}(\uz)}\le B(j)|\Delta(\uz)|^2 \end{equation} for positive quantities $B(j)$ and $b(j)$ (that depend on $M_j$). More explicitly, if $(\uz,\uo) \in \onk$ such that $(\uz',\uo)$ is as in condition 4 (in the definition of $\onk$), then from (\ref{condgaf1}), we have \[ \rho^n_{\uo,s}(\uz)=  \rho^n_{\uo,s}(\uz') \frac {|\Delta(\uz,\uo)|^2} {|\Delta(\uz',\uo)|^2} \l(\frac{D(\uz,\uo)} {D(\uz',\uo)}\r)^{n+1}. \] But from Proposition \ref{nr2} (with $\del=1/M_j$), we have \[ \bigg( - 2mK(\D,\delta)\X_n(\uo) \bigg) \frac {|\Delta(\uz)|^2} {|\Delta(\uz')|^2} \le  \frac {|\Delta(\uz,\uo)|^2} {|\Delta(\uz',\uo)|^2} \le \bigg( 2mK(\D,\delta)\X_n(\uo) \bigg) \frac {|\Delta(\uz)|^2} {|\Delta(\uz')|^2}.  \]
where $X_n(\uo)=|S_1(n_k)|+|S_2(n_k)|+|\wt{S}_3(n_k)|$. Using conditions 3 and 4(c) in the definition of $\onk$ (and recalling that $m,D$ is fixed for our purposes) we deduce that \[ C_1(j) |\Delta(\uz)|^2  \le \frac {|\Delta(\uz,\uo)|^2} {|\Delta(\uz',\uo)|^2} \le C_2(j) |\Delta(\uz)|^2 \] for positive numbers $C_1(j)$ and $C_2(j)$. In order to bound $ \l(\frac{D(\uz,\uo)} {D(\uz',\uo)}\r)^{n+1}$, we refer to (\ref{upper}), (\ref{exp2}), (\ref{exp3}) and the discussion therein. From the assumptions 4(a) and 4(b) in the definition of $\onk$, it follows that \[ C_3(j) \le \l(\frac{D(\uz,\uo)} {D(\uz',\uo)}\r)^{n+1} \le C_4(j) \] for positive quantities $C_3(j)$ and $C_4(j)$. Combining the inequalities in the two displays above, we get (\ref{compar}).
 
To obtain condition (\ref{abscond}), we introduce some further notations. On the event $\Omega^m_{n_k}$, let $\gamma_{n_k}(\uo;s)$ denote the conditional probability measure on the sum $s$ of inside zeroes given the vector of outside zeroes of $f_{n_k}$ to be $\uo$.
Further, let $\el_s$ denote the Lebesgue measure on the set $\Sigma_s$.
 
For any $A \in \AA$ and $B \in \B$ , we can write 
\begin{equation}
\label{equ1}
\tP [  (\xin \in A) \cap  ( X_{\out}^{n_k} \in B) \cap \Omega_{n_k}(j) ] = \int_{B \cap \onk} \l( \int_{A \cap \Sigma_s} \rho^{n_k}_{\uo,s}(\uz) \mathrm{d}\el_s(\uz) \r) \mathrm{d}\gamma_{n_k}(\uo;s) \mathrm{d}\P_{\xout^{n_k}}(\uo) \end{equation}
Setting $h(s,A):= \l( \int_{A \cap \Sigma_s} |\Delta(\uz)|^2 \mathrm{d}\el_s(\uz) \r) \big/ \l( \int_{\Sigma_s} |\Delta(\uz)|^2 \mathrm{d}\el_s(\uz) \r)$, we get from (\ref{compar}) that the right hand side of (\ref{equ1}) is \begin{equation} \label{equ2} \int_{B \cap \onk} \l( \int_{A \cap \Sigma_s} \rho^{n_k}_{\uo,s}(\uz) \mathrm{d}\el_s(\uz) \r) \mathrm{d}\gamma_{n_k}(\uo;s) \mathrm{d}\P_{\xout^{n_k}}(\uo) \asymp_j \int_{B \cap \onk}  h(s,A) \mathrm{d}\gamma_{n_k}(\uo;s) \mathrm{d}\P_{\xout^{n_k}}(\uo).  \end{equation}
The integral on the right hand side can be written as $\E \bigg[ h(\sn,A)1\hskip-4pt{\rm 1}[\xout^{n_k} \in B]1\hskip-4pt{\rm 1}[\onk] \bigg]$, where  $1\hskip-4pt{\rm 1}[E]$ is the indicator function of the event $E$ and the expectation is with respect to the random variable $(\sn,\xout^{n_k})$.  Note that for fixed $A$,  the function $h(s,A)$ is continuous in $s$, because $A \in \A^m$, and that $0\le h(s,A)\le 1$. Since, as $k \to \infty$, we have $\sn \to S(\xout)$ a.s. Using the Dominated Convergence Theorem we get \begin{equation} \label{equ3}\E \l[ h(\sn,A)1\hskip-4pt{\rm 1}[\xout^{n_k} \in B]1\hskip-4pt{\rm 1}[\onk] \r] = \E \l[ h(S(\xout),A)1\hskip-4pt{\rm 1}[\xout^{n_k} \in B]1\hskip-4pt{\rm 1}[\onk] \r] + o_k(1),\end{equation} where $o_k(1)$ is a quantity which $\to 0$ as $k \to \infty$, for fixed $A$ and $B$ (actually, in this particular case, it can be easily seen that the convergence is uniform in $B \in \B$) . Also, as $k \to \infty$, we have $1\hskip-4pt{\rm 1}[\xout^{n_k} \in B] \to 1\hskip-4pt{\rm 1}[\xout \in B]$ since $B$ is a compact set. Since $0\le h(s,A) \le 1$ a.s., this implies that \begin{equation} \label{equ4} \E\l[ h(S(\xout),A)1\hskip-4pt{\rm 1}[\xout^{n_k} \in B]1\hskip-4pt{\rm 1}[\onk] \r]=\E \l[ h(S(\xout),A)1\hskip-4pt{\rm 1}[\xout \in B]1\hskip-4pt{\rm 1}[\onk] \r]+o_k(1).\end{equation} 

Observe that for $\xi \in \Xi$, we have $h(S(\xout(\xi)),A)=\nu(\xout(\xi),A)$. Therefore, putting (\ref{equ1})-(\ref{equ4}) together,  we obtain (\ref{abscond}).
The only condition in Theorem \ref{abs} that remains to be verified is the fact that $\O(j)$-s exhaust $\O^m$, which will be taken up in Proposition \ref{exhaust}.

By Theorem \ref{abs}, this proves that a.s. we have $\rho(\xout,\cdot)\equiv {\nu}(\xout,\cdot)$. Since a.s. ${\nu}(\xout,\cdot)\equiv \els$, we have $\rho(\xout,\cdot)\equiv \els$, as desired. 
\end{proof}

We end this section with a proof that $\O(j)$-s exhaust $\O^m$:

\begin{proposition}
\label{exhaust}
With definitions as above, $\O(j):=\varliminf_{k \to \infty} \onk$ exhausts $\O^m$ as $j \to \infty$. 
\end{proposition}
\begin{proof}
We begin by showing that $\O(j) \subset \O^m$ for each $j$. Due to the convergence of $X^{n_k} \to X$ on compact sets, we have $\varliminf_{k \to \infty} \O^m_{n_k}=\O^m$. Since $\onk \subset \O^m_{n_k}$ for each $k$, therefore $\O(j) \subset \O^m$ for each $j$.  Since $M_j <M_{j+1}$, it is also clear that $\O_{n_k}(j)\subset \O_{n_k}(j+1)$ for each $k$. Hence $\O(j) \subset \O(j+1)$.

To show that $\P(\O^m \setminus \O(j))\to 0$ as $j\to \infty$, we will define for each $j$ three events $\O^1(j),\O^2(j)$ and $\O^3(j)$; the last event being the one which allows us to obtain the desired conclusion. The main reason for introduction of several events is as follows. The events $\O(j)$ are defined in terms of the polynomials, whereas the right setting in which the exhaustion property can be understood easily is that of the GAF. Thus, the proof of the exhaustion property involves a transition from events defined in terms of the polynomials to events defined in terms of the GAF. The sequence of events $\O^1(j),\O^2(j)$ and $\O^3(j)$ accomplishes this in steps, such that the correct relationships between the events (in terms of containment or equality except on a set of small probability) is transparent from the successive definitions.

For each $j$ we first \textbf{define the  event $\onp$} by demanding that the zeroes $(\uz,\uo)$ of $\fn$ satisfy the following conditions :
\begin{enumerate}
 \item There are exactly $m$ zeroes of $f_{n_k}$ in $\D$
 \item There are no zeroes of $f_{n_k}$ in the closed annulus of width $1/M_j$ around $\D$.
 \item $|S_1({n_k})|\le M_j, |S_2({n_k})|\le M_j, |\wt{S}_3({n_k})|\le M_j$.
 \item ${\displaystyle \l| \sum_{t=0}^{n_k} \frac{\overline{\s_t(\uz,\uo)}}{\ck} \frac{\s_{t-i}(\uo)}{\ck} \r|} \bigg/ {\displaystyle \l( \sum_{t=0}^{n_k}  \l| \frac{\s_{t}(\uz,\uo)}{\ck} \r|^2 \r) } \le M_j/n_k$ for each $2\le i \le m$.
 \item ${ \displaystyle \l| \sum_{t=0}^{n_k} \frac{\overline{\s_{t-i}(\uo)}}{\ck} \frac{\s_{t-j}(\uo)}{\ck} \r|} \bigg/  {\displaystyle \l( \sum_{k=0}^{n_k}  \l| \frac{\s_{t}(\uz,\uo)}{\ck} \r|^2 \r) } \le M_j/n_k$ for all $2 \le i,j \le m$.
 \item $  \frac{1}{M_j} \le |\Delta(\uz)|^2 \le M_j   $.
 \end{enumerate} 
Clearly, $\onp \subset \onk$ and hence $\varliminf_{k \to \infty} \onp =\O^1(j) \subset \O(j)$.

Next, we  \textbf{define the  event $\onq$} by the following conditions (refer to Proposition \ref{gafestprep} and the discussion immediately preceding it to recall  the notations):
\begin{enumerate}
 \item There are exactly $m$ zeroes of $f_{n_k}$ in $\D$
 \item There are no zeroes of $f_{n_k}$ in the closed annulus of width $1/M_j$ around $\D$.
 \item $|S_1(n_k)|\le M_j, |S_2({n_k})|\le M_j, |\wt{S}_3({n_k})|\le M_j$.
 \item Each of $|L^{(n_k)}_{0q}|,|M^{(n_k)}_{0q}|,|L^{(n_k)}_{pq}|,|M^{(n_k)}_{pq}|,|N^{(n_k)}_{pq}|$ is $\le M_j/8$ for all $2 \le p,q \le m$.
 \item $E_{n_k}/n_k \ge 1/2$.
 \item $  \frac{1}{M_j} \le |\Delta(\uz(n_k))|^2 \le M_j   $, where $\uz(n_k)$ corresponds to the zeroes of $\fn$ inside $\D$. 
\end{enumerate} 
It is clear from (\ref{rememb}) that $\onq \subset \onp$. 

Finally, we define an event $\O^3(j)$ by the conditions:
\begin{enumerate}
 \item There are exactly $m$ zeroes of $f$ in $\D$
 \item There are no zeroes of $f$ in the closed annulus of width $1/M_j$ around $\D$.
 \item $|S_1|\le M_j, |S_2|\le M_j-1, |\wt{S}_3|\le M_j-1$.
 \item Each of $|L_{0q}|,|M_{0q}|,|L_{pq}|,|M_{pq}|,|N_{pq}|$ is $\le \frac{M_j}{8}-1$ for all $2 \le p,q \le m$.
 \item $  \frac{2}{M_j} \le |\Delta(\uz)|^2 \le \frac{M_j}{2}   $, where $\uz$ corresponds to the zeroes of the GAF $f$ inside $\D$.
\end{enumerate} 
Notice (refer to Proposition \ref{uselat1}) that  for each of the random variables $Z^{(n_k)}_{pq}$ in condition 4 defining $\onq$ and $Z_{pq}$ in condition 4 defining $\O^3(j)$ (where $Z=L,M,N$) we have \begin{equation} \label{t} |Z^{(n_k)}_{pq}| \le |Z_{pq}|+ |\tau(Z_{pq}^{(n_k)})|.\end{equation}  

Now consider the event $F_{n_k}$ where any one of the following conditions hold:
\begin{enumerate}
 \item There are $m$ zeroes of $f$ in $\D$, but the same assertion does not hold for $f_{n_k}$.
 \item There are no zeroes of $f$ in the closed annulus of width $1/M_j$ around $\D$, but this is not true for $f_{n_k}$.
 \item $|S_1-S_1({n_k})|>1$ or $|S_2 - S_2({n_k})|>1$ or $|\wt{S}_3-\wt{S}_3({n_k})|>1$.
 \item $|\tau(Z_{pq}^{(n_k)})|>1$ for any of the random variables appearing in condition 4 of $\onq$.
 \item $E_{n_k}/n_k < 1/2$.
 \item $\l| | \Delta(\uz(n_k)) |^2 - |\Delta (\uz)|^2 \r|>\frac{1}{M_j}$, where $\uz(n_k)$ corresponds to the  zeroes of $\fn$ inside $\D$ and $\uz$ corresponds to the zeroes of $f$ inside $\D$.
\end{enumerate}
It is easy to see (using (\ref{t})) that $\O^3(j) \setminus F_{n_k} \subset \onq \subset \onk$ (recall that by choice of $M_j$-s, for each $j$  we have  $M_j > \sqrt{2}$, implying $\frac{M_j}{2}+\frac{1}{M_j} \le M_j$ ). However, by our choice of the sequence $n_k$, we have $\P(F_{n_k})<2^{-k}$, hence by the Borel Cantelli lemma,  $\varliminf_{k \to \infty} \l( \O^3(j) \setminus F_{n_k} \r) = \O^3(j)$. Hence, $\O^3(j) \subset \O(j)= \varliminf_{k \to \infty} \onk$.

However, each random variable that appears in the definition of $\O^3(j)$ does not put any mass at 0 or $\infty$, and the width of the annulus around $\D$ in condition 2 in its definition 
goes to $0$ as $j \to \infty$. Hence, the probability that each of the conditions 2-5 defining $\O^3(j)$ holds goes to 1 as $j \to \infty$. Finally, condition 1 is just the definition of $\O^m$. Hence, as $j \to \infty$, we have $\P\l(\O^m \setminus \O^3(j)\r)\to 0$. But $\O^3(j) \subset \O(j)$, hence $\P\l(\O^m \setminus \O(j)  \r) \to 0$ as $j \to \infty$, as desired.

\end{proof}

\section{Reconstruction of GAF from Zeroes and an Analog of Vieta's formula}
\label{reconsgaf}
In this section we prove the reconstruction Theorem \ref{reconstruction} for the planar GAF using the estimates we obtained in Section \ref{EIPZ}. En route, we establish an analogue of Vieta's formula for the planar GAF.

\subsection{Vieta's formula for the planar GAF}
\label{vie}
It is an elementary fact that for a polynomial \[p(z)=\sum_{j=0}^{N} b_jz^j \] whose roots are $\{z_j\}_{j=1}^N$, we have for any $1 \le k \le N$,   \begin{equation}
\label{v1}
 (-1)^kb_{N-k}/b_N = \sum_{i_1<i_2< \cdots < i_k} z_{i_1}\cdots z_{i_k}.                                                                                                                          \end{equation}
When $b_0 \ne 0$, we equivalently have  
 \begin{equation}
\label{v2}
 (-1)^kb_{k}/b_0 = \sum_{i_1<i_2< \cdots < i_k} \frac{1}{z_{i_1} \cdots z_{i_k}}.                                                                                                                          \end{equation}                                                                                                                           
This kind of result is referred to as Vieta's formula. For entire functions, it is not clear in general how to relate the coefficients to the zeroes of the function; for example it may not be possible to provide any reasonable interpretation to the function of the zeroes appearing on the right hand side of (\ref{v2}). 
                                                                                                                          
The quantity on the right hand side of (\ref{v1}) is the elementary symmetric function of order $k$ in the variables $z_1,\cdots,z_N$, denoted by $e_k(z_1,\cdots,z_N)$. If we introduce the power sum $\beta_k=\sum_{j=1}^Nz_j^k$, then it is known from Newton's identities that for each $k$ we have \[e_k(z_1,\cdots,z_N)=P_k(\beta_1,\cdots,\beta_k)\] where $P_k$ is a homogeneous symmetric polynomial of degree $k$ in the variables $(\beta_1,\cdots,\beta_k)$. The polynomial $P_k$ is called the Newton polynomial of degree $k$, and it is known that the coefficients of $P_k$ depend only on $k$ and do not depend on $n$ (refer \cite{Sta}, chapter 7).
                                                                                                                          
Recall from Section \ref{intro} the notation $\a_k(n)=\l(\sum_{z \in \F_n} {1}/{z^k}\r)$ and Proposition \ref{conv-gaf1} shows that for each $k$ the random variables $\a_k(n)$ converge in probability as $n \to \infty$. The limit $\a_k$ is an analogue of the sum $\l(\sum_{z\in \F} 1/z^k \r)$; in fact for $k \ge 3$ this inverse power sum converges absolutely, and coincides with $\a_k$.                                                                                                                         
\begin{proposition}
\label{Vieta} 
For the planar GAF zero process, we have, a.s. $\frac{1}{\sqrt{k!}}\frac{\xi_k}{\xi_0}=(-1)^kP_k(\a_1,\a_2,\cdots,\a_k)$ for each $k \ge 1$. 
\end{proposition}

\begin{proof}
We begin by observing that under the natural coupling of the planar GAF $f$ and its approximating polynomials $f_n$, we have by the Vieta's formula for each $f_n$:
\begin{equation}
 \label{second}
\frac{1}{\sqrt{k!}}\frac{\xi_k}{\xi_0}=(-1)^kP_k(\a_1(n),\cdots,\a_k(n)) 
\end{equation}
where $P_k$ is, as before, the Newton polynomial of degree $k$. Clearly, $P_k$ is continuous in the input variables. We also know, from Proposition \ref{conv-gaf} that, as $n \to \infty$ (possibly along some appropriately chosen subsequence), $\a_k(n)\to \a_k$ a.s., simultaneously for all $k \ge 1$. Taking this limit in (\ref{second}) and using the continuity of $P_k$, we get a.s.
\begin{equation}
 \label{third}
\frac{1}{\sqrt{k!}}\frac{\xi_k}{\xi_0}=(-1)^kP_k(\a_1,\cdots,\a_k). 
\end{equation}
\end{proof}

\subsection{Proof of Theorem \ref{reconstruction}}
Since $\xi_0$ is a complex Gaussian,  $|\xi_0|\ne0$ a.s. For $|\xi_0|\ne 0$, we can write 
\begin{equation} 
\label{expr}
 f(z)=\frac{\xi_0}{|\xi_0|}. |\xi_0|\left(1+\frac{\xi_1}{\xi_0}\frac{z}{\sqrt{1!}} + \frac{\xi_2}{\xi_0}\frac{z^2}{\sqrt{2!}} + \cdots + \frac{\xi_k}{\xi_0}\frac{z^k}{\sqrt{k!}} + \cdots \infty \right).
\end{equation}
Recall that, for two random variables $U$ and $V$ defined on the same probability space, we say that $U$ is measurable with respect to $V$ if $U$ is measurable with respect to the sigma-algebra generated by $V$. From Proposition \ref{Vieta} we  have that for each $k$ the random variable $\frac{\xi_k}{\xi_0}$ is measurable with respect to $\F $. From the strong law of large numbers, a.s. we have \[\frac{|\xi_0|^2 + \ldots + |\xi_{k-1}|^2}{k} \to 1 .\]
Therefore, \[|\xi_0|= \lim_{k \to \infty}  k^{1/2} \l( \sum_{j=0}^{k-1} \frac{|\xi_j|^2}{|\xi_0|^2} \r)^{-1/2}=\chi, \] where we recall the definition of $\chi$ from Section \ref{intro}.
Since  $k^{1/2} \l( \sum_{j=0}^{k-1} \frac{|\xi_j|^2}{|\xi_0|^2} \r)^{-1/2}$ is measurable with respect to $\F$ for each $k$, therefore $\chi=|\xi_0|$ is also measurable with respect to $\F$.

Let us define the random variable $\z=\xi_0/|\xi_0|$  (it is set to be equal to 0 when $|\xi_0|=0$). The random function $g$ as in the statement of the theorem clearly satisfies a.s.  \[g(z)=|\xi_0|\left(1+\frac{\xi_1}{\xi_0}\frac{z}{\sqrt{1!}} + \frac{\xi_2}{\xi_0}\frac{z^2}{\sqrt{2!}} + \cdots + \frac{\xi_k}{\xi_0}\frac{z^k}{\sqrt{k!}} + \cdots \right).\] Then (\ref{expr}) can be re-written as 
\begin{equation}
 \label{expr1}
f(z)=\z g(z)
\end{equation}
almost surely, and $g$ is measurable with respect to $\F$ and $\z$ is distributed uniformly on $\mathbb{S}^1$. 

Therefore, all that remains to complete the proof is to show that $\z$ and $\F$ are independent. For this, note that for  $\theta \in \mathbb{S}^1$ the random function \[  |\xi_0| + \sum_{j \ge 1} \frac{\theta  \xi_j}{\sqrt{j!}} z^j \]
has the same distribution (irrespective of $\theta$). This is because  since the $\xi_i$-s are independent complex Gaussians with mean 0,   the vectors
$ \l( |\xi_0|, \theta \xi_j  \r)_{j \ge 1}$ and $\l( |\xi_0|, \xi_j \r)_{j \ge 1}  $ have the same distribution for each fixed $\theta \in \mathbb{S}^1$. 
Now note that \[ g(z)=|\xi_0|+ \sum_{j \ge 1}\frac{\overline{\z}\xi_j}{\sqrt{j!}}z^j.\] Therefore, the distribution of $g(z)$ given $\z$ does not depend on the value of $\z$. Hence, the random function $g$ and $\z$ are independent. But, $g$ and $f$ have the same zero set a.s. and consequently $\z$ and $\F$ are independent random variables. This completes the proof.

\section{Proof of Theorem \ref{abs}}
\label{techproof}
In this section, we will complete the proof of Theorem \ref{abs}. Before the main proof, we will state two useful propositions, whose proofs will be deferred to the end of this section. First, we look at the heuristics which motivate some of our main steps. In the statement of Theorem \ref{abs}, think of the point processes $X^n$ to be approximations of the limiting point process $X$, e.g., matrix eigenvalues approximating the infinite Ginibre ensemble, or the polynomial zeroes approximating the zeroes of the GAF. Often, the $X^n$-s have finitely many points and have joint probability densities for their points, which makes it possible to estimate probabilities of events defined with respect to such processes, like in the case of the Ginibre or the Gaussian zero ensembles. On the other hand, our goal is to draw conclusions about conditional probabilities defined with respect to the presumably infinite point process $X$. The task at hand is to make a transition from the former to the latter. With this in mind, if we look at the setting in which Theorem \ref{abs} is stated, then the hypotheses involve probabilities of events defined with respect to the approximations $X^n$; whereas the conclusion of the theorem involves conditional probabilities for the limiting process $X$. Thus, broadly speaking, we need to make two kinds to transitions: one from probabilities without conditioning to conditional probabilities, and the other from $X^n$ to $X$. The main goal of these somewhat technical propositions is to make a connection between statements about probabilities of events (defined with respect to $X^n$), as in (\ref{abscond}), with conditional probabilities for $X$, as in (\ref{targeteq}).

Proposition \ref{tp1} makes a transition between the inequalities connecting probabilities of the kind that appear on the right hand side of (\ref{abscond}) to domination relations between conditional expectations. Thus, Proposition \ref{tp1} connects quantities as on the right hand side of (\ref{abscond}) with conditional measures, which appear in the conclusion (\ref{targeteq}) of Theorem \ref{abs}.

For stating Propositions \ref{tp1} and \ref{tp2}, we assume the hypotheses of Theorem \ref{abs}, namely, conditions (a) and (b). Recall the notation that the Borel sigma-algebra on $\D^m$ is denoted by $\mathfrak{B}(\D^m)$.
\begin{proposition}
 \label{tp1}
Let $\mu_j, j=1,2$  be two mappings with $\mu_j : \Xi \times \mathfrak{B}(\D^m) \to [0,1]$, such that for each $\xi \in \Xi$ we have $\mu_j(\xi,\cdot) \in \mathcal{M}(\D^m)$, and for each Borel set $A \subset \D^m$, the function $\xi \to \mu_j(\xi,A)$ is measurable. Suppose $\mu_1$ and $\mu_2$ satisfy, for each positive integer $j_0$, $A \in \A^m$ and Borel set $B \subset \S_{\out}$  \begin{equation} \label{abs0}       \int_{X_{\out}^{-1}(B) \cap \O(j_0) } \mu_1({\xi},A) \mathrm{d}{\tP}(\xi) \le c(j_0) \cdot \int_{X_{\out}^{-1}(B) \cap \O^m} \mu_2({\xi},A) \mathrm{d}{\tP}(\xi) 
\end{equation}
for some positive number $c(j_0)$.
Then a.s. on  the event $\O^m$ we have \begin{equation} \label{interm} \E\l[ \mu_1(\xi,\cdot)| \xout \r] \ll  \E\l[ \mu_2(\xi,\cdot)| \xout \r]. \end{equation}
\end{proposition}

%

\begin{remark}
 Here the conditional expectation is with respect to the sigma-algebra generated by the random variable $\xout$.
\end{remark}

Proposition \ref{tp2} connects probabilistic expressions of the type appearing on the left hand side of (\ref{abscond}), involving $X^n$, with probabilities of similar events involving $X$. Thus, Proposition \ref{tp2} broadly  effects a transition from the $X^n$-s  to $X$, in the setting of Theorem \ref{abs}.
\begin{proposition}
 \label{tp2}
Suppose we have measurable functions $h_1$ and $h_2$ mapping $\Xi \to [0,1]$, and measurable sets $U_1,U_2 \in \A^m$. Define $\P_{h_i}$ to be the finite non-negative measure on $\Xi$ given by $\mathrm{d}\tP_{h_i}(\xi) = h_i(\xi)\mathrm{d}\tP(\xi) , i=1,2$. Suppose there is a positive number $c(j_0)$ such that   the following inequality holds for all $\tv \in \mathcal{B}$: 
\begin{equation} \label{h120}  \tP_{h_1} \bigg[  (X_{\inn}^{n_k} \in U_1) \cap  ( X_{\out}^{n_k} \in \tv) \cap \Omega_{n_k}(j_0) \bigg] \le c(j_0) \cdot \tP_{h_2} \bigg[  (X_{\inn}^{n_k} \in U_2) \cap  ( X_{\out}^{n_k} \in \tv) \cap \Omega_{n_k}(j_0) \bigg] + \vartheta(k) \end{equation} 
where $\vartheta(k)=\vartheta(k;j_0,U_1,U_2,\tv) \to 0$ as $k\to \infty$  for  fixed $j_0,U_1,U_2, \tv$ .
Then for all Borel sets $V$ in $\S_{\out}$, we have
\begin{equation}
 \label{h123}
\tP_{h_1} \bigg[  (X_{\inn} \in U_1) \cap  ( X_{\out} \in V) \cap \Omega(j_0) \bigg] \le  c(j_0) \cdot  \tP_{h_2} \bigg[  (X_{\inn} \in U_2) \cap  ( X_{\out} \in V) \bigg] .
\end{equation}
\end{proposition}


We are now ready to complete the proof of Theorem \ref{abs}.

\begin{proof}[\textbf{Proof of Theorem \ref{abs}}]
We claim that it suffices to show that for every $j_0$, there are positive real numbers $M(j_0)$ and $m(j_0)$ such that for any $A \in \A^m$ and any Borel set $B$ in $\S_{\out}$
\begin{equation}
\label{abs1}
      \tP \l( (\xin \in A) \cap (X_{\out} \in B)  \cap \Omega(j_0) \r) \le M(j_0) \int_{X_{\out}^{-1}(B) \cap \O^m} \nu(\xout({\xi}),A) \mathrm{d}{\tP}(\xi)
\end{equation}
and
\begin{equation}
 \label{abs2}
 m(j_0)  \int_{X_{\out}^{-1}(B) \cap \O(j_0) } \nu(\xout({\xi}),A) \mathrm{d}{\tP}(\xi) \le  \tP \l( (\xin \in A) \cap (X_{\out} \in B) \r).
\end{equation}

Once we have (\ref{abs1}), we can invoke Proposition \ref{tp1}. Setting $\mu_1(\xi,\cdot)=\rho(\xout(\xi),\cdot),\mu_2(\xi,\cdot)=\nu(\xout(\xi),\cdot)$ and $c(j_0)=M(j_0)$ in (\ref{abs0}) we get that a.s. on the event $\Omega^m$ we have  $\rho(\xout,\cdot) \ll {\nu}(\xout,\cdot)$. 

If we have (\ref{abs2}), we can again appeal to Proposition \ref{tp1}. Setting $\mu_1(\xi,\cdot)=\nu(\xout(\xi),\cdot),\mu_2(\xi,\cdot)=\rho(\xout(\xi),\cdot)$ and $c(j_0)=1/m(j_0)$ in (\ref{abs0}) we get that a.s. on the event $\Omega^m$ we have  $\nu(\xout,\cdot)  \ll \rho(\xout,\cdot)$. 

The last two paragraphs together imply that  $\rho(\xout(\xi),\cdot) \equiv {\nu}(\xout(\xi),\cdot)$ a.s. on the event $\Omega^m$,   as desired.

To establish (\ref{abs1}) and (\ref{abs2}), we begin with a fixed $j_0$, a set $A \in \A^m$ and a Borel set $B$ in $\S_{\out}$.  We will invoke Proposition \ref{tp2} in the following manner. Recall (\ref{abscond}): for all $A \in \mathfrak{A}^m$ and $B \in \mathcal{B}$ we have
\[ 
\tP [  (\xin^{n_k} \in A) \cap  ( X_{\out}^{n_k} \in B) \cap \Omega_{n_k}(j) ] \asymp _j \int_{(X_{\out}^{n_k})^{-1}(B) \cap \Omega_{n_k}(j)  } \nu(\xout({\xi}),A) \mathrm{d}\tP(\xi) + \vartheta(k;j,A,B).   \]                              
where $\lim_{k \to \infty} \vartheta(k;j,A,B) = 0$ for each fixed $j,A$ and $B$ .

Let the multiplicative constants appearing in the $\asymp_j$ relation in (\ref{abscond}) with $j=j_0$ be $m(j_0)$ and $M(j_0)$ respectively, with $m(j_0)\le M(j_0)$ (to be more explicit, in the notations of Remark \ref{equivcond}, $M(j_0)$ is playing the role of $C_{j}$ and $m(j_0)$ is playing the role of $c_j$). In Proposition \ref{tp2} we set $h_1(\xi)=1, h_2(\xi)=\nu(\xout(\xi),A),U_1=A,U_2=\D^m, V=B,c(j_0)=M(j_0)$ to obtain (\ref{abs1}). On the other hand, setting $h_1(\xi)=\nu(\xout(\xi),A), h_2(\xi)=1, U_1=\D^m,U_2=A, V=B, c(j_0)=1/m(j_0)$ in Proposition \ref{tp2} we obtain (\ref{abs2}).

This completes the proof of theorem \ref{abs}.
\end{proof}

In the above proof, we have used the fact that the random measures $\rho(\xout,\cdot)$ and $ {\nu}(\xout,\cdot)$ are $\xout$-measurable, which is true because of the measurability proporties of $\nu$ as laid down in the assumptions of Theorem \ref{abs} and the measurability properties of regular conditional distributions on Polish spaces (pertaining to $\rho$).

We end this section with the proofs of Propositions \ref{tp1} and \ref{tp2} which are used in the proof of Theorem \ref{abs} above.

\begin{proof}[{Proof of Proposition \ref{tp1}}]
For each $A \in \A^m$ a.s. in $\xout$ it follows from (\ref{abs0}) that  \begin{equation} \label{abs25}  \E\l[ 1\hskip-4pt{\rm 1}_{\O(j_0) }  \mu_1(\xi,A)  | \xout \r]  \le c(j_0) \E[  1\hskip-4pt{\rm 1}_{\O^m } \mu_2(\xi,A)|\xout] .\end{equation}  On the event $\O^m$ (which is measurable with respect to $\xout$) the last inequality becomes
\begin{equation} \label{abs3} \E\l[ 1\hskip-4pt{\rm 1}_{\O(j_0) }  \mu_1(\xi,A)  | \xout \r]  \le c(j_0) \E[ \mu_2(\xi,A)|\xout] .\end{equation}
To see how (\ref{abs25}) follows from (\ref{abs0}), we note that (\ref{abs0}) implies \[  \int_{X_{\out}^{-1}(B) }  1\hskip-4pt{\rm 1}_{\O(j_0) }  \mu_1({\xi},A) \mathrm{d}{\tP}(\xi) \le \int_{X_{\out}^{-1}(B) } c(j_0) 1\hskip-4pt{\rm 1}_{\O^m } \mu_2({\xi},A) \mathrm{d}{\tP}(\xi) .\]
Since this is true of every Borel set $B \subset \S_{\out}$, therefore by the definition of conditional expectation, we get 
\[\E\l[   1\hskip-4pt{\rm 1}_{\O(j_0) }  \mu_1({\xi},A) | \xout \r] \le  \E[ c(j_0) 1\hskip-4pt{\rm 1}_{\O^m } \mu_2(\xi,A)|\xout]. \]

For almost every configuration $\o \in S_{\out}$ such that $N(\o)=m$ (with respect to the measure $\P_{\xout}$, and $N$ being the function as in Definition \ref{nopoints}),  or equivalently, on the event $\O^m$ which is measurable with respect to $\xout$,  we have (\ref{abs3}) for all $A \in \A^m$. Therefore by the regularity of the Borel measure on the right hand side, (\ref{abs3}) extends to all Borel sets $A \subset \D^m$ (see Proposition \ref{meas1}).  Now, for $\o \in \S_{\out}$ such that $N(\o)=m$, suppose that  $A \subset \D^m$ is a Borel set such that $\E\l[\mu_2(\xi,A)|\xout=\o \r]=0$. 
Then (\ref{abs3}) implies that 
\begin{equation} \label{adhoc2}\E\l[ 1\hskip-4pt{\rm 1}_{\O(j_0) }  \mu_1(\xi,A)  | \xout=\o \r]  =0,\end{equation} for each $j_0$.  But $\O(j_0)$ exhausts $\O^m$.
Hence, for  $\o$ such that $\xout=\o$ implies $\O^m$ occurs (equivalently, $N(\o)=m$), using the Dominated Convergence Theorem we have \[\E\l[1\hskip-4pt{\rm 1}_{\O(j_0) }  \mu_1(\xi,A)|\xout=\o\r]  \to \E\l[1\hskip-4pt{\rm 1}_{\O^m } \mu_1(\xi,A)|\xout=\o \r] = \E\l[ \mu_1(\xi,A)|\xout=\o \r] \] as $j_0 \uparrow \infty$. By letting $j_0 \uparrow \infty$ in (\ref{adhoc2}), we obtain the fact that a.s. on $\O^m$ we have  $\E\l[\mu_2(\xi,A)|\xout \r]=0$ implies $\E\l[ \mu_1(\xi,A) | \xout \r]=0$. In other words, a.s. on $\O^m$, we have (\ref{interm}).
\end{proof}

\begin{proof}[{Proof of Proposition \ref{tp2}}]

In what follows, we will denote by $\tP_h$ the non-negative finite measure on $\Xi$ obtained by setting $\mathrm{d}\tP_h(\xi)=h(\xi)\mathrm{d}\tP(\xi)$ where $h:\Xi \to [0,1]$ is a measurable function. We will proceed in stages as follows: we will first perform certain computations that hold true for a general $\P_h$. Subsequently, we use the result of these computations successively for $h=h_1$ and $h=h_2$, and then use (\ref{h120}) to connect the resulting quantities via an inequality.

We note that for any event $E$, we have $0 \le \tP_h(E) \le \tP(E)$.  Fix a $U \in \A^m$.  

For any Borel set $V$ in $\S_{\out}$ and $\eps>0$ we can find a  $\b \in \mathcal{B}$  such that  \[ \tP \l(  X_{\out}^{-1}(V) \Delta  X_{\out}^{-1} (\b)   \r) < \eps.\] This can be seen by considering the fact that the push forward probability measure $(X_{\out})_*\tP$ is a regular Borel measure on $\S_{\out}$, and $\mathcal{B}$ consists of finite unions of sets that form a topological basis for $\S_{\out}$. The purpose of this reduction is to exploit the fact that as $k\to \infty$, we have ${1\hskip-4pt{\rm 1}}_{\b}\l({X_{\out}^{n_k}}\r) \rightarrow {1\hskip-4pt{\rm 1}}_{\b}\l({X_{\out}}\r)$ a.s.

Let us consider the quantity $\tP_h [  (X_{\inn}^{n_k} \in U) \cap  ( X_{\out}^{n_k} \in \b) \cap \Omega_{n_k}(j_0) ]$. This is equal to \[ \tP_h \bigg[  (X_{\inn} \in U) \cap  ( X_{\out} \in \b) \cap \Omega_{n_k}(j_0) \bigg] + o_k(1;U,\b), \] where $o_k(1;U,\b)$ stands for a quantity that tends to $0$ as $k \to \infty$ for fixed sets $U$ and $\b$. This step uses the fact that $1\hskip-4pt{\rm 1}\l[X^{n_k}_{\out} \in \b\r] \to 1\hskip-4pt{\rm 1}\l[X_{\out} \in \b\r]$  and $1\hskip-4pt{\rm 1}\l[X^{n_k}_{\inn} \in U\r] \to 1\hskip-4pt{\rm 1}\l[X_{\inn} \in U\r]$ a.s. The expression in the last display above equals
\begin{equation} \label{tp2eq}  \tP_h \bigg[  (X_{\inn} \in U) \cap   ( X_{\out} \in V)  \cap \Omega_{n_k}(j_0) \bigg]  + o_{\eps}(1) + o_k(1;U,\b), \end{equation}  where $o_{\eps}(1)$ denotes a quantity that tends to $0$ uniformly in $k$ as $\eps \to 0$.

We can bound the probability in (\ref{tp2eq}) simply by $\tP_h \bigg[  (X_{\inn} \in U) \cap   ( X_{\out} \in V)  \bigg]$. Putting all these together, we have 
\begin{equation}
 \label{interm1}
\tP_h [  (X_{\inn}^{n_k} \in U) \cap  ( X_{\out}^{n_k} \in \b) \cap \Omega_{n_k}(j_0) ] \le  \tP_h \bigg[  (X_{\inn} \in U) \cap   ( X_{\out} \in V)  \bigg]    + o_{\eps}(1) + o_k(1;U,\b).
\end{equation}

To obtain a comparable lower bound requires some more work. We begin with \[\tP_h \bigg[  (X_{\inn} \in U) \cap   ( X_{\out} \in V)  \cap \Omega_{n_k}(j_0) \bigg] \ge \tP_h \bigg[  (X_{\inn} \in U) \cap   ( X_{\out} \in V)  \cap  \Omega_{n_k}(j_0) \cap \O(j_0) \bigg]. \]
Observe that  $ \Omega({j_0}) = \varliminf_{k \to \infty} \Omega_{n_k}(j_0)$ and $\tP \l( \big(\varliminf_{l \to \infty} \Omega_{n_l}(j_0)\big) \Delta \l( \bigcap_{l \ge k} \Omega_{n_l}(j_0)  \r)  \r) = \op_k(1)$ where $\op_k(1)$ denotes a quantity which, for a fixed $j_0$, converges to $0$ as $k \to \infty$,   uniformly in all the other quantities.

Hence \[ \tP_h \bigg[(X_{\inn} \in U) \cap ( X_{\out} \in V) \cap \Omega_{n_k}(j_0)  \cap  \Omega(j_0)   \bigg] \]\[ = \tP_h \bigg[  (X_{\inn} \in U) \cap  ( X_{\out} \in V) \cap \Omega_{n_k}(j_0)  \cap \l( \bigcap_{l \ge k} \Omega_{n_l}(j_0)  \r) \bigg]  + \op_k(1)  \hspace{3 pt}.  \]
But $ \Omega_{n_k}(j_0) \cap  \l( \bigcap_{l \ge k} \Omega_{n_l}(j_0) \r)   = \l( \bigcap_{l \ge k} \Omega_{n_l}(j_0) \r) $ and hence the probability on the right hand side of the last display equals
\[  \tP_h \bigg[  (X_{\inn} \in U) \cap  ( X_{\out} \in V) \cap \Omega(j_0)  \bigg] + \op_k(1).  \]
The arguments above result in
\begin{equation}
 \label{interm2}
\begin{aligned}
 \tP_h \bigg[  (X_{\inn}^{n_k} \in U) \cap  ( X_{\out}^{n_k} \in \b) \cap \Omega_{n_k}(j_0) \bigg] \ge &\tP_h \bigg[  (X_{\inn} \in U) \cap  ( X_{\out} \in V) \cap \Omega(j_0) \bigg] \\  &+   o_k(1;U,\b)  + o_{\eps}(1) + \op_k(1).
\end{aligned}
\end{equation}

Now, we wish to prove (\ref{h123}). We appeal to (\ref{interm1}) with $h=h_2,U=U_2$ and to (\ref{interm2}) with $h=h_1,U=U_1$ to obtain, using (\ref{h120}) (applied with $\tv=\b$),  
\begin{equation}
\label{lims} 
\begin{aligned}
\tP_{h_1} \bigg[  (X_{\inn} \in U_1) \cap  ( X_{\out} \in V) \cap \Omega(j_0) \bigg]    \le &c(j_0) \tP_{h_2} \bigg[  (X_{\inn} \in U_2) \cap   ( X_{\out} \in V)  \bigg] \\   &+ o_{\eps}(1) + o_k(1;U,\b) + \op_k(1) + \vartheta(k).
\end{aligned}
\end{equation}
We first keep all the other quantities fixed and let $k \to \infty$, after that we let $\eps \to 0$ to obtain (\ref{h123}),  as desired. 
\end{proof}

\textbf{Acknowledgements.} We are very grateful to Fedor Nazarov and Mikhail Sodin for suggesting the approach to the tolerance theorems, and for many fruitful discussions. We thank Ron Peled for helpful comments. We are also grateful to the anonymous referees for their meticulous reading of the paper, for pointing out relevant references and a gap in the  proof of Theorem \ref{gin-1}, along with numerous thoughtful suggestions for the improvement of the manuscript.

\noindent \textbf{Email:} \newline
S. Ghosh \hspace{15 pt} \textsl{subhrowork@gmail.com} \hspace{115 pt}
Y. Peres \hspace{17 pt} \textsl{peres@microsoft.com}

\end{document}